\documentclass[reqno, 11pt]{article}

\usepackage{fullpage}
\usepackage{amsmath} 
\usepackage{amssymb} 
\usepackage{amsthm} 
\usepackage{amsfonts}
\usepackage{dsfont} 
\usepackage{color}
\usepackage{mathtools}

\newtheorem{thm}{Theorem}
\newtheorem{cor}[thm]{Corollary}
\newtheorem{lem}[thm]{Lemma}
\newtheorem{prop}[thm]{Proposition}
\newtheorem{cond}[thm]{Condition}
\theoremstyle{definition}
\newtheorem{defn}[thm]{Definition}
\newtheorem{assum}[thm]{Assumption}

\theoremstyle{remark}
\newtheorem{rmk}[thm]{Remark}
\newtheorem{nrmk}[thm]{Notational Remark}

\numberwithin{equation}{section} 
\numberwithin{thm}{section}

\newcommand{\hidx}{a}
\newcommand{\widx}{\alpha}
\newcommand{\Noise}{X}
\newcommand{\noise}{X_{\hidx\widx}}
\newcommand{\Noisev}{{\boldsymbol \Noise}}
\newcommand{\Data}{Y}
\newcommand{\data}{Y_{\hidx\widx}}

\newcommand{\SNR}{\lambda}
\newcommand{\snr}{w}

\newcommand{\U}{{\boldsymbol u}}
\newcommand{\V}{{\boldsymbol v}}

\newcommand{\bs}{{\boldsymbol \sigma}}
\newcommand{\bt}{{\boldsymbol \tau}}

\newcommand{\ratio}{d_0}

\title{Free energy and fluctuations of generalized bipartite Sherrington--Kirkpatrick models and generalized Hopfield models}

\author{Hyunsuk Choo\footnote{Department of Mathematical Sciences, KAIST, Daejeon, 34141, Korea \newline email: \texttt{hyunsuk.choo@kaist.ac.kr}} \,
and Ji Oon Lee\footnote{Department of Mathematical Sciences, KAIST, Daejeon, 34141, Korea \newline email: \texttt{jioon.lee@kaist.edu}}}
 
\begin{document}

\maketitle

\abstract{We consider the free energy of generalized bipartite Sherrington--Kirkpatrick models in the high temperature regime, with one spin normalized and the other either normalized or i.i.d. We derive a unified formula for the limiting free energy in both settings. We also prove that the fluctuations of the free energy converge to a Gaussian distribution under mild assumptions. The proof is based on a bipartite multigraph expansion, which also gives free energy results for generalized Hopfield models and log likelihood ratio of spiked rectangular models.} 

\section{Introduction}\label{sec:introduction}

The Sherrington--Kirkpatrick (SK) model \cite{Sherrington-Kirkpatrick1975} describes a system of Ising spins with independent random pairwise interactions and is a central model in the theory of spin glasses.
One of the main questions is to understand the asymptotic behavior of its free energy, including the deterministic limit and the fluctuations around it.
The limiting free energy was predicted by Parisi \cite{Parisi1979, Parisi1980} and rigorously proved by Talagrand \cite{Talagrand2006}.
In the high temperature regime, Aizenman, Lebowitz, and Ruelle \cite{Aizenman-Lebowitz-Ruelle1987} proved that the rescaled fluctuations of the free energy converge to a Gaussian distribution.
A closely related mean-field model of spin glass is the Hopfield model \cite{Hopfield1982}, whose couplings form a Gram matrix constructed from stored patterns.
It was introduced as a model of associative memory, in which a stored pattern is retrieved from a partial or corrupted input, and has motivated subsequent developments such as Dense Associative Memory \cite{Krotov-Hopfield2016} and Modern Hopfield networks \cite{Modern-Hopfield}.

In this paper, we consider the generalized Hopfield partition function
\begin{equation}\label{eq:intro-hopfield}
	\hat{Z}_{M,N}(\beta) =\mathbb{E}_{\bt}\left[\exp\left(\sum_{\hidx=1}^{M} \psi\left(\beta \sum_{\widx=1}^{N}J_{\hidx\widx} \frac{\tau_\widx}{\|\bt\|} \right)\right)\right], \qquad \psi(t)=\log\mathbb{E}_{\sigma}[e^{t\sigma}],
\end{equation}
where the rows of the $M\times N$ matrix $J$ are the patterns, $\bt$ is the visible spin vector, and $\sigma$ is a hidden spin with a symmetric, unit-variance distribution.
The patterns have i.i.d., centered, unit-variance entries, and $\beta>0$ is the inverse temperature parameter.
When the hidden spin is a standard Gaussian, $\psi(t)=t^2/2$, and the partition function in \eqref{eq:intro-hopfield} reduces to
\begin{equation}\label{eq:def-hopfield}
	Z^H_{M,N}(\beta) =\mathbb{E}_{\bt}\left[\exp\left( \frac{\beta^2}{2\| \bt \|^2} \bt^T J^T J\bt \right)\right],
\end{equation}
which corresponds to the classical Hopfield model if the patterns and visible spins are Rademacher random variables, with inverse temperature $\beta^2/2$ and with diagonal terms retained.
The generalized form in \eqref{eq:intro-hopfield} arises by integrating out the hidden layer of a restricted Boltzmann machine (RBM), a bipartite model used for learning probability distributions, and it is used to study how the priors affect the phase diagram and retrieval properties \cite{Barra-Genovese-Sollich-Tantari2017,Barra-Genovese-Sollich-Tantari2018}.

The properties of the generalized Hopfield model, especially its free energy defined by
\begin{equation}\label{eq:def_free_energy}
	\hat{F}_{M,N}(\beta):=\frac{1}{N}\log\hat{Z}_{M,N}(\beta),
\end{equation}
have been extensively studied. 
The limiting free energy of the classical Hopfield model in the high temperature regime was first proved by Scacciatelli and Tirozzi \cite{scacciatelli1992fluctuation} and Talagrand \cite{Talagrand1998}.
The fluctuation of the free energy for the classical Hopfield model was also studied in \cite{scacciatelli1992fluctuation} and proved by Comets, Kurkova, and Trashorras \cite{Comets-Kurkova-Trashorras2004}.
Analogous results were proved for the Hopfield model with Gaussian patterns and Ising visible spins \cite{Barra-Guerra2008,Barra-Genovese-Guerra-Tantari2012} and general patterns and spherical visible spins \cite{auffinger2014free,Baik-Lee2016}. 
In the physics literature \cite{Barra-Genovese-Sollich-Tantari2017,Barra-Genovese-Sollich-Tantari2018}, the generalized Hopfield model was considered for patterns and spins that are the mixtures of the Gaussian and Rademacher random variables. 
However, to our best knowledge, all previously known rigorous results require certain specific properties of the patterns and/or spins. 

\subsection{Main contributions}
In this paper, we prove the Gaussian fluctuations for several quantities including the free energy of the generalized Hopfield models in the high temperature regime under mild assumptions on the patterns and spins. 
More precisely, in our first main result (Theorem \ref{thm:Hopfield}), we prove that if $M/N\to\ratio\in(0,\infty)$, then
\begin{equation}\label{eq:intro-hopfield-limit}
    \hat{F}_{M,N}(\beta)\rightarrow \ratio\tilde{\Lambda}(\beta^2/2), \qquad \tilde{\Lambda}(t):=\log\mathbb{E}_{\sigma}[e^{t\sigma^2}]
\end{equation}
in probability, and $\log\hat{Z}_{M,N}(\beta)-M\tilde{\Lambda}(\beta^2/2)$ converges in distribution to a Gaussian with explicit formulas for the mean and the variance of the limiting Gaussian.

Our proof of Theorem is based on the bipartite representation of \eqref{eq:intro-hopfield}, with Hamiltonian $\sum_{\hidx,\widx}J_{\hidx\widx}\sigma_\hidx\tau_\widx / \|\bt\|$, where the hidden spins $\sigma_\hidx$ are i.i.d. copies of $\sigma$. 
With the bipartite representation, the generalized Hopfield model can be considered as a special case of the generalized bipartite SK model. 
In our second main result (Theorems \ref{thm:bsk-iid} and \ref{thm:bsk-normalized}), we prove that our result on the generalized Hopfield model naturally extends to the generalized bipartite SK model.

To analyze the generalized bipartite model, we apply the multigraph expansion framework introduced in \cite{Choo-Han-Lee2026} with the following adaption: The two spin species contribute separate vertex weights, while the bipartite structure permits only even cycles and excludes self-loops.
We establish the corresponding multigraph comparison and cycle estimates, keeping track of both priors and the aspect ratio.
An additional step is needed for the i.i.d. hidden layer: the expansion produces the factor $\exp(\beta^2\|\bs\|^2/2)$, which we absorb into the hidden prior by a quadratic change of measure.
Its normalizing constant is $\exp(M\tilde{\Lambda}(\beta^2/2))$, giving the leading term in \eqref{eq:intro-hopfield-limit}.
We then control the moments and overlaps under the tilted law and apply the bipartite comparison to the remaining fluctuation.
The same framework also treats a normalized hidden layer, for which $\|\bs\|^2=M$ and the leading term reduces to $M\beta^2/2$.
Thus the squared norm of the hidden layer provides a common mechanism for the free energy in the normalized and i.i.d. settings.

As in \cite{Choo-Han-Lee2026}, the analysis of the bipartite multigraph expansion applies to even more general problems. 
One notable such example is the study of the log likelihood ratio in the rank-$1$ spiked rectangular model $\Data=\Noise+\sqrt{\SNR}\U\V^T$, where $\Noise$ has independent entries of variance $N^{-1}$ and the signal vectors $\U$ and $\V$ are drawn independently from given priors.
The detection problem is to distinguish a fixed positive signal strength from the unspiked model.
By the Neyman--Pearson lemma \cite{Neyman-Pearson1933}, the likelihood ratio determines the optimal tradeoff between Type-I and Type-II errors.
Extending the results for Gaussian noise \cite{Alaoui-Jordan2018, Banerjee-Ma2022}, in our third main result (Theorem \ref{thm:log-LR}), we prove a Gaussian limit for the log likelihood ratio with general noise satisfying suitable density regularity assumptions.
Here, no quadratic change of the prior is needed, so either signal vector may independently have a normalized or an i.i.d. prior.
The result covers symmetric sub-Gaussian priors in a subcritical range determined by an overlap second moment condition, including i.i.d. priors on both sides, and yields the limiting optimal testing error.

\subsection{Main results}

Our first main result is the Gaussian fluctuation for the free energy of the generalized Hopfield models in the high temperature regime. 
Since the general statement is rather involved, we first state the simplest version for which we will use the following definitions; the general results for large class of priors and the common universality proposition are given in Section \ref{sec:general-results}.
\begin{defn}[Strictly sub-Gaussian random variable]
A real-valued random variable $T$ is strictly sub-Gaussian if $\mathbb{E}[e^{\alpha T}]\le e^{\alpha^2/2}$ for any (possibly $N$-dependent) constant $\alpha$.
\end{defn} 

\begin{defn}[Lee--Yang type random variable] \label{defn:Lee-Yang}
A real-valued random variable $T$ is of Lee--Yang type if it satisfies following three properties:
\begin{itemize}
	\item $T$ is symmetric about $0$.
	\item $\mathbb{E}[\exp(cT^2)]<\infty$ for some $c>0$.
	\item For $z\in \mathbb{C}$, $\mathbb{E}[\exp(zT)]$ only has zeros on the pure imaginary axis.
\end{itemize}
\end{defn} 
(See also Remark \ref{rem:Lee-Yang}.)

Now, we state our first main result.
\begin{thm} \label{thm:Hopfield}
	Recall the definition of the free energy  in \eqref{eq:intro-hopfield} and \eqref{eq:def_free_energy}. 
	Suppose that $J_{\hidx\widx}(1\le \hidx\le M, 1\le \widx\le N)$ are i.i.d., symmetric, and strictly sub-Gaussian random variables with variance $1$.
	Suppose that $\sigma_\hidx(1\le \hidx\le M)$ are i.i.d., Lee--Yang type random variables with variance $1$, independent of $J_{\hidx\widx}$.
	Suppose that $\tau_{\widx}{(1\le \widx\le N)}$ are i.i.d., symmetric, and strictly sub-Gaussian random variables with variance $1$, independent of $J_{ \hidx\widx}$ and $\sigma_\hidx$. 
	Recall the definition of $\tilde{\Lambda}(t)$ in \eqref{eq:intro-hopfield-limit}. 
Assume that in the small neighborhood of $t=\beta^2 /2$, $\tilde{\Lambda}, \tilde{\Lambda}', \tilde{\Lambda}''<\infty$.
	Define
	\begin{equation*}
		\tilde{\sigma}^2 = \tilde{\Lambda}'\left( \frac{\beta^2}{2} \right),\quad \tilde{m}_4=\tilde{\Lambda}''\left( \frac{\beta^2}{2} \right)\tilde{\sigma}^{-4} +1
	\end{equation*}
	Then, for any $\beta\tilde{\sigma}\in (0, \ratio^{-1 /4})$,
	\[ { \hat{F}}_{M,N}(\beta) \to \tilde{F}(\beta):= \ratio\tilde{\Lambda}\left( \frac{\beta^2}{2} \right).\] and
	\[ N\left({ \hat{F}}_{M,N}(\beta)-\frac{M}{N}\tilde{\Lambda}\left(\frac{\beta^2}{2}\right)\right) \implies \mathcal{N}(\tilde{m}_F, \tilde{V}_F),\]
	as $M,N\to \infty$, i.e., $N{ \hat{F}}_{M,N}(\beta)-M\tilde{\Lambda}(\beta^2 /2)$ converges in distribution to the Gaussian with mean $\tilde{m}_F$ and the variance $\tilde{V}_F$ as $M,N\to \infty$.
	Here,
	\[  \tilde{m}_F =\frac{1}{4}\log\left( 1-\ratio\beta^4\tilde{\sigma}^4 \right)+\frac{\ratio\beta^4\tilde{\sigma}^4}{24}(x_4-3)(\tilde{m}_4n_4-3) \] 
	and 
	\[ \tilde{V}_F= -\frac{1}{2} \log( 1-\ratio\beta^4\tilde{\sigma}^4) +\frac{\ratio\beta^4\tilde{\sigma}^4}{4}(x_4-3) ,\]
	where we let $n_4=\mathbb{E}[\tau_1^4]$ and $x_4=\mathbb{E}[J_{11}^4]$.	
\end{thm}

Note that Theorem \ref{thm:Hopfield} covers the Gaussian-Rademacher interpolation model studied in \cite{Barra-Genovese-Sollich-Tantari2017, Barra-Genovese-Sollich-Tantari2018}, since the Lee--Yang property is preserved under independent sums.
Note also that Theorem \ref{thm:Hopfield} (and subsequent theorems in this section) covers the full high temperature regime; see Remark \ref{rem:criticality}.

Applying Theorem \ref{thm:Hopfield} with Gaussian hidden prior, using cumulant generating function of Gaussian random variable $\psi(t) = t^2 /2$, we get the following result for the Hopfield model.
\begin{cor}\label{cor:Hopfield-Gaussian}
	Recall the definition of the partition function of the Hopfield model in \eqref{eq:def-hopfield}. Define its free energy
	\[ F^H_{M,N} \equiv F^H_{M,N}(\beta) = \frac{1}{N}\log Z^H_{M,N}. \]
	Suppose that the assumptions in Theorem \ref{thm:Hopfield} hold for $J$ and $\bt$. Then, for any $\beta\in (0,{(1+\sqrt{\ratio})^{-1 /2}})$, 
	\[ F^H_{M,N}(\beta)\to  F^H(\beta):=-\frac{\ratio}{2}\log(1-{\beta^2}),\]
	and
	\[ N\left( F^H_{M,N}(\beta)-\left( -\frac{M}{2N}\log(1-{\beta^2}) \right) \right)\implies \mathcal{N}(m_F^H, V_F^H),\]
	as $M,N\to \infty$.
	Here,
	\[ m_F^H=\frac{1}{4}\log\left( 1-\frac{\ratio{\beta^4}}{(1-{\beta^2})^2} \right)+\frac{\ratio{\beta^4}}{8(1-\beta^2)^2}(x_4-3)(n_4-1)\]
	and
	\[V_F^H=-\frac{1}{2}\log\left( 1-\frac{\ratio{\beta^4}}{(1-{\beta^2})^2} \right)+\frac{\ratio{\beta^4}}{4(1-\beta^2)^2}(x_4-3),\]
	where we let $n_4=\mathbb{E}[\tau_1^4]$ and $x_4=\mathbb{E}[J_{11}^4]$.
\end{cor}

Our proof of Theorem \ref{thm:Hopfield} is based on the equivalence between the bipartite SK model with an i.i.d. prior and the generalized Hopfield model. With the normalization 
\begin{equation} \label{eq:renormalization}
	J_{\hidx\widx} = \sqrt{N}\noise  (1\le \hidx\le M, 1\le \widx\le N), \quad \sigma_\hidx = \sqrt{M}u_\hidx, \quad \tau_\widx = \sqrt{N}v_\widx, 
\end{equation}
the partition function and the free energy of a bipartite SK model is defined by
\begin{equation} \label{eq:def_free_bipartite_iid}
	\tilde{Z}_{M,N}(\beta)=\mathbb{E}_{\mathcal{U}\times \mathcal{V}}\left[\exp\left( \beta \sqrt{MN}\sum_{\hidx,\widx}\noise u_\hidx \frac{v_\widx}{\|\V\|} \right) \right],\quad \tilde{F}_{M,N}(\beta)=\frac{1}{N}\log \tilde{Z}_{M,N}.
\end{equation}
(When $\|\V\|=0$, we define $v_\widx /\|\V\|=0$ for all $\widx$.)
Note that the prior is i.i.d. on one side and normalized on the other side. Then, integrating on i.i.d. spin,
\[ \tilde{Z}_{M,N}(\beta)=\mathbb{E}_{\mathcal{V}}\left[\exp\left(\sum_{\hidx}\psi_{u}\left(\sum_{\widx}\beta\sqrt{N}\noise\frac{v_\widx}{\|\V\|} \right) \right) \right]=\mathbb{E}_{\mathcal{V}}\left[\exp\left(\sum_{\hidx}\psi_u\left(\beta\sqrt{N}\frac{\Noisev_\hidx\V}{\|\V\|}\right) \right) \right] \]
where $\psi_u(t) := \log\mathbb{E}[\exp(t\sqrt{M}u_\hidx)]$ is the cumulant generating function of $\sqrt{M}u_\hidx$ and $\Noisev_\hidx$ is the $\hidx$-th row vector of $\Noise$. 
In this representation, we refer to $\U$ and $\V /\|\V\|$ as the hidden and visible prior, respectively. Our second main result is the following theorem that identifies the limiting free energy and gives the corresponding fluctuations.

\begin{thm}\label{thm:bsk-iid}
	Suppose that the assumptions in Theorem \ref{thm:Hopfield} hold for $\sqrt{N}\noise$, $\sqrt{M}u_\hidx(1\le \hidx\le M)$, and $\sqrt{N}v_\widx(1\le \widx\le N)$ with the normalization in \eqref{eq:renormalization}. Recall the definition of $\tilde{F}_{M,N}(\beta)$ in \eqref{eq:def_free_bipartite_iid}. Then, the conclusion in Theorem \ref{thm:Hopfield} holds without any change if ${ \hat{F}}_{M,N}(\beta)$ is replaced by $\tilde{F}_{M,N}(\beta)$.	
\end{thm}

We can also consider the case where the hidden prior is also normalized. The precise result for this case is the following theorem:

\begin{thm}\label{thm:bsk-normalized}
	Suppose that $\sqrt{N}\noise(1\le \hidx\le M, 1\le \widx\le N)$ are i.i.d., symmetric, and strictly sub-Gaussian random variables with variance $1$.
	Suppose that $\sqrt{M}u_\hidx(1\le \hidx\le M)$ are i.i.d., symmetric, and strictly sub-Gaussian random variables with variance $1$, independent of $\noise$.
	Suppose that $\sqrt{N}v_\widx(1\le \widx\le N)$ are i.i.d., symmetric, and strictly sub-Gaussian random variables with variance $1$, independent of $\noise$ and $u_\hidx$.
	Let
	\begin{equation}\label{eq:bsk-normalized-partition-function-free-energy}
		\overline{Z}_{M,N}(\beta)=\mathbb{E}_{\mathcal{U}\times \mathcal{V}}\left[\exp\left( \beta \sqrt{MN}\sum_{\hidx,\widx}\noise\frac{u_\hidx v_\widx}{\|\U\|\|\V\|} \right) \right],\quad \overline{F}_{M,N}(\beta)=\frac{1}{N}\log \overline{Z}_{M,N}.
	\end{equation}
	(When $\|\U\|$ or $\|\V\|=0$, we define $u_\hidx v_\widx /(\|\U\|\|\V\|)=0$ for all $\hidx,\widx$.)
	Then, for any $\beta\in(0,\ratio^{-1 / 4})$,
	\[ \overline{F}_{M,N}(\beta) \to \overline{F}(\beta):= \frac{\ratio\beta^2}{2}\] 
	and
	\[ N\left(\overline{F}_{M,N}(\beta)-\frac{\beta^2M}{2N}\right) \implies \mathcal{N}(\overline{m}_F, \overline{V}_F),\]
	as $M,N\to \infty$, i.e., $N(\overline{F}_{M,N}(\beta)-\beta^2M/(2N))$ converges in distribution to the Gaussian with mean $\overline{m}_F$ and the variance $\overline{V}_F$ as $M,N\to \infty$.
	Here,
	\[  \overline{m}_F =\frac{1}{4}\log( 1-\ratio\beta^4 )+\frac{\ratio\beta^4}{24}(x_4-3)(m_4n_4-3) \] and \[ \overline{V}_F= -\frac{1}{2}\log( 1-\ratio\beta^4)+\frac{\ratio\beta^4}{4}(x_4-3) ,\]
	where we let $m_4=M^2\mathbb{E}[u_1^4], n_4=N^2\mathbb{E}[v_1^4]$ and $x_4=N^2\mathbb{E}[\noise^4]$.	
\end{thm}

\begin{rmk} \label{rem:Lambda}
In Theorems \ref{thm:Hopfield}, \ref{thm:bsk-iid}, and \ref{thm:bsk-normalized}, the limiting free energy is $\ratio \Lambda_M(\beta^2/2)$, where
\[ \Lambda_M(t) := \frac{1}{M}\log\left( \mathbb{E}\left[e^{tM \|\U\|^2}\right] \right) .\]
When $\U$ is also normalized, then $\Lambda_M(t) =t$, which leads to the universal limiting free energy $\ratio \beta^2 / 2$.
When the entries of $\U$ are i.i.d., we have
\begin{equation}\label{eq:Lambda-iid}
	\Lambda_M(t) = \frac{1}{M}\log\left( \mathbb{E}\left[e^{tMu_\hidx^2}\right] \right)^M=\tilde{\Lambda}(t).
\end{equation}
\end{rmk}

\begin{rmk} \label{rem:Lee-Yang}
Our definition for the Lee--Yang type random variable in Definition \ref{defn:Lee-Yang} is from \cite{Newman-Wu2019a}.
The Lee--Yang class is a classical and important class of distributions in ferromagnetic statistical mechanics.
Although the Lee--Yang condition here does not play a role as in ferromagnetic statistical mechanics, it provides a broad, explicit class of priors---including Rademacher, Gaussian and $\phi^4$ type spins---for which the change of measure arising in the proof preserves sharp sub-Gaussianity; see Remark \ref{rmk:hopfield-bgst} and \ref{rmk:phi-4}.
\end{rmk}

\begin{rmk} \label{rem:criticality}
	Note that for Theorems \ref{thm:bsk-iid} and \ref{thm:bsk-normalized}, we recovered the results for the full high temperature regime since the mean and the variance in Theorem \ref{thm:bsk-iid} (resp. Theorem \ref{thm:bsk-normalized}) diverge as $\beta\tilde{\sigma}\nearrow\ratio^{-1 /4}$ (resp. $\beta\nearrow\ratio^{-1 /4}$).
\end{rmk}

Finally, we state the result on the log likelihood ratio.
We will assume the existence of the densities of the disorder and the following regularity conditions for them.
\begin{assum}\label{assum:density}
	Suppose that $\sqrt{N}\noise(1\le \hidx \le M, 1\le \widx\le N)$ are i.i.d. with density function $p$, and have variance $1$.
	We assume the following:
	\begin{enumerate}
		\item $p$ is smooth, positive everywhere, and symmetric (about $0$).
		\item $p$ and their all derivatives vanish at infinity.
		\item $p^{(s)}/p$ are polynomially bounded, i.e., for any positive integer $s$ there exist constants $C_s, r_s >0$, independent of $N$, such that $|p^{(s)}(x)/p(x)|\le C_s(1+|x|)^{r_s}$ uniformly on $x$. (Here, $p^{(s)}$ is the $s$-th derivative of $p$.)
	\end{enumerate}
\end{assum}

Our result for the fluctuation of the log likelihood ratio is the following theorem. 
\begin{thm}\label{thm:log-LR}
	Suppose that Assumption \ref{assum:density} holds for $\Noise$.
	Suppose that $\sqrt{M}u_\hidx(1\le \hidx\le M)$ are i.i.d., symmetric, and strictly sub-Gaussian random variables with variance $1$, independent of $\noise$.
	Suppose that $\sqrt{N}v_\widx(1\le \widx\le N)$ are i.i.d., symmetric, and strictly sub-Gaussian random variables with variance $1$, independent of $\noise$ and $u_\hidx$.
	Define the likelihood ratios by
	\[ \mathcal{L}(\snr, \epsilon_u, \epsilon_v) := \mathbb{E}_{\mathcal{U}\times \mathcal{V}}\left[\prod_{\hidx,\widx}\frac{p(\sqrt{N}\noise-\sqrt{\snr N}u_\hidx v_\widx / (\|\U\|^{\epsilon_u}\|\V\|^{\epsilon_v}))}{p(\sqrt{N}\noise)}\right],\]
	where $\epsilon_u, \epsilon_v \in \{0,1\}$. 
	Define \[ F_p:= \int_{-\infty}^{\infty} \frac{p'(x)^2}{p(x)}\mathrm{d}x, \quad G_p:= \int_{-\infty}^{\infty} \frac{p''(x)^2}{p(x)}\mathrm{d}x.\]
	Then, for any $\snr\in(0,\sqrt{\ratio} / F_p)$,
	\[ \log\mathcal{L}(\snr,\epsilon_u, \epsilon_v) \implies \mathcal{N}(-\rho_{L}, 2\rho_{L}),\]
	as $M,N\to \infty$, i.e., $\log\mathcal{L}(\snr,\epsilon_u, \epsilon_v)$ converges in distribution to the Gaussian with mean $-\rho_{L}$ and the variance $2\rho_{L}$ as $M,N\to \infty$, where
	\[ \rho_{L}:= -\frac{1}{4}\left( \log\left(1-\frac{\snr^2F_p^2}{\ratio}\right)+\frac{\snr^2}{2\ratio}(2F_p^2-G_p) \right) .\]
\end{thm}
Note that the results are the same for the i.i.d. spikes and normalized spikes.

\subsection{Related works}

The free energy fluctuation of the SK model in the high temperature regime was first proved using the cluster expansion by Aizenman, Lebowitz, and Ruelle \cite{Aizenman-Lebowitz-Ruelle1987} and later proved with other methods \cite{frohlich1987some,comets1995sherrington,talagrand2010mean2}.
For more results about the free energy fluctuation for the variants of the SK model, we refer to \cite{Collins-Woodfin-Le2025} and the references therein.

The free energy fluctuation of the classical Hopfield model in the high temperature regime was first proved using the cluster expansion by Scacciatelli and Tirozzi \cite{scacciatelli1992fluctuation}.
Other fluctuation results include those for the classical and Gaussian-pattern Hopfield models \cite{Comets-Kurkova-Trashorras2004,Barra-Guerra2008,Barra-Genovese-Guerra-Tantari2012}, the generalized Hopfield model with symmetric disorder, Rademacher visible prior, and general even overlap potentials studied by Trashorras \cite{Trashorras2006},  and the model with a uniform spherical visible prior and quadratic interaction \cite{Baik-Lee2016}.
Related RBMs with spherical visible and Gaussian hidden units have also been studied through Legendre equivalences of the free energy by Genovese and Tantari \cite{Genovese-Tantari2020}.

The bipartite SK model is a special case of general multi-species SK (MSK) models with Ising spins.
The Parisi formula for the MSK with positive-definite variance profile was first considered by Barra, Contucci, Mingione, and Tantari \cite{Barra-Contucci-Mingione-Tantari2015} and later extended in \cite{panchenko2015,bates2019replica,Dey-Wu2021,Liu2021}. 
The free energy fluctuation was proved by Dey and Wu \cite{Dey-Wu2023}. 
The limiting free energy for the bipartite spherical SK (BSSK) model was obtained by Auffinger and Chen \cite{auffinger2014free} for small enough $\beta$. 
The general result for the BSSK model, including the free energy fluctuation, was first proved in \cite{Baik-Lee2020} and later extended in \cite{collins2025free}.

The spiked rectangular model was first introduced by Johnstone \cite{Johnstone2001}. 
The fluctuations of the log likelihood ratio below the classical BBP transition \cite{BBP2005} were considered by Onatski, Moreira, and Hallin \cite{Onatski2013,Onatski2014} and Johnstone and Onatski \cite{Onatski2015}. 
The optimal threshold for the detection problem was considered in \cite{lesieur2015mmse,Perry-Wein-Bandeira-Moitra2018}. 
The asymptotic normality of the log likelihood ratio for general prior was proved by El Alaoui and Jordan \cite{Alaoui-Jordan2018}. 
The fluctuation of the log likelihood ratio and the linear spectral statistics of the associated Gram matrix were considered in \cite{jung2021,Banerjee-Ma2022,jung2024detection}.

\subsection{Organization of the paper}
The rest of the paper is organized as follows: In Section \ref{sec:general-results}, we state and explain more general statements of our results.
In Section \ref{sec:thm-proof}, we prove the main results from the general results in Section \ref{sec:general-results}, including the quadratic change of measure.
In Sections \ref{sec:multigraph} and \ref{sec:prop-proof}, we develop the bipartite multigraph expansion and prove the universal result, Proposition \ref{prop:general}. Technical proofs and estimates are collected in Section \ref{sec:technical}.

\subsection{Statement on AI use}
We developed the bipartite multigraph expansion method to obtain results on the bipartite SK model with normalized priors and the spiked rectangular matrices, including the proofs of the Proposition \ref{prop:general} through Section \ref{sec:multigraph}-\ref{sec:technical}.
For Lemma \ref{lem:simple-cycle}, GPT-6 Astra generated a proof of the bipartite cluster expansion following the methods of \cite{Aizenman-Lebowitz-Ruelle1987}; a related expansion appears in \cite{Dey-Wu2023} in a different setting.
We checked this proof in detail and rewrote it.

While we were studying the Hopfield model using Lemma \ref{lem:rare-event-cut}, GPT-5.5 Thinking proposed a Gaussian change-of-measure technique.
We recognized that same technique extends to general i.i.d. hidden priors and developed and wrote the proof of Theorem \ref{thm:free-energy}, with input from GPT-5.6 Thinking.

Further, GPT-5.6 Pro identified Lee--Yang priors as a class satisfying the strict sub-Gaussianity condition under quadratic tilting, leading to the Lee--Yang formulation of the Theorem~\ref{thm:Hopfield}. 
Later, the examples of Lee--Yang type, Gaussian-Rademacher interpolation(Remark \ref{rmk:hopfield-bgst}) and $\phi^4$ type(Remark \ref{rmk:phi-4}), was calculated by GPT-6 Astra, which we verified, rewrote and compared with earlier results.

GPT-6 Astra, and Gemini-4 Argon were also used to locate relevant references and check the proofs for possible errors.
We have manually verified all statements and proofs and polished the paper, and we take full reponsibility for its contents.

\section{General results}\label{sec:general-results}

In this section, we state our main results, Theorem \ref{thm:free-energy} and Theorem \ref{thm:lr-additive}, which generalize Theorem \ref{thm:bsk-normalized}, Theorem \ref{thm:bsk-iid}, and Theorem \ref{thm:log-LR}.
We assume that the disorder in both the free energy and the log likelihood ratio is given as a rectangular matrix.
The definition of a Rectangular matrix is as follows:

\begin{defn}(Rectangular matrix)\label{defn:rect}
	We say a $M \times N$ random matrix	$\Noise=(\noise)$ is a (real) rectangular matrix if the entries $\noise (1\le \hidx \le M, 1\le \widx\le N)$ are i.i.d. centered real random variables satisfying the following conditions:
	\begin{enumerate}
		\item The variance of the entries are normalized, i.e., $\mathbb{E}[\noise^2]=N^{-1}$ for all $\hidx, \widx$. 
		\item For any integer $d>2$, there exists a constant $B_d$, independent of $M, N$, such that $\mathbb{E}[|\sqrt{N}\noise|^d]\le B_d$.
		\item $M / N \to \ratio \in (0,\infty)$ as $M, N\to \infty$.
	\end{enumerate}
\end{defn}

Similarly, we use the following definition for the prior in both the free energy and the log likelihood ratio.
\begin{defn}[Prior]\label{defn:prior}
	For a random variable $\U=(u_1, \ldots,u_M)\in \mathbb{R}^M$ whose distribution is given by a prior $\mathcal{U}$, we assume that for $\hidx=1,\ldots,M$,
	\begin{itemize}
		\item The entries $u_\hidx$'s are i.i.d. and the distribution of $\sqrt{M}u_\hidx$ does not depend on $M$.
		\item The entries $u_\hidx$'s are symmetric (about $0$) and their variances are normalized so that $\mathbb{E}[u_\hidx^2]=M^{-1}$.
		\item For any integer $s>2$, there exists a constant $C_s^u$ such that $\mathbb{E}[|\sqrt{M}u_\hidx|^s]\le C_s^u$.
	\end{itemize}
	Similarly, for a random variable $\V=(v_1, \ldots,v_N)\in \mathbb{R}^N$ whose distribution is given by a prior $\mathcal{V}$, we assume that for $\widx=1,\ldots,N$,
	\begin{itemize}
		\item The entries $v_\widx$'s are i.i.d. and the distribution of $\sqrt{N}v_\widx$ does not depend on $N$.
		\item The entries $v_\widx$'s are symmetric (about $0$) and their variances are normalized so that $\mathbb{E}[v_\widx^2]=N^{-1}$.
		\item For any integer $s>2$, there exists a constant $C_s^v$ such that $\mathbb{E}[|\sqrt{N}v_\widx|^s]\le C_s^v$.
	\end{itemize}
	We further assume that $u_\hidx$'s and $v_\widx$'s are independent. 
\end{defn}

\begin{defn}[Coordinatewise symmetric event]
	We say an event $\Omega$ in $\mathcal{U}$ is coordinatewise symmetric if, for any $1\le \hidx\le M$, $\Omega$ holds with $(u_1, \ldots , u_{\hidx-1}, -u_\hidx, u_{\hidx+1}, \ldots ,u_M)$ whenever it holds with $(u_1,\ldots,u_{\hidx-1}, u_\hidx, u_{\hidx+1}, \ldots u_M)$ .
\end{defn}

Our proof relies on the boundedness of the second moment, which is defined as following condition.
\begin{cond}[Assumption 2.3 in \cite{Choo-Han-Lee2026}]\label{cond:prior-separable}
	Let $\snr_c>0$ be a given constant.	
	For a prior $\U \in \mathbb{R}^M$, there exists a sequence of coordinatewise symmetric events $\Omega_M$ in $\mathcal{U}$ with $\mathbb{P}_{\mathcal{U}}(\Omega_M^c)=o(1)$ such that
	\[ \mathbb{E}_{\U,\U'\sim\mathcal{U}}\left[\mathds{1}(\Omega_M(\U))\mathds{1}(\Omega_M(\U'))\exp\left( \frac{M\snr_c \langle\U,\U'\rangle^2}{2} \right) \right]=O(1).\]
\end{cond}
Here, the notation $\U,\U'\sim\mathcal{U}$ means that the spin vectors $\U$ and $\U'$ are independently drawn from the same prior distribution $\mathcal{U}$.
If $\sqrt{M}u_1, \ldots ,\sqrt{M}u_M$ are i.i.d. sub-Gaussian random variable with the variance proxy $\sigma^2$, it is well known that $w_c<1 /\sigma^2$ satisfies Condition \ref{cond:prior-separable} (see, Section 2.3 in \cite{Choo-Han-Lee2026}).

In the statement and the proof of our main results, we use the bipartite version of above condition.
\begin{cond}\label{cond:prior}
	Let $\snr_c>0$ be a given constant.
	There exists a sequence of coordinatewise symmetric events $\Omega_{M,N}$ in $\mathcal{U}\times \mathcal{V}$ with $\mathbb{P}_{\mathcal{U}\times \mathcal{V}}(\Omega_{M,N}^c)=o(1)$ such that
	\[ \mathbb{E}_{\U,\U'\sim\mathcal{U},\V,\V'\sim\mathcal{V}}\left[\mathds{1}(\Omega_{M,N}(\U,\V))\mathds{1}(\Omega_{M,N}(\U',\V'))\exp\left(\sqrt{MN}\snr_c \langle\U,\U'\rangle \langle\V,\V'\rangle \right) \right]=O(1).\]
\end{cond}
From Condition \ref{cond:prior-separable}, we can reconstruct the Condition \ref{cond:prior}.
\begin{lem}\label{lem:prior-joint}
	Assume that a constant $\snr_c^u> 0$ satisfies Condition \ref{cond:prior-separable} with a sequence of events $\Omega_M$ in $\mathcal{U}$.
	Assume that a constant $\snr_c^v>0$ satisfies Condition \ref{cond:prior-separable} with a prior $\V\in \mathbb{R}^N$, and a sequence of events $\Omega_N$ in $\mathcal{V}$.
	Then, $\snr_c=\sqrt{\snr_c^u\snr_c^v}$ satisfies Condition \ref{cond:prior} with $\Omega_{M,N}(\U,\V)=\Omega_M(\U)\cap\Omega_N(\V)$.
\end{lem}
\begin{proof}
	This immediately follows from the AM-GM inequality,
	\[ \sqrt{MN\snr_c^u\snr_c^v}\langle\U,\U'\rangle \langle\V,\V'\rangle\le \frac{M\snr_c^u \langle\U,\U'\rangle^2}{2}+\frac{N\snr_c^v \langle\V,\V'\rangle^2}{2}.\]
	Note that $\Omega_{M,N}(\U,\V)=\Omega_M(\U)\cap\Omega_N(\V)$ is coordinatewise symmetric.
\end{proof}
Thus, if $\sqrt{M}u_\hidx$'s are i.i.d. sub-Gaussian with variance proxy $\sigma_u^2$ and $\sqrt{N}v_\widx$'s are i.i.d. sub-Gaussian with variance proxy $\sigma_v^2$, then $\snr_c < 1 /(\sigma_u \sigma_v)$ satisfies Condition \ref{cond:prior}.

\subsection{Free energy of generalized bipartite SK model}
We define the partition function and the free energy of a generalized bipartite SK model, which includes both cases when $u_\hidx$ is i.i.d.($\epsilon_u=0$) and normalized($\epsilon_u=1$).
\begin{defn}\label{defn:free-energy}
	For an $M\times N$ rectangular matrix $\Noise=(\noise)$ and a prior $\mathcal{U}$ and $\mathcal{V}$, we define the partition function $Z_{M,N}\equiv Z_{M,N}(\beta)$ and the free energy $F_{M,N}\equiv F_{M,N}(\beta)$ at inverse temperature $\beta >0$ by
	\[ Z_{M,N}(\beta)=\mathbb{E}_{\mathcal{U}\times \mathcal{V}}\left[\exp\left( \beta\sqrt{MN} \sum_{\hidx,\widx}\noise \frac{u_\hidx v_\widx}{\|\U\|^{\epsilon_u} \|\V\|}\right)\right], \quad F_{M,N}(\beta)=\frac{1}{N}\log Z_{M,N}, \]
	for $\epsilon_u\in\{0,1\}$.
	Here, when $\|\U\|$ or $\|\V\|=0$, we define $u_\hidx v_\widx / (\|\U\|^{\epsilon_u}\|\V\|)=0$ for all $\hidx, \widx$.
\end{defn}

To extract the limiting free energy, we use a change-of-measure technique.
We define the quadratic measure tilting as follows:
\begin{defn}\label{defn:quadratic-tilt}
	Let a random variable $X$ have law $\mu$.
	For each $t\in \mathbb{R}$ such that $Z(t):=\mathbb{E}\left[e^{tX^2}\right]<\infty$, let $X^t$ be a random variable with the quadratically tilted law
	\[ \mu^t(\mathrm{d}x):=\frac{e^{tx^2}}{Z(t)} \mu(\mathrm{d}x).\]
	We refer to $X^t$ as the tilted version of $X$.	
\end{defn}

Our result on the free energy, which generalizes Theorem \ref{thm:bsk-iid} and Theorem \ref{thm:bsk-normalized}, is as follows:
\begin{thm}\label{thm:free-energy}
	Recall the defintion of the free energy $F_{M,N}(\beta)$ in Definition \ref{defn:free-energy}.
	Suppose that $\Noise$ is a Rectangular matrix, $\sqrt{N}\noise$ is symmetric, and strictly sub-Gaussian.
	Define
	\begin{equation}\label{eq:Lambda}
		\Lambda(t)=\frac{1}{M}\log\mathbb{E}_{\U}\left[e^{tM\|\U\|^{2(1-\epsilon_u)}}\right].
	\end{equation}
	Assume that in the small neighborhood of $t=\beta^2 / 2$, $\Lambda, \Lambda', \Lambda''<\infty$.
	Define
	\begin{equation}\label{eq:tilted-moments}
		\sigma^2 = \Lambda'\left( \frac{\beta^2}{2} \right),\quad m'_4=\begin{dcases}\tilde{m}_4=\Lambda''\left( \frac{\beta^2}{2} \right)\sigma^{-4} +1  & (\epsilon_u=0) \\m_4=M^2\mathbb{E}[u_1^4] & (\epsilon_u=1) \end{dcases}.
	\end{equation}
	Suppose that $\snr_c^u(\beta)$ satisfies Condition \ref{cond:prior-separable} with $\U^{(1-\epsilon_u)M\beta^2 /2} / \sigma$, a coordinatewise tilted version of $\U$ with normalization.
	Suppose that $\snr_c^v$ satisfies Condition \ref{cond:prior-separable} with $\V$.
	Let $\snr_c(\beta) = \sqrt{\snr_c^u(\beta)\snr_c^v}$.
	Then, for any $\beta>0$ that satisfies $\sqrt{\ratio}\beta^2\sigma^2 < \snr_c(\beta)$,
	\[ F_{M,N}(\beta) \to F(\beta):= \ratio\Lambda\left( \frac{\beta^2}{2} \right).\] and
	\[ N\left(F_{M,N}(\beta)-\frac{M}{N}\Lambda\left(\frac{\beta^2}{2}\right)\right) \implies \mathcal{N}(m_F, V_F),\]
	as $M,N\to \infty$, i.e., $NF_{M,N}(\beta)-M\Lambda(\beta^2 /2)$ converges in distribution to the Gaussian with mean $m_F$ and the variance $V_F$ as $M,N\to \infty$.
	Here,
	\[  m_F =\frac{1}{4}\log\left( 1-\ratio\beta^4\sigma^4 \right)+\frac{\ratio\beta^4\sigma^4}{24}(x_4-3)(m_4'n_4-3) \] and \[ V_F= -\frac{1}{2}\left( \log\left( 1-\ratio\beta^4\sigma^4 \right)-\frac{\ratio\beta^4\sigma^4}{2}(x_4-3)\right) ,\]
	where we let $n_4=N^2\mathbb{E}[v_1^4]$ and $x_4=N^2\mathbb{E}[\noise^4]$.	
\end{thm}

We prove Theorem \ref{thm:free-energy} in Section \ref{subsec:proof-free-energy}.
Note that $\sqrt{M}u_\hidx^{(1-\epsilon)M\beta^2 /2}$ does not depend on $M$ since
\[ \sqrt{M}u_\hidx^{(1-\epsilon)M\beta^2 /2} \equiv \left(\sqrt{M}u_\hidx\right)^{(1-\epsilon)\beta^2 /2},\]
and $\sqrt{M}u_\hidx$ does not depend on $M$.

\begin{rmk}\label{rmk:free-energy-normalized}
	When the prior is normalized, i.e. $\epsilon_u=1$, since $\Lambda(t)=t$ and $\snr_c^u(\beta)\equiv \snr_c^u$ which satisfies Condition \ref{cond:prior-separable} with $\U$, Theorem \ref{thm:free-energy} is equivalent to Theorem \ref{thm:bsk-normalized} without any change for any $\beta \in (0, (\snr_c^u \snr_c^v / \ratio)^{1/4})$.
\end{rmk}

\begin{rmk}\label{rmk:free-energy-iid}
	When the prior is i.i.d., i.e. $\epsilon_u=0$, if $\snr_c^u(\beta)$ satisfies Condition \ref{cond:prior-separable} with $\U^{M\beta^2 /2} / \sigma$, then there exists a sequence of coordinatewise symmetric events $\Omega_M$ in $\mathcal{U}$ with $\mathbb{P}_{\mathcal{U}}(\Omega_M(\U^{M\beta^2 /2} /\sigma)^c)=o(1)$ such that
	\begin{equation}\label{eq:tilted-threshold-expanded}
		\begin{split}
			&e^{-2M\tilde{\Lambda}(\beta^2 /2)}\mathbb{E}_{\U,\U'\sim\mathcal{U}}\left[\mathds{1}(\Omega'_M(\U))\mathds{1}(\Omega'_M(\U'))\exp\left(\frac{M\snr_c^u(\beta)}{2\sigma^4}\langle\U,\U'\rangle^2 +\frac{M\beta^2}{2}(\|\U\|^2+\|\U'\|^2)\right) \right] \\
			&=O(1),
		\end{split}
	\end{equation}
	where we let $\Omega'_M(\U):=\Omega_M(\U^{M\beta^2 /2} /\sigma)$.
	Since $\sigma$ is an increasing function of $\beta$, $\snr_c^u$ is not a monotonic function of $\beta$ due to the competition between two terms in exponential part.
	For this reason, it is hard to get the result for the whole sub-critical regime, i.e., $\ratio\beta^4 \sigma^4 <1$.

	Note that $\Lambda(\beta^2 /2)<\infty$ implies that $\sqrt{M}u_\hidx$ is sub-Gaussian with some variance proxy.
	Therefore, the tilted variable itself is also a sub-Gaussian.
\end{rmk}

\begin{rmk}\label{rmk:known-results}
	To compare with the known results where normalization is $(M+N)^{-1 /2}$, we fix the ratio $M / N=\ratio$.
	Then, our results can be interpreted with the change of effective $\beta$ as follows:
	\begin{align*} 
		Z'_{M,N}(\beta)&:=\mathbb{E}_{\mathcal{U}\times\mathcal{V}}\left[\exp\left(\frac{\beta}{\sqrt{M+N}}\sum_{\hidx,\widx}(\sqrt{N}\noise)\frac{(\sqrt{M}u_\hidx)(\sqrt{N}v_\widx)}{\|\U\|^{\epsilon_u}\|\V\|}\right) \right]=Z_{M,N}\left(\frac{\beta}{\sqrt{\ratio+1}}\right),\\
		F'_{M,N}(\beta)&:=\frac{1}{M+N}\log Z'_{M,N}(\beta)=\frac{1}{\ratio+1} F_{M,N}\left( \frac{\beta}{\sqrt{\ratio+1}} \right).
	\end{align*}
	Define
	\[ \snr_c'=\snr_c\left( \frac{\beta}{\sqrt{\ratio+1}} \right),\quad (\sigma')^2 = \Lambda'\left( \frac{\beta^2}{2(\ratio+1)} \right) .\]
	Applying Theorem \ref{thm:free-energy}, for any $\beta >0$ that satisfies $\sqrt{\ratio}\beta^2(\sigma')^2 <(\ratio+1)\snr_c'$, 
	\[ F'_{M,N}(\beta)\to F'(\beta):=\frac{\ratio}{\ratio+1}\Lambda\left( \frac{\beta^2}{2(\ratio+1)} \right)  .\]
	and
	\[ (M+N)\left( F_{M,N}'(\beta)-\frac{M}{M+N}\Lambda\left( \frac{\beta^2}{2(\ratio+1)} \right) \right)\implies \mathcal{N}(m_F', V_F'),\]
	as $M,N\to \infty$ where
	\[ m_F'=\frac{1}{4}\log\left( 1-\frac{\ratio\beta^4(\sigma')^4}{(\ratio+1)^2}\right)+\frac{\ratio\beta^4(\sigma')^4}{24(\ratio+1)^2}(x_4-3)(m_4''n_4-3) \]
	and
	\[ V_F'=-\frac{1}{2}\left( \log\left( 1-\frac{\ratio\beta^4(\sigma')^4}{(\ratio+1)^2} \right)-\frac{\ratio\beta^4(\sigma')^4}{2(\ratio+1)^2}(x_4-3) \right),\]
	with 
	\[ m_4''=\begin{cases}
		\Lambda''\left( \frac{\beta^2}{2(\ratio+1)} \right)(\sigma')^{-4}+1 & (\epsilon_u=0) \\
		M^2\mathbb{E}[u_1^4] & (\epsilon_u=1)
	\end{cases}. \]

	\begin{enumerate}
		\item 	For the bipartite SK model, where both $\mathcal{U}$ and $\mathcal{V}$ are Rademacher prior and $\epsilon_u = 1$, we have $\Lambda(t)=t, \Lambda'(t)=1,\sigma'=1, m''_4=n_4=1$.
			Since $\snr_c<1$ satisfies the Condition, for $\beta\in (0,1)$,
			\begin{align*}
				m'_F &= \frac{1}{4}\log\left( 1-\frac{\ratio\beta^4}{(\ratio+1)^2} \right) -\frac{\ratio\beta^4}{(\ratio+1)^2}\frac{x_4-3}{12}, \\
				V'_F &= -\frac{1}{2}\left( \log\left(1-\frac{\ratio\beta^4}{(\ratio+1)^2}\right)-\frac{\ratio\beta^4}{(\ratio+1)^2}\frac{x_4-3}{2} \right),
			\end{align*}
			which recovers the result in \cite{Dey-Wu2023} when $\noise$ is strictly sub-Gaussian.
		\item For the bipartite SSK model, where both spins are the spherical prior, which is equal to the normalized prior with Gaussian entries, we let $\epsilon_u=1$. 
			With $m''_4 =n_4= 3$ and $\snr_c < 1$ satisfying the Condition, for $\beta\in (0,1)$,
			\begin{align*}
				m_F &= \frac{1}{4}\log\left( 1-\frac{\ratio\beta^4}{(\ratio+1)^2} \right)+\frac{\ratio\beta^4}{(\ratio+1)^2}\frac{x_4-3}{4}, \\
				V_F &= -\frac{1}{2}\left( \log\left( 1-\frac{\ratio\beta^4}{(\ratio+1)^2} \right) -\frac{\ratio\beta^4}{(\ratio+1)^2}\frac{x_4-3}{2} \right) . 
			\end{align*}
			It recovers the result in \cite{Baik-Lee2020} when $\noise$ is strictly sub-Gaussian.
	\end{enumerate}
\end{rmk}

\subsection{Free energy of generalized Hopfield Model}\label{subsec:hopfield}
Recall the definition of the partition function and the free energy of a generalized Hopfield model.
\begin{defn}[Partition function and free energy of generalized Hopfield model]\label{defn:free-energy-hopfield}
	For an $M\times N$ rectangular matrix $\Noise=(\noise)$, a hidden prior $\mathcal{U}$ and a visible prior $\mathcal{V}$, we define the partition function $ \hat{Z}_{M,N}(\beta)$ and the free energy $\hat{F}_{M,N}(\beta)$ at inverse temperature $\beta>0$ by
	\begin{align*}
		\hat{Z}_{M,N}&\equiv \hat{Z}_{M,N}(\beta)=\mathbb{E}_{\mathcal{V}}\left[\exp\left(\sum_{\hidx}\psi_{u}\left( \beta\sqrt{N}\sum_{\widx}\noise\frac{v_\widx}{\|\V\|} \right)\right)\right], \\
		\hat{F}_{M,N}&\equiv \hat{F}_{M,N}(\beta)=\frac{1}{N}\log \hat{Z}_{M,N},
	\end{align*}
	where $\psi_u(t) = \log\mathbb{E}[\exp(t\sqrt{M}u_\hidx)]$ is a cumulant generating function of $\sqrt{M}u_\hidx$.
	Here, when $\|\V\|=0$, we define $v_\widx / \|\V\|=0$ for all $\widx$.
\end{defn}

Directly from the Theorem \ref{thm:free-energy}, we get:
\begin{cor}\label{cor:free-energy-hopfield}
	Recall the definition of the free energy $\hat{F}_{M,N}(\beta)$ in Definition \ref{defn:free-energy-hopfield}.
	Suppose that $\Noise$ is a Rectangular matrix, $\sqrt{N}\noise$ is symmetric and strictly sub-Gaussian.
	Then, the statement in Theorem \ref{thm:free-energy} with $\epsilon_u=0$ holds with $\hat{F}_{M,N}(\beta)$.
\end{cor}

\begin{rmk}[Gaussian-Rademacher hidden prior]\label{rmk:hopfield-bgst}
	Consider the Gaussian-Rademacher interpolation studied in \cite{Barra-Genovese-Sollich-Tantari2017,Barra-Genovese-Sollich-Tantari2018}.
	Let $\sqrt{N}v_\widx$ be i.i.d. Rademacher variables and let $\sqrt{M}u_\hidx$ be independent copies of
	\[ H_s =\sqrt{s}z + \sqrt{1-s}\tau, \quad 0\le s\le 1, \]
	where $z\sim \mathcal{N}(0, 1)$ and $\tau$ is an independent Rademacher variable.
	We maintain the assumptions on $\Noise$ in Corollary \ref{cor:free-energy-hopfield}.
	Then $\|\V\|=1, n_4=1$, and
	\[ \psi_u(t) = \frac{s}{2}t^2+\log\cosh(\sqrt{1-s}t) .\]
	Conditioning on $\tau$ and using cumulant generating function of noncentral chi-squared distribution, we get
	\begin{equation}\label{eq:bgst-Lambda}
		\tilde{\Lambda}_s(t)=\log\mathbb{E}[e^{tH_s^2}] =-\frac{1}{2}\log(1-2st)+\frac{(1-s)t}{1-2st},
	\end{equation}
	for $st<1 /2$.
	Applying $t = \beta^2 /2$, we obtain
	\[  \tilde{\sigma}_s^2=\frac{1-s^2\beta^2}{(1-s\beta^2)^2}, \quad	\tilde{m}_{4,s}=3-\frac{2(1-s)^2}{(1-s^2\beta^2)^2}. \]
	Since $H_s$ is a Lee-Yang type random variable, the conditions $s\beta^2 <1$ and $\ratio \beta^4 \tilde{\sigma}_s^4 <1$ are equivalent to $0<\beta<\beta_c(s)$, where
	\begin{equation}\label{eq:bgst-beta-critical}
		\beta_c(s)=\left(s+\frac{\sqrt{\ratio}}{2}+\frac{1}{2}\sqrt{\ratio+4\sqrt{\ratio}\,s(1-s)}\right)^{-1/2}.
	\end{equation}
	Thus Theorem \ref{thm:Hopfield} {implies} that, throughout this regime,
	\begin{equation}\label{eq:bgst-free-energy-limit}
		\hat{F}_{M,N}(\beta)\to -\frac{\ratio}{2}\log(1-s\beta^2)+\frac{\ratio(1-s)\beta^2}{2(1-s\beta^2)}
	\end{equation}
	and
	\begin{equation}\label{eq:bgst-free-energy-fluctuation}
		\log\hat{Z}_{M,N}(\beta)-\left(-\frac{M}{2}\log(1-s\beta^2)+\frac{M(1-s)\beta^2}{2(1-s\beta^2)}\right) \implies\mathcal{N}(\tilde{m}_s,\tilde{V}_s),
	\end{equation}
	where
	\begin{align*} 
		m_s &=\frac{1}{4}\log\left(1-\frac{\ratio \beta^4(1-s^2\beta^2)^2}{(1-s\beta^2)^4}\right)-\frac{\ratio\beta^4(1-s)^2}{12(1-s\beta^2)^{4}}(x_4-3),\\
		V_s &=-\frac{1}{2}\log\left(1-\frac{\ratio\beta^4(1-s^2\beta^2)^2}{(1-s\beta^2)^4}\right)+\frac{\ratio\beta^4(1-s^2\beta^2)^2}{4(1-s\beta^2)^4}(x_4-3). 
	\end{align*}
	In particular, if the disorder also follows the Gaussian-Rademacher interpolation with parameter $t$, then $x_4=3-2(1-t)^2$.
	The endpoints $s=0$ and $s=1$ recover the bipartite SK model and the Hopfield model with Gaussian hidden spins, respectively.

	The threshold $\beta_c(s)$ in \eqref{eq:bgst-beta-critical} agrees with the paramagnetic-spin glass transition boundary predicted in \cite{Barra-Genovese-Sollich-Tantari2017, Barra-Genovese-Sollich-Tantari2018} for Rademacher visible spins.
	For $\beta\in (0,\beta_c(s))$, our result rigorously establishes the limiting free energy and its Gaussian fluctuations for general symmetric, strictly sub-Gaussian disorder, extending beyond the Gaussian-Rademacher interpolation family considered in those works.
\end{rmk}

\begin{rmk}[$\phi^4$ hidden prior]\label{rmk:phi-4}
	Quartic potentials provide a natural way to control the amplitudes of continuous hidden priors.
	They appear in the $\phi^4$ restricted Boltzmann machines of Bachtis, Aarts and Lucini \cite{Bachtis-Aarts-Lucini2021}, where both priors may have quadratic and quartic terms.
	We consider the i.i.d. $\phi^4$ hidden prior with normalized visible prior.
	Fix $a>0$ and $b\in \mathbb{R}$, and define
	\[ \Phi_a(b)=\int_\mathbb{R}e^{-ax^4+bx^2}\mathrm{d}x, \quad \nu^2=\frac{\Phi'_a(b)}{\Phi_a(b)} .\]
	Let $W$ have density $\Phi_a(b)^{-1}\exp(-ax^4+bx^2)$ and let $\sqrt{M}u_\hidx$ be independent copies of the unit-variance hidden spin $W / \nu$.
	The potential has a single well for $b\le 0$ and double wells for $b>0$.
	These laws have the Lee--Yang property \cite{Simon-Griffiths1973}.
	The quadratic tilt is explicit:
	\[ \tilde{\Lambda}_\phi(t)=\log\frac{\Phi_a(b+t /\nu^2)}{\Phi_a(b)} .\]
	In particular, the tilted prior remains in the Lee--Yang family.
	With $b_\beta = b+\beta^2 /(2\nu^2)$, we obtain
	\[ \tilde{\sigma}_\phi^2 = \frac{\Phi'_a(b_\beta)}{\nu^2\Phi_a(b_\beta)}, \quad \tilde{m}_{4,\phi}=\frac{\Phi''_a(b_\beta)\Phi_a(b_\beta)}{\Phi'_a(b_\beta)^2} .\]
	Suppose that the disorder and visible prior satisfy assumptions of Theorem \ref{thm:bsk-iid}.
	Then, for every $\beta>0$ such that $\ratio \beta^4 \tilde{\sigma}_\phi^4 <1$,
	\[ \hat{F}_{M,N}(\beta)\to \ratio \log\frac{\Phi_a(b_\beta)}{\Phi_a(b)},\quad \log \hat{Z}_{M,N}(\beta)-M \log\frac{\Phi_a(b_\beta)}{\Phi_a(b)} \implies \mathcal{N}(m_\phi, V_\phi),\]
	where
	\begin{align*}
		m_\phi &= \frac{1}{4}\log(1-\ratio\beta^4\tilde{\sigma}_\phi^4)+\frac{\ratio\beta^4\tilde{\sigma}_\phi^4}{24}(x_4-3)(\tilde{m}_{4,\phi}n_4-3), \\
		V_\phi &=-\frac{1}{2}\log(1-\ratio\beta^4\tilde{\sigma}_\phi^4)+\frac{\ratio\beta^4\tilde{\sigma}_\phi^4}{4}(x_4-3). 
	\end{align*}
	
\end{rmk}

Applying Corollary \ref{cor:free-energy-hopfield} with Gaussian hidden prior, we get the following result:
\begin{cor}\label{cor:free-energy-gaussian-hopfield}
	Recall the definition of the free energy $F_{M,N}^H(\beta)$ in Corollary \ref{cor:Hopfield-Gaussian}.
	Suppose that $\Noise$ is a Rectangular matrix, $\sqrt{N}\noise$ is symmetric and strictly sub-Gaussian.
	Suppose that a constant $\snr_c^v>0$ satisfies Condition \ref{cond:prior-separable} with $\V$.
	Let $\snr_c=\sqrt{\snr_c^v}$.
	Then, the statement in Corollary \ref{cor:Hopfield-Gaussian} holds without any change for any $\beta\in (0, (1+\sqrt{\ratio}/\snr_c)^{-1/2})$.
\end{cor}

\begin{rmk}
	For the classical Hopfield model, where both the disorder $\mathcal{\Noise}$ and the prior $\mathcal{V}$ are Rademacher distribution, we have $n_4=1, x_4=1$.
	Since $\snr_c< 1$ satisfies Condition \ref{cond:prior}, for $\beta\in {(0,(1+\sqrt{\ratio})^{-1 /2})}$, 
	\[ m_F^{H} = \frac{1}{4}\log\left(1-\frac{\ratio\beta^4}{(1-\beta^2)^2}\right), \quad V_F^{H}=-\frac{1}{2}\log\left( 1-\frac{\ratio\beta^4}{(1-\beta^2)^2}\right)-\frac{\ratio \beta^4}{2(1-\beta^2)^2}. \]
	This recovers the results in \cite{Comets-Kurkova-Trashorras2004}.
	In \cite{Comets-Kurkova-Trashorras2004}, they proved that
	\[\log Z^{H}_{M,N} - \log\mathbb{E}[Z^{H}_{M,N}] \implies \mathcal{N}(-\rho', 2\rho'), \quad \rho'= -\frac{1}{4}\log\left(1-\frac{\ratio\beta^4}{(1-\beta^2)^2}\right)-\frac{\ratio \beta^4}{4(1-\beta^2)^2} \]
	as $M, N\to \infty$.
	The difference between limiting mean, $\ratio\beta^4 / (4(1-\beta^2)^2)$ is due to the difference in centering, since
	\[ \log\mathbb{E}[Z^{H}_{M,N}]= -\frac{M}{2}\log(1-\beta^2)-\frac{M}{N}\frac{\beta^4}{4(1-\beta^2)^2}+o(1) .\]
	(See (12) in \cite{Comets-Kurkova-Trashorras2004}.)

	For the Gaussian Hopfield model, where the disorder $\mathcal{\Noise}$ is Gaussian distribution and the prior $\mathcal{V}$ is Rademacher distribution, we have $n_4=1, x_4=3$.
	Since $\snr_c<1$ satisfies Condition \ref{cond:prior}, for $\beta \in {(0, (1+\sqrt{\ratio})^{-1 /2})}$,
	\[ m_F^{H}=\frac{1}{4}\log\left( 1-\frac{\ratio\beta^4}{(1-\beta^2)^2} \right), \quad V_F^{H}=-\frac{1}{2}\log\left( 1-\frac{\ratio\beta^4}{(1-\beta^2)^2} \right). \]
	which recovers the results in \cite{Barra-Genovese-Guerra-Tantari2012}.

	Also, SSK model with sample covariance matrix in Section 3.3 of \cite{Baik-Lee2016} is equivalent to Hopfield model with normalized gaussian spin.
	Then $n_4=3$ and hence
	\begin{align*} 
		m_F^H &= \frac{1}{4}\log\left(1-\frac{\ratio\beta^4}{(1-\beta^2)^2}\right)+\frac{\ratio \beta^4}{4(1-\beta^2)^2}(x_4-3), \\
		V_F^H &= -\frac{1}{2}\log\left( 1-\frac{\ratio\beta^4}{(1-\beta^2)^2}\right)+\frac{\ratio\beta^4}{4(1-\beta^2)^2}(x_4-3), 
	\end{align*}
	which recovers the results in \cite{Baik-Lee2016}.
\end{rmk}

\subsection{Log likelihood ratio in spiked rectangular model}

The spiked rectangular matrices are defined as follows.
\begin{defn}(Spiked rectangular matrix)\label{defn:additive}
	We say a $M \times N$ random matrix $\Data=\sqrt{\SNR}\U\V^T+\Noise$ is a spiked rectangular matrix with mean $\U, \V$ and a signal-to-noise ratio(SNR) $\SNR$ if $\U=(u_1, \ldots , u_M)^T\in \mathbb{R}^M, \V=(v_1,\ldots,v_N)^T\in \mathbb{R}^N$ with $\|\U\|=\|\V\|=1$, and $\Noise$ is a rectangular matrix.
\end{defn}

Denote by $\mathbf{H}_0$ the null hypothesis and $\mathbf{H}_1$ the alternative hypothesis, i.e.
\[ \mathbf{H}_0:\SNR=0, \qquad \mathbf{H}_1:\SNR=\snr>0.\]
Denote by $\mathbb{P}_0$ the joint probability of the data matrix under $\mathbf{H}_0$ and $\mathbb{P}_1$ the joint probability of the data matrix under $\mathbf{H}_1$.
Then the likelihood ratio is
\[ \mathcal{L}(\Data;\U,\V,\snr)= \prod_{\hidx,\widx}\frac{p(\sqrt{N}\data-\sqrt{\snr N}u_\hidx v_\widx)}{p(\sqrt{N}\data)}.\]
With the Bayesian viewpoint, we consider the likelihood ratio
\[ \mathcal{L}(\Data;\snr) = \mathbb{E}_{\mathcal{U}\times\mathcal{V}}[\mathcal{L}(\Data;\U,\V,\snr)]=\mathbb{E}_{\mathcal{U}\times \mathcal{V}}\left[\prod_{\hidx,\widx}\frac{p(\sqrt{N}\data-\sqrt{\snr N}u_\hidx v_\widx)}{p(\sqrt{N}\data)}\right].\]

We assume $\mathbf{H}_0$, i.e., $\Data=\Noise$, since the likelihood ratio under $\mathbf{H}_0$ naturally corresponds to the free energy of the generalized bipartite SK model in our analysis.
For the likelihood ratio, there are no terms corresponding to the limiting free energy but the leading order term is asymptotically normally distributed.
We thus consider the likelihood ratio for both normalized and i.i.d. priors as follows:
\begin{defn}(Likelihood ratio)\label{defn:lr-additive}
	For a $M\times N$ rectangular matrix $\Noise=(\noise)$ and a prior $\U\in \mathcal{U}$ and $\V\in \mathcal{V}$, we define the likelihood ratio given $\mathbf{H}_0$, $\mathcal{L}(\snr)$, by
	\[ \mathcal{L}(\snr, \epsilon_u, \epsilon_v) = \mathbb{E}_{\mathcal{U}\times \mathcal{V}}\left[\prod_{\hidx,\widx}\frac{p(\sqrt{N}\noise-\sqrt{\snr N}u_\hidx v_\widx / (\|\U\|^{\epsilon_u}\|\V\|^{\epsilon_v}))}{p(\sqrt{N}\noise)}\right],\]
	where $\epsilon_u, \epsilon_v \in \{0,1\}$. 
	Here, when $\|\U\|=0$ or $\|\V\|=0$, we define $u_\hidx v_\widx / (\|\U\|^{\epsilon_u}\|\V\|^{\epsilon_v})=0$ for all $\hidx,\widx$. 
\end{defn}

Our result on the asymptotic normality of the log likelihood ratio, which generalizes Theorem \ref{thm:log-LR}, is as follows:

\begin{thm}\label{thm:lr-additive}
	Recall the definition of the likelihood ratio $\mathcal{L}(\snr,\epsilon_u, \epsilon_v)$ in Definition \ref{defn:lr-additive}.
	Suppose that $\Noise$ satisfies Assumption \ref{assum:density}.
	Suppose that $\snr_c$ satisfies Condition \ref{cond:prior}.
	Then, the statement in Theorem \ref{thm:log-LR} holds without any change for any $\snr\in(0,\snr_c\sqrt{\ratio}/F_p)$.
\end{thm}

We prove Theorem \ref{thm:lr-additive} in Section \ref{subsec:proof-lr-additive}.

\begin{rmk}
	Theorem \ref{thm:lr-additive} implies that the log likelihood ratio converges in distribution to $\mathcal{N}(\rho_L, 2\rho_L)$ under $\mathbf{H}_1$, which is a consequence of Le Cam's third lemma \cite{Cam-1960, Van-2000}.
	It also implies that the sum of the Type-I error and the Type-II error of the likelihood ratio test converges to
	\[ \mathrm{erfc}(\sqrt{\rho_L} /2) = \frac{2}{\sqrt{\pi}}\int_{\sqrt{\rho_L} /2}^\infty e^{-x^2} \mathrm{d}x.\]
	(See, e.g., Corollary 2.5 in \cite{Chung-Lee-Lee2025}.)	
\end{rmk}

\subsection{Universal result}\label{subsec:prop}

The proof of Theorem \ref{thm:free-energy} and \ref{thm:lr-additive} is based on the following bipartite counterpart of Proposition 2.12 of \cite{Choo-Han-Lee2026}.
It concerns the same fourth-order polynomial expansion, now indexed by rectangular entries and averaged over two independent priors.
We state the assumptions and limiting parameters before explaining the changes to the multigraph argument.
\begin{prop}\label{prop:general}
	Suppose that the random variables $\{P_{\hidx\widx}^{(n)}\}_{\hidx,\widx}\;(n=1,2,3,4)$ satisfy the following:
	\begin{itemize}
		\item The random vectors $(P_{\hidx\widx}^{(1)}, P_{\hidx\widx}^{(2)}, P_{\hidx\widx}^{(3)}, P_{\hidx\widx}^{(4)})$ are independent and identically distributed.
		\item For any $n$ and any integer $s>2$, there exists a constant $C_s^P$ such that $\mathbb{E}[|P_{\hidx\widx}^{(n)}|^s]\le C_s^P$.
		\item For any $n$, $\mathbb{E}[P_{\hidx\widx}^{(n)}]=0$, $\mathbb{E}[(P_{\hidx\widx}^{(n)})^2]=\tau'_n$ and $\tau'_n\to \tau_n$ as $M,N\to \infty$, for some ($M,N$-independent) constants $\tau_n$.
		\item $\mathbb{E}[P_{\hidx\widx}^{(1)}P_{\hidx\widx}^{(2)}]=\mathbb{E}[P_{\hidx\widx}^{(2)}P_{\hidx\widx}^{(3)}]=\mathbb{E}[P_{\hidx\widx}^{(3)}P_{\hidx\widx}^{(4)}]=\mathbb{E}[P_{\hidx\widx}^{(1)}P_{\hidx\widx}^{(4)}]=0$ and $\mathbb{E}[P_{\hidx\widx}^{(1)}(P_{\hidx\widx}^{(2)})^2]=0$.
	\end{itemize}
	Fix $\epsilon\in (0, 1 / 60)$.
	Let
	\begin{equation}\label{eq:cutting-event}
		\begin{split}
			\Omega^\epsilon_{\mathcal{U}\times \mathcal{V}} := 
			&\{\max_{1\le\hidx\le M}|u_\hidx|\le M^{-1 /2+\epsilon}\}\cap \{|\|\U\|-1|\le M^{-1 /5}\} \\
			&\cap\{\max_{1\le \widx\le N}|v_\widx|\le N^{-1 /2+\epsilon}\}\cap \{|\|\V\|-1|\le N^{-1 /5}\}.
		\end{split}
	\end{equation}
	and 
	\begin{equation}\label{eq:cutting-event-noise}
		\Omega_{\mathcal{\Noise}}^\epsilon := \{\max_{\hidx,\widx}|\noise| \le N^{-1/ 2 +\epsilon}\}.
	\end{equation}
	Suppose that $\snr_c$ satisfies Condition \ref{cond:prior}.
	Suppose that on $\Omega^\epsilon_{\mathcal{U}\times \mathcal{V}}\cap\Omega_{\mathcal{\Noise}}^\epsilon$,
	\[\prod_{\hidx,\widx}\left( 1+P_{\hidx\widx}^{(1)}\frac{\sqrt{N}u_\hidx v_\widx}{\|\U\|^{\epsilon_u}\|\V\|^{\epsilon_v}}+P_{\hidx\widx}^{(2)}\frac{Nu_\hidx^2 v_\widx^2}{\|\U\|^{2\epsilon_u}\|\V\|^{2\epsilon_v}}+P_{\hidx\widx}^{(3)}\frac{N^{3 /2 }u_\hidx^3 v_\widx^3}{\|\U\|^{3\epsilon_u}\|\V\|^{3\epsilon_v}}+P_{\hidx\widx}^{(4)}\frac{N^2u_\hidx^4 v_\widx^4}{\|\U\|^{4\epsilon_u}\|\V\|^{4\epsilon_v}}\right) \]
	is nonnegative.
	If $\tau_1 <\snr_c\sqrt{\ratio}$,
	\begin{equation}\label{eq:expansion}
		\begin{split}
			\log\mathbb{E}_{\mathcal{U}\times \mathcal{V}}\left[\mathds{1}(\Omega^\epsilon_{\mathcal{U}\times \mathcal{V}})\prod_{\hidx,\widx}\right.
			&\left( 1+P_{\hidx\widx}^{(1)}\frac{\sqrt{N}u_\hidx v_\widx}{\|\U\|^{\epsilon_u}\|\V\|^{\epsilon_v}}+P_{\hidx\widx}^{(2)}\frac{Nu_\hidx^2 v_\widx^2}{\|\U\|^{2\epsilon_u}\|\V\|^{2\epsilon_v}} \right.\\
			&\quad\left.\left.+P_{\hidx\widx}^{(3)}\frac{N^{3 /2 }u_\hidx^3 v_\widx^3}{\|\U\|^{3\epsilon_u}\|\V\|^{3\epsilon_v}}+P_{\hidx\widx}^{(4)}\frac{N^2u_\hidx^4 v_\widx^4}{\|\U\|^{4\epsilon_u}\|\V\|^{4\epsilon_v}}\right) \right] 
			\implies \mathcal{N}(-\rho, 2\rho)
		\end{split}
	\end{equation}
	as $M, N\to \infty$ for any $\epsilon_u, \epsilon_v\in\{0,1\}$, where
	\[ \rho:= -\frac{1}{4}\left( \log\left(1-\frac{\tau_1^2}{\ratio}\right)+\frac{1}{\ratio}(\tau_1^2-2\tau_2) \right).\]
\end{prop}

Sections \ref{sec:multigraph}--\ref{sec:technical} prove the above proposition using the multigraph strategy of Section 4--6 in \cite{Choo-Han-Lee2026}.
Main difference is there are no self-loops, and every cycle has even length, so the cycle contribution and the limiting parameters differ from those in the Wigner setting.
The estimates also account for $M/N\to\ratio$ and $\tau_n'\to\tau_n$; the strict inequality $\tau_1<\snr_c\sqrt{\ratio}$ provides room to control these varying parameters.

We remark that the orthogonality condition of $P_{\hidx\widx}^{(i)}$ is satisfied if there exist symmetric, i.i.d. random variables $V_{\hidx\widx}$ such that $P_{\hidx\widx}^{(1)}$ and $P_{\hidx\widx}^{(3)}$ are odd functions of $V_{\hidx\widx}$ and $P_{\hidx\widx}^{(2)}$ and $P_{\hidx\widx}^{(4)}$ are even functions of $V_{\hidx\widx}$.

Note that the nonnegativity assumption is obtained from the fact that the above form is from the Taylor expansion of likelihood ratio or partition function, which are both positive.

\section{Proof of main results}\label{sec:thm-proof}

We derive the main results from Proposition \ref{prop:general}.
We first recover the introductory statements from the general theorems.
Theorems \ref{thm:bsk-normalized} and \ref{thm:log-LR} are special cases of Theorems \ref{thm:free-energy} and \ref{thm:lr-additive}, respectively.
\begin{proof}[Proof of Theorem \ref{thm:bsk-normalized} and Theorem \ref{thm:log-LR}]
	Theorem \ref{thm:bsk-normalized} and Theorem \ref{thm:log-LR} are special cases of Theorem \ref{thm:free-energy} with $\epsilon_u=1$ and Theorem \ref{thm:lr-additive}, respectively, since by definition the strictly sub-Gaussian prior is the sub-Gaussian prior with variance proxy $1$.
\end{proof}

{We first note that \ref{thm:Hopfield} is equivalent to Theorem \ref{thm:bsk-iid} with different normalization.}
For the Theorem \ref{thm:bsk-iid}, we use the property of Lee--Yang type random variables to show that Theorem \ref{thm:bsk-iid} is a corollary of the general result, Theorem \ref{thm:free-energy}.
\begin{proof}[Proof of Theorem \ref{thm:bsk-iid}]
	If $\sqrt{M}u_\hidx$ is a Lee--Yang type prior, then by Section 2.2 Theorem 2 of \cite{Newman-Wu2019}, the tilted prior $\sqrt{M}u_\hidx^{M\beta^2 /2}$ is also of Lee--Yang type.
	Also by Proposition 13 and (8) of \cite{Newman-Wu2019a}, Lee--Yang type prior with variance $1$ is a strictly sub-Gaussian prior, thus $\sqrt{M}u_\hidx^{M\beta^2 /2} / \sigma$ is a strictly sub-Gaussian prior.
	Therefore, $\snr_c^u <1$ satisfies Condition \ref{cond:prior-separable}.
	Now by the strict sub-Gaussianity of $\sqrt{N}v_\widx$, $\snr_c^v<1$ always satisfies Condition \ref{cond:prior-separable} with $\V$.
	We conclude that Theorem \ref{thm:bsk-iid} is a special case of Theorem \ref{thm:free-energy} with $\epsilon_u=0$.
\end{proof}

We next prove the general theorems, starting with Theorem \ref{thm:lr-additive}, the likelihood ratio in Section \ref{subsec:proof-lr-additive}.
As in Section 3 of \cite{Choo-Han-Lee2026}, a fourth-order polynomial expansion reduces this problem to the universality proposition; here the coefficients and limiting parameters reflect the rectangular model.
For Theorem \ref{thm:free-energy}, the free energy part, proof in Section \ref{subsec:proof-free-energy} requires an additional reduction when one prior is i.i.d.
We first separate the leading order term by a quadratic change of measure, then establish the truncation and moment estimates under the tilted law before applying Proposition \ref{prop:general}.

We first introduce the following shorthand notations, which will be frequently used in the rest of the paper.
\begin{nrmk}
	We use the standard big-$O$ and little-$o$ notation:
	$a_N=O(b_N)$ implies that there exists $N_0$ such that $|a_N|\le C|b_N|$ for some constant $C>0$ independent of $N$ for all $N\ge N_0$;
	$a_N=o(b_N)$ implies that for any positive constant $\epsilon$ there exists $N_0$ such that $|a_N|\le \epsilon|b_N|$ for all $N\ge N_0$.

	For probabilistic bounds, we use the notation $X=o_{\mathcal{X}}(Y)$ or $X=o_{\mathcal{U}}(Y), X=o_{\mathcal{V}}(Y)$, which implies that for any (small) $\epsilon>0$, $\mathbb{P}_{\mathcal{X}}(|X|>\epsilon|Y|)\to 0$ or $\mathbb{P}_{\mathcal{U}}(|X|>\epsilon|Y|), \mathbb{P}_{\mathcal{V}}(|X|>\epsilon|Y|)\to 0$ as $N\to \infty$, respectively.

	For an event $\Omega$, we say that $\Omega$ holds with high probability if $\mathbb{P}(\Omega^c)=o(1)$.
	Again, we say $\Omega$ holds with high probability with respect to $\mathcal{X}$ (or $\mathcal{U}, \mathcal{V}$) if $\mathbb{P}_{\mathcal{X}}(\Omega^c)=o(1)$ (or $\mathbb{P}_{\mathcal{U}}(\Omega^c)=o(1), \mathbb{P}_{\mathcal{V}}(\Omega^c)=o(1))$.

	For an event $\Omega$, we say that $\Omega$ holds with overwhelming probability if for any (large) $D>0$, there exists $N_0\equiv N_0(D)$ such that $\mathbb{P}(\Omega^c)<N^{-D}$ whenever $N>N_0$.
	Again, we say $\Omega$ holds with overwhelming probability with respect to $\mathcal{X}$ (or $\mathcal{U}, \mathcal{V}$) if $\mathbb{P}_{\mathcal{X}}(\Omega^c)<N^{-D}$ (or $\mathbb{P}_{\mathcal{U}}(\Omega^c)<N^{-D}$, $\mathbb{P}_{\mathcal{V}}(\Omega^c)<N^{-D}$).

	For a sequence of random variables, the notation $\implies$ denotes the convergence in distribution as $M,N\to \infty$.

	For integers $a$ and $b$, the notation $[[a,b]]$ denotes $\{n\in \mathbb{Z}: a\le n\le b\}$.
\end{nrmk}
Note that we can replace $N$ with $M$ for these notations since we are assuming $M, N\to \infty$ with $M / N \to \ratio$.

We begin by introducing a lemma that we will use frequently to neglect rare events in $\mathcal{U}$ and $\mathcal{V}$.

\begin{lem}\label{lem:rare-event-cut}
	Suppose that $R_{M,N}$ is a non-negative random variable depending on $\U,\V$ and $\Noise$.
	If $\Omega_{M,N}\equiv\Omega_{M,N} (\U,\V)$ is an event in $\mathcal{U}\times \mathcal{V}$ with
	\[ \mathbb{E}_{\mathcal{U}\times \mathcal{V}}\left[\mathds{1}(\Omega_{M,N}^c)\mathbb{E}_{\mathcal{\Noise}}[R_{M,N}]\right]=o(1),\]
	then 
	\[ \mathbb{E}_{\mathcal{U}\times \mathcal{V}}[R_N]-\mathbb{E}_{\mathcal{U}\times \mathcal{V}}[\mathds{1}(\Omega_{M,N})R_{M,N}]=o_{\mathcal{\Noise}}(1) .\]
\end{lem}
\begin{proof}[Proof of Lemma \ref{lem:rare-event-cut}]
	Taking the partial expectation $\mathbb{E}_{\mathcal{\Noise}}$ to $\mathbb{E}_{\mathcal{U}\times \mathcal{V}}[R_{M,N}]-\mathbb{E}_{\mathcal{U}\times \mathcal{V}}[\mathds{1}(\Omega_{M,N})R_{M,N}]$,	
	\begin{align*}
		&\mathbb{E}_{\mathcal{\Noise}}\left[\mathbb{E}_{\mathcal{U}\times \mathcal{V}}[R_{M,N}]-\mathbb{E}_{\mathcal{U}\times \mathcal{V}}[\mathds{1}(\Omega_{M,N})R_{M,N}]\right] \\
		&= \mathbb{E}_\mathcal{\Noise}\mathbb{E}_{\mathcal{U}\times\mathcal{V}}[\mathds{1}(\Omega_{M,N}^c)R_{M,N}]=\mathbb{E}_{\mathcal{U}\times \mathcal{V}}[\mathds{1}(\Omega_{M,N}^c)\mathbb{E}_\mathcal{\Noise}[R_{M,N}]]=o(1).
	\end{align*}
	We get the desired lemma by applying Markov's inequality in $\mathcal{\Noise}$.
\end{proof}

We will frequently use a special case of this lemma when $\mathbb{E}_{\mathcal{\Noise}}[R_{M,N}]$ is bounded.
\begin{cor}\label{cor:rare-event-cut}
	Suppose that $R_{M,N}$ is a non-negative random variable depending on $\U,\V$ and $\Noise$.
	If $\Omega_{M,N}\equiv\Omega_{M,N} (\U,\V)$ is an event in $\mathcal{U}\times \mathcal{V}$ with
	$\mathbb{P}_{\mathcal{U}\times \mathcal{V}}(\Omega_{M,N}^c)=o(a_{M,N}^{-1})$ for some deterministic $a_{M,N}$ and $\mathbb{E}_{\mathcal{\Noise}}[R_{M,N}]=O(a_{M,N})$ almost surely with respect to $\mathcal{U}\times \mathcal{V}$, then
	\[ \mathbb{E}_{\mathcal{U}\times \mathcal{V}}[R_N]-\mathbb{E}_{\mathcal{U}\times \mathcal{V}}[\mathds{1}(\Omega_{M,N})R_{M,N}]=o_{\mathcal{\Noise}}(1) .\]
\end{cor}
\begin{proof}[Proof of Corollary \ref{cor:rare-event-cut}]
	Since $\mathbb{E}_\mathcal{\Noise}[R_{M,N}]=O(a_{M,N})$,
	\[ \mathbb{E}_{\mathcal{U}\times \mathcal{V}}\left[\mathds{1}(\Omega_{M,N}^c)\mathbb{E}_{\mathcal{\Noise}}[R_{M,N}]\right]\le Ca_{M,N}\mathbb{E}_{\mathcal{U}\times \mathcal{V}}[\mathds{1}(\Omega_{M,N}^c)]=Ca_{M,N}\mathbb{P}_{\mathcal{U}\times \mathcal{V}}(\Omega_{M,N}^c) .\]
	for some constant $C$. Now the right side is $o(1)$, we get the desired result by applying lemma \ref{lem:rare-event-cut}.
\end{proof}

\subsection{Proof of Theorem \ref{thm:lr-additive}}\label{subsec:proof-lr-additive}

We now prove our main result on the fluctuation of the log likelihood ratio.

\begin{proof}[Proof of Theorem \ref{thm:lr-additive}]
	Recall the definition of the likelihood ratio $\mathcal{L}(\snr, \epsilon_u, \epsilon_v)$ in Definition \ref{defn:lr-additive}.
	Note that the entries of $\Noise$ are independent.
	Since
	\begin{align*}
		\mathbb{E}_{\mathcal{\Noise}}\left[\frac{p(\sqrt{N}\noise-\sqrt{\snr N}u_\hidx v_\widx / (\|\U\|^{\epsilon_u}\|\V\|^{\epsilon_v}))}{p(\sqrt{N}\noise)}\right]
		&=\int_{\mathbb{R}}\frac{p(x-\sqrt{\snr N}u_\hidx v_\widx / (\|\U\|^{\epsilon_u}\|\V\|^{\epsilon_v}))}{p(x)}\;p(x)\mathrm{d}x \\
		&=\int_{\mathbb{R}}p(x-\sqrt{\snr N}u_\hidx v_\widx / (\|\U\|^{\epsilon_u}\|\V\|^{\epsilon_v}))\mathrm{d}x =1,
	\end{align*}
	we find that $\mathbb{E}_{\mathcal{\Noise}}[\mathcal{L}(\snr, \epsilon_u, \epsilon_v)]=1$, and hence we can apply Corollary \ref{cor:rare-event-cut} to neglect any events in $\mathcal{U}\times \mathcal{V}$ with probability $o(1)$.

	Let
	\begin{equation}
		\begin{split}
			\Omega_{\mathcal{U}\times \mathcal{V}}^\eta:=
			&\{\max_{1\le\hidx\le M}|u_\hidx|\le M^{-1 /2+\eta}\}\cap \{|\|\U\|-1|\le M^{-1 /5} \}\\
			&\{\max_{1\le \widx\le N}|v_\widx|\le N^{-1 /2+\eta}\}\cap \{|\|\V\|-1|\le N^{-1 /5} \}
		\end{split}
	\end{equation}
	and
	\begin{equation}
		\Omega^\eta_{\mathcal{\Noise}}:=\{\max_{\hidx,\widx}|\noise|\le N^{-1 /2+\eta}\}
	\end{equation}
	for some sufficiently small $\eta>0$, which will be explained later.
	From Markov's inequality, we find that $\Omega_{\mathcal{U}\times \mathcal{V}}^\eta$ holds with overwhelming probability.
	Also, $\Omega^{\eta}_{\mathcal{\Noise}}$ holds with overwhelming probability with respect to $\mathcal{\Noise}$.
	Expanding the product in the right side of Definition \ref{defn:lr-additive}, we get
	\begin{align*}
		&\mathds{1}(\Omega_{\mathcal{U}\times \mathcal{V}}^\eta)\prod_{\hidx,\widx}\frac{p(\sqrt{N}\noise-\sqrt{\snr N}u_\hidx v_\widx / (\|\U\|^{\epsilon_u}\|\V\|^{\epsilon_v}))}{p(\sqrt{N}\noise)} \\
		&=\mathds{1}(\Omega_{\mathcal{U}\times \mathcal{V}}^\eta)\mathds{1}(\Omega^\eta_{\mathcal{\Noise}})\prod_{\hidx,\widx}\frac{p(\sqrt{N}\noise-\sqrt{\snr N}u_\hidx v_\widx / (\|\U\|^{\epsilon_u}\|\V\|^{\epsilon_v}))}{p(\sqrt{N}\noise)}+o_{\mathcal{\Noise}}(1) \\
		&=\mathds{1}(\Omega_{\mathcal{U}\times \mathcal{V}}^\eta)\mathds{1}(\Omega^\eta_{\mathcal{\Noise}})\prod_{\hidx,\widx}\left( 1+p_{\hidx\widx}^{(1)}\frac{\sqrt{N}u_\hidx v_\widx}{\|\U\|^{\epsilon_u}\|\V\|^{\epsilon_v}}+p_{\hidx\widx}^{(2)}\frac{N u_\hidx^2 v_\widx^2}{\|\U\|^{2\epsilon_u}\|\V\|^{2\epsilon_v}}\right.\\ 
		&\left.\qquad\qquad\qquad\qquad\qquad\quad+p_{\hidx\widx}^{(3)}\frac{N^{3 / 2}u_\hidx^3 v_\widx^3}{\|\U\|^{3\epsilon_u}\|\V\|^{3\epsilon_v}}+ p_{\hidx\widx}^{(4)}\frac{N^2 u_\hidx^4 v_\widx^4}{\|\U\|^{4\epsilon_u}\|\V\|^{4\epsilon_v}} \right)+o_{\mathcal{\Noise}}(1)\\
		&=\mathds{1}(\Omega_{\mathcal{U}\times \mathcal{V}}^\eta)\prod_{\hidx,\widx}\left( 1+p_{\hidx\widx}^{(1)}\frac{\sqrt{N}u_\hidx v_\widx}{\|\U\|^{\epsilon_u}\|\V\|^{\epsilon_v}}+p_{\hidx\widx}^{(2)}\frac{N u_\hidx^2 v_\widx^2}{\|\U\|^{2\epsilon_u}\|\V\|^{2\epsilon_v}} \right.\\
		&\left.\qquad\qquad\qquad\qquad+p_{\hidx\widx}^{(3)}\frac{N^{3 / 2}u_\hidx^3 v_\widx^3}{\|\U\|^{3\epsilon_u}\|\V\|^{3\epsilon_v}}+ p_{\hidx\widx}^{(4)}\frac{N^2 u_\hidx^4 v_\widx^4}{\|\U\|^{4\epsilon_u}\|\V\|^{4\epsilon_v}} \right)+o_{\mathcal{\Noise}}(1)
	\end{align*}
	where we used
	\[ p_{\hidx\widx}^{(n)}=\frac{(-1)^n\snr^{n / 2}p^{(n)}(\sqrt{N}\noise)}{n!\,p(\sqrt{N}\noise)}\qquad (n=1,2,3,4) .\]
	Note that the product term above is nonnegative for sufficiently large $M,N$, since the difference with likelihood ratio is $o_\mathcal{\Noise}(1)$.
	We need to choose $\eta>0$ small that, assuming $\Omega_{\mathcal{U}\times \mathcal{V}}^\eta$ and $\Omega_{\mathcal{\Noise}}^\eta$,
	\[ \sum_{\hidx,\widx}|p^{(5)}(x)/p(x)|N^{5 /2}u_\hidx^5v_\widx^5 / (\|\U\|^{5\epsilon_u}\|\V\|^{5\epsilon_v})\le N^{-1 /10}\]
	for any $x\in [\sqrt{N}\noise-\sqrt{\snr N}u_\hidx v_\widx / (\|\U\|^{\epsilon_u}\|\V\|^{\epsilon_v}), \sqrt{N}\noise]$, uniformly in $\hidx, \widx$.
	(Note that this is possible due to the assumption that the function $p^{(5)}/p$ is polynomially bounded in Assumption \ref{assum:density}.)

	Since $(\Omega_{\mathcal{U}\times \mathcal{V}}^\eta)^c$ is the event in $\mathcal{U}\times \mathcal{V}$ with probability $o(1)$, We now find that 
	\begin{align*}
		&\mathcal{L}(\snr, \epsilon_u, \epsilon_v) \\
		&=\mathbb{E}_{\mathcal{U}\times \mathcal{V}}\left[\mathds{1}(\Omega_{\mathcal{U}\times \mathcal{V}}^\epsilon)\prod_{\hidx,\widx}\left( 1+p_{\hidx\widx}^{(1)}\frac{\sqrt{N}u_\hidx v_\widx}{\|\U\|^{\epsilon_u}\|\V\|^{\epsilon_v}}+p_{\hidx\widx}^{(2)}\frac{N u_\hidx^2 v_\widx^2}{\|\U\|^{2\epsilon_u}\|\V\|^{2\epsilon_v}}\right.\right. \\
		&\left.\left. \qquad\qquad\qquad\qquad\qquad+p_{\hidx\widx}^{(3)}\frac{N^{3 / 2}u_\hidx^3 v_\widx^3}{\|\U\|^{3\epsilon_u}\|\V\|^{3\epsilon_v}}+ p_{\hidx\widx}^{(4)}\frac{N^2 u_\hidx^4 v_\widx^4}{\|\U\|^{4\epsilon_u}\|\V\|^{4\epsilon_v}} \right)\right]+o_\mathcal{\Noise}(1).
	\end{align*}
	Since $p_{\hidx\widx}^{(n)}$ are odd functions if $n$ is odd, and even functions if $n$ is even, we find that we can apply Proposition \ref{prop:general} with $\epsilon = \eta$, if $\tau_1'=\tau_1=\snr F_p <\snr_c\sqrt{\ratio}$.
	The mean and the variance of the limiting Gaussian can be computed from that $\tau_2'=\tau_2 = \snr^2 G_p / 4$.
	This proves that Theorem \ref{thm:lr-additive} holds for $\mathcal{L}(\snr, \epsilon_u, \epsilon_v)$ if Condition \ref{cond:prior} holds.
\end{proof}

\subsection{Proof of Theorem \ref{thm:free-energy}}\label{subsec:proof-free-energy}

We next prove the result on the free energy of bipartite SK model.
First, we separate the deterministic term $M\Lambda(\beta^2/2)$ from $\log Z_{M,N}$ while maintaing an expansion to which Proposition \ref{prop:general} applies.
For the i.i.d. prior, the quadratic factor is absorbed into the product law by the change of measure in Definition \ref{defn:quadratic-tilt}.
Consequently, the relevant truncation events and moments must be controlled under the tilted law.
We next rescale the tilted coordinates using $\sigma^2=\Lambda'(\beta^2/2)$, and then identify the coefficients in the fourth-order expansion.
This reduction is the additional process needed beyond the normalized-spin free energy argument of \cite{Choo-Han-Lee2026}.

\begin{proof}[Proof of Theorem \ref{thm:free-energy}]
	Recall the definition of $\Lambda$ and $\tilde{\Lambda}$ from \eqref{eq:Lambda} and \eqref{eq:Lambda-iid}, respectively.
	Note that both $\Lambda$ and $\tilde{\Lambda}$ are $M,N$-independent function.
	Define
	\begin{equation}
		\begin{split}
			{\tilde Z}&:=\exp\left( -M\Lambda\left( \frac{\beta^2}{2} \right)+ \beta\sqrt{MN}\sum_{\hidx,\widx}\noise\frac{u_\hidx v_\widx}{\|\U\|^{\epsilon_u}\|\V\|} \right)
		\end{split}
	\end{equation}
	From the strictly sub-Gaussianity of the disorder $\sqrt{N}\noise$,
	\[ \mathbb{E}_{\mathcal{X}}\left[\exp\left( \beta\sqrt{MN}\noise \frac{u_\hidx v_\widx}{\|\U\|^{\epsilon_u}\|\V\|} \right) \right]\le \exp\left(\frac{\beta^2M}{2}\frac{u_\hidx^2 v_\widx^2}{\|\U\|^{2\epsilon_u}\|\V\|^2} \right), \]
	and from the independence of $\noise$'s, we have
	\begin{equation}\label{eq:disorder-expectation}
		\begin{aligned}[b]
			&\mathbb{E}_{\mathcal{X}}\left[\exp\left( \beta\sqrt{MN}\sum_{\hidx,\widx}\noise \frac{u_\hidx v_\widx}{\|\U\|^{\epsilon_u}\|\V\|} \right) \right] \\
			&\le \exp\left( \frac{\beta^2M}{2}\left(\sum_\hidx \frac{u_\hidx^2}{\|\U\|^{2\epsilon_u}}\right)\left( \sum_\widx \frac{v_\widx^2}{\|\V\|^2} \right)  \right) 
			=\exp\left( \frac{\beta^2M}{2}\|\U\|^{2-2\epsilon_u}\right).
		\end{aligned}
	\end{equation}

	For the first case where the prior is normalized($\epsilon_u=1$), the right side of \eqref{eq:disorder-expectation} is exactly $\exp(M\Lambda(\beta^2/ 2))$.
	Therefore, $\mathbb{E}_{\mathcal{X}}[\tilde{Z}]=O(1)$ and we can apply Corollary \ref{cor:rare-event-cut} to cut $o(1)$ events on $\mathcal{U}\times \mathcal{V}$.

	For the second case where the prior is i.i.d.($\epsilon_u=0$), let
	\[ \Omega_\mathcal{U}^{\epsilon}=\left\{ \max_{1\le \hidx\le M}|u_\hidx|\le M^{-1 /2+\epsilon} \right\}  .\]
	Now since 
	\[ \mathbb{E}_{\mathcal{U}}\left[\exp\left(\frac{\beta^2 M}{2} \|\U\|^2\right)\right]= \exp\left(M\Lambda\left( \frac{\beta^2 }{2} \right)\right), \]
	the tilted prior $u_\hidx^{M\beta^2 /2}$ has the law
	\[ \mu^{M\beta^2 /2}(\mathrm{d}x) := \exp\left(-\tilde{\Lambda}\left( \frac{\beta^2}{2} \right)+\frac{\beta^2M}{2}x^2 \right)\mu(\mathrm{d}x),\]
	where $\mu$ is the law of $u_\hidx$.
	Note that $u_\hidx^{M\beta^2 /2}$ is still i.i.d. and symmetric with this change of measure.
	Combining the tilted priors on every coordinates, we have total tilted prior $\U^{M\beta^2 /2}$ as 
	\begin{equation}\label{eq:tilted-prior}
		\mathbb{\mu}^{M\beta^2 /2}(\mathrm{d}{\boldsymbol x}):=\exp\left(-M\Lambda\left( \frac{\beta^2}{2} \right)+\frac{\beta^2M}{2}\|{\boldsymbol x}\|^2 \right)\mathbb{\mu}(\mathrm{d}{\boldsymbol x}),
	\end{equation}
	since $\tilde{\Lambda} = \Lambda$ for the i.i.d. case.
	Let $\mathcal{U}^{M\beta^2 /2}$ be the distribution of the tilted prior.
	We get
	\[\mathbb{E}_{\mathcal{U}}\left[\mathds{1}((\Omega_{\mathcal{U}}^\epsilon)^c)\mathbb{E}_{\mathcal{X}}[\tilde Z]\right] 
	\le \mathbb{E}_{\mathcal{U}}\left[\mathds{1}((\Omega_{\mathcal{U}}^\epsilon)^c)\exp\left(-M\Lambda\left( \frac{\beta^2}{2} \right)+\frac{\beta^2M}{2}\|\U\|^2 \right) \right]
	=\mathbb{P}_{\mathcal{U}^{M\beta^2 /2}}((\Omega_{\mathcal{U}}^\epsilon)^c). \]
	Now by union bound,
	\[ \mathbb{P}_{\mathcal{U}^{M\beta^2 /2}}((\Omega_{\mathcal{U}}^\epsilon)^c) \le M\mathbb{P}(|u_\hidx^{M\beta^2 /2}|> M^{-1/ 2+\epsilon}).\]
	To bound the tail probability, we use exponential Markov inequality on the tilted prior
	\begin{align*} 
		&\mathbb{P}(|u_\hidx^{M\beta^2 /2}|>M^{-1 /2+\epsilon})  \\
		&\le \mathbb{E}\left[e^{\rho M\left(u_\hidx^{M\beta^2 /2}\right)^2}\right]e^{-\rho M^{2\epsilon}}=\exp\left(\tilde{\Lambda}\left( \frac{\beta^2 }{2}+\rho \right)-\tilde{\Lambda} \left( \frac{\beta^2}{2} \right) \right)e^{-\rho M^{2\epsilon}}.
	\end{align*}
	where we choose fixed $\rho>0$ small such that $\tilde{\Lambda} (\rho + \beta^2 / 2)<\infty$.
	Therefore, the tail probability of the tilted prior is $o(1)$ for any $\epsilon >0$ and we get
	\[ \mathbb{E}_{\mathcal{U}\times\mathcal{V}}\left[\mathds{1}((\Omega_\mathcal{U}^\epsilon)^c)\mathbb{E}_{\mathcal{X}}[\tilde Z]\right]=o(1).\]
	Applying Lemma \ref{lem:rare-event-cut},
	\[ \mathbb{E}_{\mathcal{U}\times \mathcal{V}}[\tilde Z]=\mathbb{E}_{\mathcal{U}\times \mathcal{V}}[\mathds{1}(\Omega_{\mathcal{U}}^\epsilon)\tilde Z ]+o_{\mathcal{X}}(1) .\]
	Further, for any integer $s>2$ and $x\in \mathbb{R}$, since $|x|^s \le C_{s,\rho}e^{\rho x^2}$ for some constant $C_{s,\rho}>0$, we have
	\begin{align*} 
		\mathbb{E}[|\sqrt{M}u_\hidx^{M\beta^2 /2}|^s]
		&=\exp\left( -\tilde{\Lambda}\left( \frac{\beta^2}{2} \right) \right)\mathbb{E}\left[|\sqrt{M}u_\hidx|^s\exp\left( \frac{\beta^2 M}{2}u_\hidx^2 \right) \right] \\
		&\le C_{s,\rho}\exp\left(\tilde{\Lambda}\left( \frac{\beta^2}{2}+\rho \right)-\tilde{\Lambda}\left( \frac{\beta^2}{2} \right) \right)=O(1).
	\end{align*}
	Therefore the tilted prior $\sqrt{M}u_\hidx^{M\beta^2 /2}$ has finite every moments.
	Using Markov inequality on the tilted prior, we can concentrate further moments.
	First, we calculate the variance of $u_\hidx^{M\beta^2 /2}$.
	For the i.i.d. prior, since
	\[ \Lambda'(t) =\tilde{\Lambda}'(t)=(\log\mathbb{E}[e^{tMu_\hidx^2}])'=M\frac{\mathbb{E}[u_\hidx^2 e^{tMu_\hidx^2}]}{\mathbb{E}[e^{tMu_\hidx^2}]}=M\exp\left(-\tilde{\Lambda}(t)\right)\mathbb{E}[u_\hidx^2 e^{tMu_\hidx^2}],\]
	we have
	\[ \Lambda'\left(\frac{\beta^2}{2}\right)=M\mathbb{E}[(u_\hidx^{M\beta^2 /2})^2].\]
	Since we only tilt i.i.d. prior and $\Lambda'(t)=1$ for normalized prior, we define
	\[ \sigma^2 := \Lambda'\left( \frac{\beta^2}{2} \right)=\mathrm{Var}\left(\sqrt{M}u_1^{(1-\epsilon_u)M\beta^2 /2}\right).\]
	Now define
	\begin{align*}
		\Omega_\mathcal{U}':=
			&\left\{\max_{1\le \hidx\le M}|u_\hidx|\le \sigma M^{-1 /2+\epsilon}\right\}\cap\left\{\left|\|\U\|-\sigma\right|\le M^{-1 /5}\right\} \\
			&\cap \left\{\left|M\sum_{\hidx}u_\hidx^4-M^2 \mathbb{E}\left[\left(u_1^{(1-\epsilon_u)M\beta^2 /2}\right)^4\right]\right|\le M^{-1 /5}\right\} \\
			&\cap\left\{\left|M^3\sum_{\hidx}u_\hidx^8-M^4\mathbb{E}\left[\left(u_1^{(1-\epsilon_u)M\beta^2 /2}\right)^8\right]\right|\le M^{-1 /5}\right\}.
	\end{align*}
	Then, we have
	\[ \mathbb{E}_{\mathcal{U}\times \mathcal{V}}\left[\mathds{1}((\Omega_{\mathcal{U}}')^c)\mathbb{E}_{\mathcal{X}}[\tilde{Z}]\right]=o(1).\]

	Note that for any event $\Omega_\mathcal{V}$ in $\mathcal{V}$ which holds with high probability with respect to $\mathcal{V}$,
	\begin{align*}
		&\mathbb{E}_{\mathcal{U}\times \mathcal{V}}\left[\mathds{1}(\Omega_\mathcal{V}^c)\mathbb{E}_\mathcal{X}[\tilde Z]\right]
		=\mathbb{E}_{\mathcal{V}}[\mathds{1}(\Omega_\mathcal{V}^c)\mathbb{E}_{\mathcal{U}}[\mathbb{E}_\mathcal{X}[\tilde Z]]] \\
		&\le \mathbb{E}_\mathcal{V}\left[\mathds{1}(\Omega_\mathcal{V}^c)\mathbb{E}_\mathcal{U}\left[\exp\left(-M\Lambda\left( \frac{\beta^2}{2} \right)+\frac{\beta^2M}{2}\|\U\|^{2(1-\epsilon_u)} \right) \right]\right] = \mathbb{P}_\mathcal{V}(\Omega_\mathcal{V}^c)=o(1).
	\end{align*}
	Therefore, by Lemma \ref{lem:rare-event-cut},
	\[ \mathbb{E}_{\mathcal{U}\times \mathcal{V}}[\tilde Z]=\mathbb{E}_{\mathcal{U}\times \mathcal{V}}[\mathds{1}(\Omega_{\mathcal{V}})\tilde Z ]+o_{\mathcal{X}}(1) .\]

	We define
	\begin{equation}
		\begin{split}
			\Omega'_\mathcal{V}
			&:= \left\{\max_{1\le \widx\le N}|v_\widx|\le N^{-1 /2+\epsilon}\right\}\cap\left\{|\|\V\|-1|\le N^{-1 / 5}\right\}\cap \left\{\left|N\sum_{\widx}v_\widx^4 - N^2\mathbb{E}[v_1^4]\right|\le N^{-1/5}\right\}\\
			&\quad \cap \left\{\left|N^3\sum_{\widx}v_\widx^8 - N^4\mathbb{E}[v_1^8]\right|\le N^{-1 /5}\right\}.
		\end{split}
	\end{equation}
	Since $\Omega'_\mathcal{V}$ holds with overwhelming probability with respect to $\mathcal{V}$, we get
	\[ \exp\left( -M\Lambda\left( \frac{\beta^2 }{2} \right) \right)Z_{M,N}=\mathbb{E}_{\mathcal{U}\times \mathcal{V}}[\tilde Z]=\mathbb{E}_{\mathcal{U}\times \mathcal{V}}[\mathds{1}(\Omega'_\mathcal{U})\mathds{1}(\Omega'_\mathcal{V})\tilde Z]+o_{\mathcal{X}}(1), \]
	for both normalized and i.i.d. prior.
	Next, we separate the partial expectation
	\begin{align*}
		&\mathbb{E}_{\mathcal{U}\times \mathcal{V}}[\mathds{1}(\Omega'_\mathcal{U})\mathds{1}(\Omega_\mathcal{V}')\tilde{Z}] \\
		&= \mathbb{E}_{\mathcal{V}}\left[\mathds{1}(\Omega_\mathcal{V}')\mathbb{E}_{\mathcal{U}}\left[\mathds{1}(\Omega'_\mathcal{U})
		\exp\left( -M\Lambda\left( \frac{\beta^2 }{2} \right)+ \beta\sqrt{MN}\sum_{\hidx,\widx}\noise \frac{u_\hidx v_\widx}{\|\U\|^{\epsilon_u}\|\V\|} \right)\right]\right].
	\end{align*}
	Then, for i.i.d. prior,i.e. $\epsilon_u=0$, we get
	\begin{equation}\label{eq:full-expectation}
		\begin{split}
			&\mathbb{E}_{\mathcal{U}}\left[\mathds{1}(\Omega'_\mathcal{U})
			\exp\left( -M\Lambda\left( \frac{\beta^2 }{2} \right)+ \beta\sqrt{MN}\sum_{\hidx,\widx}\noise \frac{u_\hidx v_\widx}{\|\U\|^{\epsilon_u}\|\V\|} \right)\right] \\
			&=\mathbb{E}_{\mathcal{U}^{M\beta^2 /2}}\left[\mathds{1}(\Omega'_\mathcal{U})
			\exp\left(\beta\sqrt{MN}\sum_{\hidx,\widx}\noise \frac{u_\hidx^{M\beta^2 /2} v_\widx}{\|\V\|}-\frac{\beta^2M}{2}\sum_{\hidx,\widx}\frac{\left(u_\hidx^{M\beta^2 /2}\right)^2 v_\widx^2}{\|\V\|^2} \right)\right].
		\end{split}
	\end{equation}
	where we changed to the tilted prior.
	To apply Proposition, we need to normalize $u_\hidx^{M\beta^2 /2}$ to make variance $1$.
	Since $\mathrm{Var}\left(\sqrt{M}u_\hidx^{M\beta^2 /2}\right)=\sigma^2$, we get
	\begin{equation}\label{eq:renormalized}
		\begin{split}
			&\mathbb{E}_{\mathcal{U}\times \mathcal{V}}[\mathds{1}(\Omega'_\mathcal{U})\mathds{1}(\Omega_\mathcal{V}')\tilde{Z}] \\
			&=\mathbb{E}_{\tilde{\mathcal{U}} \times \mathcal{V}}\left[\mathds{1}(\Omega'_\mathcal{U})\mathds{1}(\Omega_\mathcal{V}')
			\exp\left(\beta\sigma\sqrt{MN}\sum_{\hidx,\widx}\noise \frac{\tilde{u}_\hidx v_\widx}{\|\tilde{\U}\|^{\epsilon_u}\|\V\|}-\frac{\beta^2\sigma^2M}{2}\sum_{\hidx,\widx} \frac{\tilde{u}_\hidx^2 v_\widx^2}{\|\tilde{\U}\|^{2\epsilon_u}\|\V\|^2} \right)\right]
		\end{split}
	\end{equation}
	where we let $\tilde{u}_\hidx = u_\hidx^{(1-\epsilon_u)M\beta^2 /2} / \sigma$ and $\tilde{\mathcal{U}}$ as the distribution of $\tilde{u}_\hidx$.
	Note that $\Omega'_{\mathcal{U}}\cap\Omega'_{\mathcal{V}}\subset \Omega_{\tilde{\mathcal{U}}\times \mathcal{V}}^\epsilon$, which is a $\tilde{\U}$ version of $\Omega_{\mathcal{U}\times \mathcal{V}}^\epsilon$ that holds with high probability with respect to $\tilde{\mathcal{U}}$.
	Therefore, we get
	\[ \exp\left( -M\Lambda\left( \frac{\beta^2 }{2} \right) \right)Z_{M,N}=\mathbb{E}_{\mathcal{U}\times \mathcal{V}}[\tilde Z]=\mathbb{E}_{\mathcal{U}\times \mathcal{V}}[\mathds{1}(\Omega'_\mathcal{U})\mathds{1}(\Omega'_\mathcal{V})\mathds{1}(\Omega_{\tilde{\mathcal{U}}\times \mathcal{V}}^\epsilon)\tilde Z]+o_{\mathcal{X}}(1).\]

	We next expand the sum in the exponent in \eqref{eq:renormalized}.
	Define $\Omega=\Omega'_\mathcal{U} \cap \Omega_\mathcal{V}'$ and $\Omega_\mathcal{\Noise}^\epsilon=\{\max_{\hidx,\widx}|\noise|\le N^{-1 /2+\epsilon}\}$.
	Following the argument in the proof of Theorem \ref{thm:lr-additive},
	\begin{align*}
		&\mathds{1}(\Omega)\mathds{1}(\Omega_{\tilde{\mathcal{U}}\times \mathcal{V}}^\epsilon)\exp\left(\beta\sigma\sqrt{MN}\sum_{\hidx,\widx}\noise \frac{\tilde{u}_\hidx v_\widx}{\|\tilde{\U}\|^{\epsilon_u}\|\V\|} -\frac{\beta^2\sigma^2M}{2}\sum_{\hidx,\widx}\frac{\tilde{u}_\hidx^2 v_\widx^2}{\|\tilde{\U}\|^{2\epsilon_u}\|\V\|^2}\right) \\
		&=\mathds{1}(\Omega)\mathds{1}(\Omega_{\tilde{\mathcal{U}}\times \mathcal{V}}^\epsilon)\mathds{1}(\Omega_\mathcal{\Noise}^\epsilon)\exp\left(\beta\sigma\sqrt{MN}\sum_{\hidx,\widx}\noise \frac{\tilde{u}_\hidx v_\widx}{\|\tilde{\U}\|^{\epsilon_u}\|\V\|} -\frac{\beta^2\sigma^2M}{2}\sum_{\hidx,\widx}\frac{\tilde{u}_\hidx^2 v_\widx^2}{\|\tilde{\U}\|^{2\epsilon_u}\|\V\|^2}\right)+o_\mathcal{\Noise}(1) \\
		&=\prod_{\hidx,\widx}\left( 1+P_{\hidx\widx}^{(1)}\frac{\sqrt{N}\tilde{u}_\hidx v_\widx}{\|\tilde{\U}\|^{\epsilon_u}\|\V\|}+P_{\hidx\widx}^{(2)}\frac{N\tilde{u}_\hidx^2 v_\widx^2}{\|\tilde{\U}\|^{2\epsilon_u}\|\V\|^2}+P_{\hidx\widx}^{(3)}\frac{N^{3 / 2}\tilde{u}_\hidx^3 v_\widx^3}{\|\tilde{\U}\|^{3\epsilon_u}\|\V\|^3}+\tilde{P}_{\hidx\widx}^{(4)}\frac{N^2\tilde{u}_\hidx^4 v_\widx^4}{\|\tilde{\U}\|^{4\epsilon_u}\|\V\|^4} \right) \\
		&\quad\times\mathds{1}(\Omega)\mathds{1}(\Omega_{\tilde{\mathcal{U}}\times \mathcal{V}}^\epsilon)\mathds{1}(\Omega_\mathcal{\Noise}^\epsilon) +o_\mathcal{\Noise}(1)\\
		&=\prod_{\hidx,\widx}\left( 1+P_{\hidx\widx}^{(1)}\frac{\sqrt{N}\tilde{u}_\hidx v_\widx}{\|\tilde{\U}\|^{\epsilon_u}\|\V\|}+P_{\hidx\widx}^{(2)}\frac{N\tilde{u}_\hidx^2 v_\widx^2}{\|\tilde{\U}\|^{2\epsilon_u}\|\V\|^2}+P_{\hidx\widx}^{(3)}\frac{N^{3 / 2}\tilde{u}_\hidx^3 v_\widx^3}{\|\tilde{\U}\|^{3\epsilon_u}\|\V\|^3}+\tilde{P}_{\hidx\widx}^{(4)}\frac{N^2\tilde{u}_\hidx^4 v_\widx^4}{\|\tilde{\U}\|^{4\epsilon_u}\|\V\|^4} \right)\\
		&\quad\times\mathds{1}(\Omega)\mathds{1}(\Omega_{\tilde{\mathcal{U}}\times \mathcal{V}}^\epsilon)+o_\mathcal{\Noise}(1)\\
	\end{align*}
	with
	\begin{equation}\label{eq:free-energy-P-def}
		\begin{split}
			&P_{\hidx\widx}^{(1)}=\beta\sigma\sqrt{M}\noise, \quad P_{\hidx\widx}^{(2)}=\frac{\beta^2\sigma^2M}{2N}\left((\sqrt{N}\noise)^2-1\right),\\
			&P_{\hidx\widx}^{(3)}=\frac{\beta^3\sigma^3}{6}\left( \frac{M}{N} \right)^{\frac{3}{2}}\left((\sqrt{N}\noise)^3-3(\sqrt{N}\noise)\right), \\
			&{\tilde P}_{\hidx\widx}^{(4)}=\frac{\beta^4\sigma^4M^2}{24N^2}\left( (\sqrt{N}\noise)^4-6(\sqrt{N}\noise)^2+3 \right).
		\end{split}
	\end{equation}
	We remark that the product term is nonnegative on $\Omega_{\mathcal{U}\times \mathcal{V}}^\epsilon\cap\Omega_{\mathcal{\Noise}}^\epsilon$.
	Note that $\mathbb{E}[P_{\hidx\widx}^{(1)}]=\mathbb{E}[P_{\hidx\widx}^{(2)}]=\mathbb{E}[P_{\hidx\widx}^{(3)}]=0$ and $\mathbb{E}[\tilde P_{\hidx\widx}^{(4)}]=\beta^4\sigma^4M^2 (x_4-3)/ (24N^2)$.
	To make each term in the product centered, we use the relation
	\begin{align*}
		&\mathds{1}(\Omega)\mathds{1}(\Omega_{\tilde{\mathcal{U}}\times \mathcal{V}}^\epsilon)\prod_{\hidx,\widx}\left( 1+P_{\hidx\widx}^{(1)}\frac{\sqrt{N}\tilde{u}_\hidx v_\widx}{\|\tilde{\U}\|^{\epsilon_u}\|\V\|}+P_{\hidx\widx}^{(2)}\frac{N\tilde{u}_\hidx^2 v_\widx^2}{\|\tilde{\U}\|^{2\epsilon_u}\|\V\|^2}+P_{\hidx\widx}^{(3)}\frac{N^{3 / 2}\tilde{u}_\hidx^3 v_\widx^3}{\|\tilde{\U}\|^{3\epsilon_u}\|\V\|^3}+\tilde{P}_{\hidx\widx}^{(4)}\frac{N^2\tilde{u}_\hidx^4 v_\widx^4}{\|\tilde{\U}\|^{4\epsilon_u}\|\V\|^4} \right) \\
		&=\mathds{1}(\Omega)\mathds{1}(\Omega_{\tilde{\mathcal{U}}\times \mathcal{V}}^\epsilon)\prod_{\hidx,\widx}\left( 1+P_{\hidx\widx}^{(1)}\frac{\sqrt{N}\tilde{u}_\hidx v_\widx}{\|\tilde{\U}\|^{\epsilon_u}\|\V\|}+P_{\hidx\widx}^{(2)}\frac{N\tilde{u}_\hidx^2 v_\widx^2}{\|\tilde{\U}\|^{2\epsilon_u}\|\V\|^2}+P_{\hidx\widx}^{(3)}\frac{N^{3 / 2}\tilde{u}_\hidx^3 v_\widx^3}{\|\tilde{\U}\|^{3\epsilon_u}\|\V\|^3}+P_{\hidx\widx}^{(4)}\frac{N^2\tilde{u}_\hidx^4 v_\widx^4}{\|\tilde{\U}\|^{4\epsilon_u}\|\V\|^4} \right) \\
		&\quad\times \left( 1+\frac{\beta^4\sigma^4M^2(x_4-3)}{24N^2} \frac{N^2\tilde{u}_\hidx^4 v_\widx^4}{\|\tilde{\U}\|^{4\epsilon_u}\|\V\|^4} \right)+ o_\mathcal{\Noise}(1)
	\end{align*}
	which replaces ${\tilde P}_{\hidx\widx}^{(4)}$ in \eqref{eq:free-energy-P-def} by
	\[ P_{\hidx\widx}^{(4)}=\frac{\beta^4\sigma^4M^2}{24N^2}\left((\sqrt{N}\noise)^4-6(\sqrt{N}\noise)^2+6-x_4\right) .\]
	We can easily check the orthogonality conditions of $P_{\hidx\widx}^{(n)}$.
	Note that the additional factor can be approximated by
	\begin{align*}
		&\mathds{1}(\Omega)\mathds{1}(\Omega_{\tilde{\mathcal{U}}\times \mathcal{V}}^\epsilon)\prod_{\hidx,\widx}\left( 1+\frac{\beta^4\sigma^4M^2(x_4-3)}{24N^2} \frac{N^2\tilde{u}_\hidx^4 v_\widx^4}{\|\tilde{\U}\|^{4\epsilon_u}\|\V\|^4} \right) \\
		&=(1+O(N^{-1}))\mathds{1}(\Omega)\mathds{1}(\Omega_{\tilde{\mathcal{U}}\times \mathcal{V}}^\epsilon)\prod_{\hidx,\widx}\exp\left( \frac{\beta^4\sigma^4M^2(x_4-3)}{24}\frac{\tilde{u}_\hidx^4 v_\widx^4}{\|\tilde{\U}\|^{4\epsilon_u}\|\V\|^4} \right) \\
		&=(1+O(N^{-1}))\mathds{1}(\Omega)\mathds{1}(\Omega_{\tilde{\mathcal{U}}\times \mathcal{V}}^\epsilon)\exp\left( \frac{\beta^4\sigma^4M^2(x_4-3)}{24}\sum_{\hidx, \widx}\frac{\tilde{u}_\hidx^4 v_\widx^4}{\|\tilde{\U}\|^{4\epsilon_u}\|\V\|^4} \right) \\
		&=(1+O(N^{-1 / 5}))\mathds{1}(\Omega)\mathds{1}(\Omega_{\tilde{\mathcal{U}}\times \mathcal{V}}^\epsilon)\exp\left( \frac{\beta^4\sigma^4M(x_4-3)}{24N}m'_4 n_4 \right).
	\end{align*}
	Here, $n_4=N^2\mathbb{E}[v_\widx^4]$ and
	\[ m'_4:=M^2\mathbb{E}[\tilde{u}_\hidx^4]= \frac{M^2}{\sigma^4}\mathbb{E}\left[\left(u_\hidx^{(1-\epsilon_u)M\beta^2 /2}\right)^4\right]=\begin{dcases} \Lambda''\left( \frac{\beta^2}{2} \right)\sigma^{-4}+1 & (\epsilon_u=0) \\ m_4 & (\epsilon_u=1). \end{dcases}\]

	So far, we have seen that
	\begin{align*}
		&\exp\left( -M\Lambda\left( \frac{\beta^2 }{2} \right) \right)Z_{M,N}=\mathbb{E}_{\mathcal{U}\times \mathcal{V}}[\mathds{1}(\Omega)\mathds{1}(\Omega_{\tilde{\mathcal{U}}\times \mathcal{V}}^\epsilon)\tilde Z]+o_{\mathcal{X}}(1)  \\
		&=o_\mathcal{\Noise}(1)+\mathbb{E}_{\tilde{\mathcal{U}}\times \mathcal{V}}\left[\mathds{1}(\Omega)\mathds{1}(\Omega_{\tilde{\mathcal{U}}\times \mathcal{V}}^\epsilon)(1+O(N^{-1 / 5}))\exp\left( \frac{\beta^4\sigma^4M(x_4-3)}{24N}m'_4n_4  \right) \right.  \\
		&\quad\left. \times \prod_{\hidx,\widx}\left( 1+P_{\hidx\widx}^{(1)}\frac{\sqrt{N}\tilde{u}_\hidx v_\widx}{\|\tilde{\U}\|^{\epsilon_u}\|\V\|}+P_{\hidx\widx}^{(2)}\frac{N\tilde{u}_\hidx^2 v_\widx^2}{\|\tilde{\U}\|^{2\epsilon_u}\|\V\|^2}+P_{\hidx\widx}^{(3)}\frac{N^{3 / 2}\tilde{u}_\hidx^3 v_\widx^3}{\|\tilde{\U}\|^{3\epsilon_u}\|\V\|^3}+P_{\hidx\widx}^{(4)}\frac{N^2\tilde{u}_\hidx^4 v_\widx^4}{\|\tilde{\U}\|^{4\epsilon_u}\|\V\|^4} \right)\right] \\
		&=o_\mathcal{\Noise}(1)+(1+O(N^{-1 / 5}))\exp\left( \frac{\beta^4\sigma^4M(x_4-3)}{24N}m'_4 n_4  \right) \\
		&\times \mathbb{E}_{\tilde{\mathcal{U}}\times \mathcal{V}}\left[\mathds{1}(\Omega)\mathds{1}(\Omega_{\tilde{\mathcal{U}}\times \mathcal{V}}^\epsilon)\prod_{\hidx,\widx}\left( 1+P_{\hidx\widx}^{(1)}\frac{\sqrt{N}\tilde{u}_\hidx v_\widx}{\|\tilde{\U}\|^{\epsilon_u}\|\V\|}+P_{\hidx\widx}^{(2)}\frac{N\tilde{u}_\hidx^2 v_\widx^2}{\|\tilde{\U}\|^{2\epsilon_u}\|\V\|^2}+P_{\hidx\widx}^{(3)}\frac{N^{3 / 2}\tilde{u}_\hidx^3 v_\widx^3}{\|\tilde{\U}\|^{3\epsilon_u}\|\V\|^3}+P_{\hidx\widx}^{(4)}\frac{N^2\tilde{u}_\hidx^4 v_\widx^4}{\|\tilde{\U}\|^{4\epsilon_u}\|\V\|^4} \right)\right]
	\end{align*}
	which shows that
	\begin{align*}
		&(1+O(N^{-1 / 5}))\exp\left( -M\Lambda\left( \frac{\beta^2 }{2} \right)-\frac{\beta^4\sigma^4M(x_4-3)}{24N}m'_4 n_4 \right)Z_{M,N}\\
		&=\mathbb{E}_{\tilde{\mathcal{U}}\times \mathcal{V}}\left[\mathds{1}(\Omega)\mathds{1}(\Omega_{\tilde{\mathcal{U}}\times \mathcal{V}}^\epsilon)\prod_{\hidx,\widx}\left( 1+P_{\hidx\widx}^{(1)}\frac{\sqrt{N}\tilde{u}_\hidx v_\widx}{\|\tilde{\U}\|^{\epsilon_u}\|\V\|}+P_{\hidx\widx}^{(2)}\frac{N\tilde{u}_\hidx^2 v_\widx^2}{\|\tilde{\U}\|^{2\epsilon_u}\|\V\|^2}+P_{\hidx\widx}^{(3)}\frac{N^{3 / 2}\tilde{u}_\hidx^3 v_\widx^3}{\|\tilde{\U}\|^{3\epsilon_u}\|\V\|^3}+P_{\hidx\widx}^{(4)}\frac{N^2\tilde{u}_\hidx^4 v_\widx^4}{\|\tilde{\U}\|^{4\epsilon_u}\|\V\|^4} \right)\right] \\
		&\quad+o_\mathcal{\Noise}(1)\\
		&=\mathbb{E}_{\tilde{\mathcal{U}}\times \mathcal{V}}\left[\mathds{1}(\Omega_{\tilde{\mathcal{U}}\times \mathcal{V}}^\epsilon)\prod_{\hidx,\widx}\left( 1+P_{\hidx\widx}^{(1)}\frac{\sqrt{N}\tilde{u}_\hidx v_\widx}{\|\tilde{\U}\|^{\epsilon_u}\|\V\|}+P_{\hidx\widx}^{(2)}\frac{N\tilde{u}_\hidx^2 v_\widx^2}{\|\tilde{\U}\|^{2\epsilon_u}\|\V\|^2}+P_{\hidx\widx}^{(3)}\frac{N^{3 / 2}\tilde{u}_\hidx^3 v_\widx^3}{\|\tilde{\U}\|^{3\epsilon_u}\|\V\|^3}+P_{\hidx\widx}^{(4)}\frac{N^2\tilde{u}_\hidx^4 v_\widx^4}{\|\tilde{\U}\|^{4\epsilon_u}\|\V\|^4} \right)\right] \\
		&\quad+o_\mathcal{\Noise}(1),
	\end{align*}
	where the last equality comes from the nonnegativity of the product term on $\Omega_{\tilde{\mathcal{U}}\times \mathcal{V}}^\epsilon\cap\Omega_{\mathcal{\Noise}}^\epsilon$.
	Since we only tilted $\U$ coordinatewisely, $\tilde{\U}$ and $\V$ are still independent and we checked that $\tilde{\U}$ satisfies the conditions of Proposition \ref{prop:general}.
	As in the proof of Theorem \ref{thm:lr-additive}, we apply Proposition \ref{prop:general} to find
	\[ \left( -M\Lambda\left( \frac{\beta^2 }{2}\right) -\frac{\beta^4\sigma^4M(x_4-3)}{24N}m'_4 n_4 \right)+\log Z_{M,N}\implies \mathcal{N}(-\rho, 2\rho),\]
	which proves that
	\[N\left(F_{M,N}(\beta)- \frac{M}{N}\Lambda\left( \frac{\beta^2}{2}\right)\right)\]
	converges in distribution to a Gaussian.

	For the limiting free energy, 
	\[  \frac{M}{N}\Lambda\left( \frac{\beta^2}{2} \right) \to \ratio\Lambda\left( \frac{\beta^2}{2} \right). \]

	It remains to find the mean and the variance of the limiting distribution.
	Since
	\begin{align*}
		\tau'_1 &=\frac{M}{N}\beta^2\sigma^2\to \ratio\beta^2\sigma^2=:\tau_1  \\
		\tau'_2 &=\left(\frac{M}{N}\right)^2\frac{\beta^4\sigma^4}{4}(x_4-1) \to \frac{\ratio^2\beta^4\sigma^4}{4}(x_4-1)=:\tau_2,
	\end{align*}
	it can be readily checked that $\rho$ in Proposition is given by
	\[ \rho=-\frac{1}{4}\left( \log\left( 1-\ratio\beta^4\sigma^4 \right)-\frac{\ratio\beta^4\sigma^4}{2}(x_4-3) \right),  \]
	and the variance $V_F=2\rho$.
	Finally, the mean $m_F$ is given by
	\begin{align*}
		m_F&=-\rho+\frac{\ratio\beta^4\sigma^4}{24}(x_4-3)m'_4n_4 \\
			 &=\frac{1}{4}\log\left( 1-\ratio\beta^4\sigma^4 \right)+\frac{\ratio\beta^4\sigma^4}{24}(x_4-3)(m'_4n_4-3).
	\end{align*}
	Note that the condition of Proposition \ref{prop:general}, $\tau_1<\snr_c\sqrt{\ratio}$ is equivalent to $\sqrt{\ratio}\beta^2 \sigma^2 < \snr_c $.
	This proves Theorem.
\end{proof}

\section{Bipartite multigraph expansion}\label{sec:multigraph}

Sections \ref{sec:multigraph}--\ref{sec:technical} adapt the multigraph framework of Sections 4--6 in Choo, Han, and Lee \cite{Choo-Han-Lee2026} to the bipartite setting.
We maintain the sequence of reductions and the notation to identify the corresponding results from that work.

\subsection{Definition of bipartite multigraphs}

\begin{defn}[Bipartite graph]
	We say a graph $\gamma$ is a bipartite graph if the vertices can be partitioned into $2$ classes $U,V$ such that every edge has its ends in different classes.
\end{defn}

\begin{defn}[Bipartite $l$-multigraph]
	We say that a bipartite multigraph $\gamma$ with vertices $(U,V)$ where $U:=\{1,2,\ldots,M\}$ and $V:=\{1,2,\ldots,N\}$ is an bipartite $l$-multigraph if any pair of vertices can be connected by at most $l$ edges.
	The multiplicity $m_{\hidx\widx}^{\gamma}$ between the vertices $\hidx\in U$ and $\widx\in V$ is the number of separate edges connecting them.
	We let $m_{\hidx\widx}^{\gamma} = 0$ if the vertices $\hidx$ and $\widx$ are not connected.
	If $m_{\hidx\widx}^{\gamma} = 1,2,3,4$, we say that the edge between the nodes $\hidx$ and $\widx$ is single, double, triple, quadruple, respectively.
\end{defn}
Note that by the definition of the bipartite graph, self edge is not allowed. 

The degree of the vertex $\hidx\in U$ and $\widx \in V$ is defined by
\[ d_\hidx^\gamma = \sum_{\widx=1}^{N}m_{\hidx\widx}^{\gamma},\quad d_\widx^\gamma = \sum_{\hidx=1}^{M}m_{\hidx\widx}^{\gamma},\]
respectively.
The size $m^\gamma$ and the length $n^\gamma$ is defined by
\[ m^\gamma=\sum_{\hidx,\widx}m_{\hidx\widx}^\gamma=\sum_{\hidx=1}^{M}d_\hidx^\gamma=\sum_{\widx=1}^{N}d_\widx^\gamma, \quad n^\gamma=\sum_{\hidx,\widx}\mathds{1}(m_{\hidx\widx}^\gamma > 0), \]
respectively.
The number of single (respectively, double, triple, quadruple) edges $n_1^\gamma$ (respectively, $n_2^\gamma, n_3^{\gamma}, n_4^{\gamma}$) is defined by
\[ n_r^{\gamma}=\sum_{\hidx,\widx}\mathds{1}(m_{\hidx\widx}^\gamma=r),\quad r=1,2,3,4.\]

The set of all bipartite $l$-multigraphs is denoted by  $\mathcal{G}^l$. (The empty graph is a bipartite multigraph and included in $\mathcal{G}^l$.)

We expand the product \eqref{eq:expansion} and rewrite it by using the bipartite $4$-multigraphs as
\begin{align*} 
	&\sum_{\gamma\in \mathcal{G}^4}\left(\prod_{\hidx,\widx}P_{\hidx\widx}^{(m_{\hidx\widx}^\gamma)}\left(\frac{\sqrt{N}u_\hidx v_\widx}{\|\U\|^{\epsilon_u}\|\V\|^{\epsilon_v}}\right)^{m_{\hidx\widx}^\gamma} \right) \\
	&= \sum_{\gamma\in \mathcal{G}^4}N^{m^{\gamma} / 2}\left( \prod_{\hidx,\widx}P_{\hidx\widx}^{(m_{\hidx\widx}^\gamma)} \right) \left( \prod_{\hidx=1}^{M}\left(\frac{u_\hidx}{\|\U\|^{\epsilon_u}}\right)^{d_\hidx^\gamma} \right)\left(\prod_{\widx=1}^{N}\left(\frac{v_\widx}{\|\V\|^{\epsilon_v}}\right)^{d_\widx^\gamma}\right), 
\end{align*}
where we define $P_{\hidx\widx}^{(0)}=1$.

Due to the symmetry of the prior, we find that if any of the vertices in $\gamma$ has an odd degree then 
\[ \mathbb{E}_{\mathcal{U}\times \mathcal{V}}\left[\left( \prod_{\hidx=1}^{M}\left(\frac{u_\hidx}{\|\U\|^{\epsilon_u}}\right)^{d_\hidx^\gamma} \right)\left(\prod_{\widx=1}^{N}\left(\frac{v_\widx}{\|\V\|^{\epsilon_v}}\right)^{d_\widx^\gamma}\right)\right]
=\mathbb{E}_\mathcal{U}\left[\prod_{\hidx=1}^{M}\left(\frac{u_\hidx}{\|\U\|^{\epsilon_u}}\right)^{d_\hidx^\gamma}\right]\mathbb{E}_\mathcal{V}\left[\prod_{\widx=1}^{N}\left(\frac{v_\widx}{\|\V\|^{\epsilon_v}}\right)^{d_\widx^\gamma}\right]=0. \]
Thus, we are led to consider bipartite $4$-multigraphs with even degrees for which we use the following definition:

\begin{defn}[Bipartite $l$-multicycle]
	We say that a bipartite multigraph $\gamma$ is a bipartite multicycle if the degree of any vertex of $\gamma$ is even.
	The set of all bipartite $l$-multicycles is denoted by $\mathcal{B}^l$.
\end{defn}
Note that every cycle in bipartite graph has an even length.

From now on, we will exclusively consider the bipartite multicycles. More precisely, we decompose the prior by letting
\begin{equation}
	\frac{u_\hidx}{\|\U\|^{\epsilon_u}}  = \frac{1}{\sqrt{M}}r^{(1)}_\hidx s_\hidx, \quad \frac{v_\widx}{\|\V\|^{\epsilon_v}} = \frac{1}{\sqrt{N}}r^{(2)}_\widx t_\widx,
\end{equation}
where $r^{(1)}$'s and $r^{(2)}$'s are independent Rademacher random variables and, $s_\hidx$'s and $t_\widx$'s are (respectively) non-negative identically distributed random variables.
Denote by $\mathcal{S}, \mathcal{T}$ the probability space for $\boldsymbol{s}=(s_1,\ldots,s_M)$ and $\boldsymbol{t}=(t_1,\ldots,t_N)$ respectively.
Since $r^{(1)}$'s and $r^{(2)}$'s are Rademacher random variables, by considering the expectation with respect to $r^{(1)}$ and $r^{(2)}$ first, and then the expectation with respect to $s_\hidx$'s and $t_\widx$'s, we are led to consider
\begin{equation}\label{eq:A-decomp-full}
	\sum_{\gamma\in \mathcal{B}^4}\frac{1}{M^{m^\gamma / 2}}\mathbb{E}_\mathcal{S}\left[ \prod_{\hidx=1}^M s_\hidx^{d_\hidx^\gamma} \right] \mathbb{E}_\mathcal{T}\left[\prod_{\widx=1}^N t_\widx^{d_\widx^\gamma}\right]\left( \prod_{\hidx,\widx}P_{\hidx\widx}^{(m_{\hidx\widx}^\gamma)} \right).
\end{equation}

\subsection{Notations and assumptions for the bipartite multigraph expansion}
Instead of analyzing the object \eqref{eq:A-decomp-full} directly, we use the 'conditioning', i.e., we assume a high probability event with respect to $\mathcal{U}\times \mathcal{V}$ in the analysis of \eqref{eq:A-decomp-full}.
Since the product term is nonnegative on $\Omega_{\mathcal{U}\times \mathcal{V}}^\epsilon\cap\Omega_{\mathcal{\Noise}}^\epsilon$ and
\[ \mathbb{E}\left[\left( \prod_{\hidx,\widx}P_{\hidx\widx}^{(m^\gamma_{\hidx\widx})} \right) \right]=0\]
for any $\gamma\in \mathcal{B}^4$ except when $\gamma$ is empty, by applying Lemma \ref{lem:rare-event-cut}, we can see that any event with probability $o(1)$ is negligible in the proof of Proposition \ref{prop:general}.

One of the most important features of the conditioning is that we need Condition \ref{cond:prior} to control large graphs.
Recall that in Condition \ref{cond:prior}, we assumed that
\[ \mathbb{E}_{\substack{\U,\U'\sim\mathcal{U},\\\V,\V'\sim\mathcal{V}}}
	\left[\mathds{1}(\Omega_{M,N}(\U,\V))\mathds{1}(\Omega_{M,N}(\U',\V'))
\exp\left(\sqrt{MN}\snr_c\langle\U,\U'\rangle \langle\V,\V'\rangle \right)\right]=O(1)\]
is bounded.
If we let $\snr = \snr_c\sqrt{M / N}$, the quantity above corresponds to the second moment of the likelihood ratio for i.i.d. prior on the event $\Omega_{M,N}$ with GOE disorder, i.e.,
\[ \mathbb{E}_{\mathcal{\Noise}}\left[\mathbb{E}_{\mathcal{U}\times\mathcal{V}}\left[\mathds{1}(\Omega_{M,N})\prod_{\hidx,\widx} \frac{p(\sqrt{N}\noise-\sqrt{\snr N}u_\hidx v_\widx)}{p(\sqrt{N}\noise)}\right]^2\right]\]
where $p$ is the density of the standard Gaussian distribution.

We will strengthen Condition \ref{cond:prior} as follows.
\begin{lem}\label{lem:prior-strong}
	Suppose that $\snr_c$ satisfies Condition \ref{cond:prior}.
	Then, there exists a sequence of events $\Omega'_{M,N}$ in $\mathcal{U}\times \mathcal{V}$ with $\mathbb{P}_{\mathcal{U}\times \mathcal{V}}((\Omega'_{M,N})^c)=o(1)$ such that for any $\snr <\snr_c\sqrt{\ratio}$,
	\begin{equation}\label{eq:prior-strong-iid}
		\begin{split}
			&\mathbb{E}_{\substack{\U,\U'\sim\mathcal{U},\\\V,\V'\sim\mathcal{V}}}
			\left[\mathds{1}(\Omega'_{M,N}(\U,\V))\mathds{1}(\Omega'_{M,N}(\U',\V'))
			\exp\left(\frac{N\snr \langle\U,\U'\rangle \langle\V,\V'\rangle}{\|\U\|^{\epsilon_u}\|\U'\|^{\epsilon_u}\|\V\|^{\epsilon_v}\|\V'\|^{\epsilon_v}} \right)\right] \\
			&=\left(1-\frac{\snr^2}{\ratio}\right)^{-1 / 2}+o(1),
		\end{split}
	\end{equation}
	for every $\epsilon_u, \epsilon_v \in \{0,1\}$.
\end{lem}
We will prove Lemma \ref{lem:prior-strong} in Section \ref{sec:technical}.

In the rest of the paper, we will assume that the following event holds:
\begin{defn}\label{defn:omega}
	We assume that $\Omega = \Omega'_{M,N}\cap\Omega_{\mathcal{U}\times \mathcal{V}}^\epsilon\cap\Omega''_{M,N}$, where $\Omega'_{M,N}$ is the event in Lemma \ref{lem:prior-strong}, $\Omega_{\mathcal{U}\times \mathcal{V}}^\epsilon$ in Proposition \ref{prop:general}, and
	\begin{align*} 
		\Omega''_{M,N} 
		&:= \left\{|\|\U\|-1|\le M^{-1 / 5}\right\}\cap \left\{\left|M\sum_{\hidx}u_\hidx^4 - m_4\right|\le M^{-1/5}\right\} \cap \left\{\left|M^3\sum_{\hidx}u_\hidx^8 - m_8\right|\le M^{-1 /5}\right\}\\
		&\quad\cap\left\{|\|\V\|-1|\le N^{-1 / 5}\right\}\cap \left\{\left|N\sum_{\widx}v_\widx^4 - n_4\right|\le N^{-1/5}\right\} \cap \left\{\left|N^3\sum_{\widx}v_\widx^8 - n_8\right|\le N^{-1 /5}\right\}.
	\end{align*}
	where we let
	\[ m_4:= M^2\mathbb{E}[u_1^4], \quad m_8:=M^4\mathbb{E}[u_1^8],\quad n_4:=N^2\mathbb{E}[v_1^4],\quad n_8:=N^4\mathbb{E}[v_1^8].\]
	Note that $\Omega$ holds with high probability with respect to $\mathcal{U}\times \mathcal{V}$ such that the event $\Omega$ is still coordinatewise symmetric.
\end{defn}

We now let
\begin{equation}
	A:=\sum_{\gamma\in \mathcal{B}^4}\frac{1}{M^{m^\gamma /2}}\mathbb{E}_{\mathcal{S}\times \mathcal{T}}\left[\mathds{1}(\Omega)\prod_{\hidx=1}^M s_\hidx^{d_\hidx^\gamma}\prod_{\widx=1}^N t_\widx^{d_\widx^\gamma}\right]\left( \prod_{\hidx,\widx}P_{\hidx\widx}^{(m_{\hidx\widx}^\gamma)} \right) 
\end{equation}
with the event $\Omega$ in Definition \ref{defn:omega}.
It is clear that in order to prove Proposition \ref{prop:general}, it suffices to show that $\log A \implies \mathcal{N}(-\rho, 2\rho)$.

Lastly, we introduce the following to ease the notation in the analysis of $A$.
\begin{defn}\label{defn:factors}
	For a bipartite multicycle $\gamma$, we define its prior factor $V(\gamma)$ by
	\[ V(\gamma):=\mathbb{E}_{\mathcal{S}\times \mathcal{T}}\left[\mathds{1}(\Omega)\prod_{\hidx=1}^M s_\hidx^{d_\hidx^\gamma}\prod_{\widx=1}^N t_\widx^{d_\widx^\gamma}\right] ,\]
	and the random variable associated with the graph $\gamma$ 
	\[  W(\gamma):= \prod_{\hidx,\widx}P_{\hidx\widx}^{(m_{\hidx\widx}^{\gamma})} .\]
\end{defn}

With Definition \ref{defn:factors}, we can find
\begin{equation}\label{eq:A-decomp}
	A = \sum_{\gamma\in \mathcal{B}^4}\frac{1}{M^{m^\gamma / 2}}V(\gamma)W(\gamma).
\end{equation}
We remark that if $\gamma$ and $\gamma'$ do not intersect, i.e., $m_{\hidx\widx}^{\gamma}m_{\hidx\widx}^{\gamma'}=0$ for any $\hidx,\widx$, then $W(\gamma \cup \gamma')=W(\gamma)W(\gamma')$.

\section{Proof of Proposition \ref{prop:general}}\label{sec:prop-proof}

We prove Proposition \ref{prop:general} by following the sequence of multigraph estimates in Section 5 of \cite{Choo-Han-Lee2026}.
Unlike the Wigner setting, no separate self-loop contribution occurs here.

\subsection{Bipartite 1-multicycle universality}

The following estimate and monotonicity statements are the bipartite counterparts of Lemmas 5.1 and 5.2 in \cite{Choo-Han-Lee2026}.
The orthogonality argument is unchanged.

\begin{lem}\label{lem:1-cycle}
	Define
	\begin{equation}\label{eq:1-cycle}
		A_1 := \sum_{\gamma\in \mathcal{B}^1}\frac{1}{M^{m^\gamma / 2}}V(\gamma)W(\gamma)
	\end{equation}
	and
	\begin{equation}\label{eq:1-cycle-rademacher}
		R_1 := \sum_{\gamma\in \mathcal{B}^1}\frac{1}{M^{m^\gamma / 2}}W(\gamma).
	\end{equation}
	Then, $\mathbb{E}[(A_1-R_1)^2]=o(1)$.
\end{lem}
We first introduce a result on the monotonicity of the second moment of $A_1$ with respect to $\tau_1$, which will be used later.

\begin{lem}\label{lem:1-cycle-monotonicity}
	The variance $\mathbb{E}[A_1^2]$ depends only on $\mathcal{S},\mathcal{T}$ and $\tau'_1$.
	If we consider $\mathbb{E}[A_1^2]$ as a function of $\tau'_1$, $f(\tau'_1)=\mathbb{E}[A_1^2]$, then for any $0<\alpha_1 \le \alpha_2$, we have $f(\alpha_1)\le f(\alpha_2)$.
\end{lem}
\begin{proof}
	We use the same orthogonality as in Lemma 5.2 of \cite{Choo-Han-Lee2026}, with the following notation.

	\begin{defn}[Orthogonal bipartite multicycles]\label{defn:orthogonality}
		We say that two bipartite multicycles $\gamma$ and $\gamma'$ are orthogonal if 
		\begin{equation}\label{eq:orthogonality}
			\mathbb{E}[W(\gamma)W(\gamma')]=\mathbb{E}\left[\prod_{\hidx,\widx}P_{\hidx\widx}^{(m_{\hidx\widx}^\gamma)} \prod_{\hidx',\widx'}P_{\hidx'\widx'}^{(m_{\hidx'\widx'}^{\gamma'})}\right]=0
		\end{equation}
	\end{defn}
	Note that the orthogonality depends not only on the bipartite multicycles $\gamma$ and $\gamma'$ but also $P_{\hidx\widx}^{(n)}$.
	Distinct graphs in $\mathcal{B}^1$ have an edge present in only one of them.
	Independence across edges and $\mathbb{E}[P_{\hidx\widx}^{(1)}]=0$ therefore imply their orthogonality.
	Consequently, we find that
	\[ \mathbb{E}[A_1^2]=\sum_{\gamma\in \mathcal{B}^1}\frac{1}{M^{m^\gamma}}V(\gamma)^2\mathbb{E}[W(\gamma)^2]= \sum_{\gamma\in \mathcal{B}^1}\left(\frac{\tau'_1}{M}\right)^{m^\gamma}V(\gamma)^2,\]
	which shows that $\mathbb{E}[A_1^2]$ depends only on $\mathcal{S},\mathcal{T}$, and $\tau_1'$, and also $\mathbb{E}[A_1^2]$ is an increasing function of $\tau_1'$.
\end{proof}
We now return to the proof of Lemma \ref{lem:1-cycle}.
\begin{proof}[Proof of Lemma \ref{lem:1-cycle}]
	Adapting the proof of Lemma \ref{lem:1-cycle-monotonicity}, we find
	\[ \mathbb{E}[(A_1-R_1)^2]=\mathbb{E}\left[\left( \sum_{\gamma\in \mathcal{B}^1}\frac{1}{M^{m^\gamma / 2}}(V(\gamma)-1)W(\gamma) \right)^2\right]
	=\sum_{\gamma\in \mathcal{B}^1}\left( \frac{\tau'_1}{M} \right)^{m^\gamma}(V(\gamma)-1)^2.\]
	We further decompose it as
	\begin{equation}\label{eq:1-cycle-decompose}
		\begin{split}
			\mathbb{E}[(A_1-R_1)^2]
			&= \sum_{\gamma\in \mathcal{B}^1, V(\gamma)\le 1}\left( \frac{\tau'_1}{M} \right)^{m^\gamma}(V(\gamma)-1)^2
			+\sum_{\gamma\in \mathcal{B}^1, V(\gamma)> 1}\left( \frac{\tau'_1}{M} \right)^{m^\gamma}(V(\gamma)-1)^2\\
			&\le  \sum_{\gamma\in \mathcal{B}^1, V(\gamma)\le 1}\left( \frac{\tau'_1}{M} \right)^{m^\gamma}(1-V(\gamma)^2)
			+\sum_{\gamma\in \mathcal{B}^1, V(\gamma)> 1}\left( \frac{\tau'_1}{M} \right)^{m^\gamma}(V(\gamma)^2-1)
		\end{split}
	\end{equation}
	To estimate the first term in the right side of \eqref{eq:1-cycle-decompose}, we consider the cutoff property for Rademacher prior, which is a bipartite version of Lemma 5.4 in \cite{Choo-Han-Lee2026}.
	\begin{lem}[Bipartite 1-multicycle cutoff property for Rademacher prior]\label{lem:1-cycle-cutoff-rademacher}
		If $\alpha <\sqrt{\ratio}$, then for any $\eta > 0$, there exists a constant $s_1$, independent of any other quantities, such that for any $s>s_1$
		\[ \sum_{\gamma\in \mathcal{B}^1, n^\gamma \ge s}\left( \frac{\alpha}{M} \right)^{m^\gamma} \le s\left(\frac{\alpha}{\sqrt{\ratio}}+\eta\right)^s, \]
		for sufficiently large $M,N$.
	\end{lem}
	We prove Lemma \ref{lem:1-cycle-cutoff-rademacher} in Section \ref{sec:technical}.
	Since $\tau'_1 = \tau_1 + o(1)$ and $\tau_1 < \snr_c\sqrt{\ratio} \le \sqrt{\ratio}$, for sufficiently large $M, N$, we have $\tau'_1 < \sqrt{\ratio}$.
	By Lemma \ref{lem:1-cycle-cutoff-rademacher}, taking $\eta>0$ small enough that $\eta + \tau_1'/\sqrt{\ratio} <1$, we notice that for any fixed $\epsilon$, there exists ($M,N$-independent) constant $K_\epsilon \equiv K_\epsilon(\tau_1)$ such that for any $s\ge K_\epsilon$ and for any sufficiently large $M$, 
	\begin{equation}\label{eq:1-cycle-cutoff-rademacher}
		\sum_{\gamma\in \mathcal{B}^1, n^\gamma\ge s}\left( \frac{\tau'_1}{M} \right)^{m^\gamma} <\epsilon
	\end{equation}
	Since the summand in the left side of \eqref{eq:1-cycle-cutoff-rademacher} is positive,
	\[
		\sum_{\gamma\in \mathcal{B}^1, V(\gamma)\le 1, n^\gamma\ge K_\epsilon}\left( \frac{\tau'_1}{M} \right)^{m^\gamma}V(\gamma)^2
		\le  \sum_{\gamma\in \mathcal{B}^1, V(\gamma)\le 1, n^\gamma\ge K_\epsilon}\left( \frac{\tau'_1}{M} \right)^{m^\gamma} 
		\le  \sum_{\gamma\in \mathcal{B}^1, n^\gamma\ge K_\epsilon}\left( \frac{\tau'_1}{M} \right)^{m^\gamma}<\epsilon,
	\]
	and in particular,
	\begin{equation}\label{eq:1-cycle-small-cutoff}
		\sum_{\gamma\in \mathcal{B}^1, V(\gamma)\le 1, n^\gamma \ge K_\epsilon} \left( \frac{\tau'_1}{M} \right)^{n^\gamma}(1-V(\gamma)^2)<\epsilon.
	\end{equation}

	To control the sum in the case $V(\gamma)\le 1, n^\gamma < K_\epsilon$, we apply the following lemma which is a modified version of Lemma 5.5 in \cite{Choo-Han-Lee2026}:
	\begin{lem}\label{lem:prior-dichotomy}
		Assume that $\mathcal{U}$ or $\mathcal{V}$ is not a Rademacher prior.
		Fix an ($M, N$-independent) constant $s$.
		Then, for any bipartite multicycle $\gamma$ with $n^\gamma<s$,
		\begin{itemize}
			\item if all node of $\gamma$ which are not a Rademacher prior are with degree $2$, then $V(\gamma)\le 1$ and $1-V(\gamma)=o(1)$.
			\item if at least one node of $\gamma$ which is not a Rademacher prior is with degree $4$ or higher, then $V(\gamma)> 1$ for any sufficiently large $M,N$
		\end{itemize}
	\end{lem}
	Lemma \ref{lem:prior-dichotomy} is a direct consequence of Lemma 5.5 in \cite{Choo-Han-Lee2026}.
	We omit the detail.
	Main difference between the above lemma and Lemma 5.5 in \cite{Choo-Han-Lee2026} is that since we are dealing with bipartite graph, there is a possibility that one prior is Rademacher and the other is not.
	If it is the case, then the degree of nodes of Rademacher prior does not change $V(\gamma)$, so we can only infer the characteristic of non-Rademacher node from $V(\gamma)$.
	From Lemma \ref{lem:prior-dichotomy}, we find that
	\begin{equation}\label{eq:1-cycle-small-finite}
		\sum_{\gamma\in \mathcal{B}^1, V(\gamma)\le 1, n^\gamma < K_\epsilon}\left( \frac{\tau'_1}{M} \right)^{n^\gamma}(1-V(\gamma)^2)=o(1),
	\end{equation}
	since the number of bipartite $1$-multicycles with the length $n^\gamma$ whose all nodes are with degree $2$ is $O(N^{n^\gamma})$.
	If there exists a node with degree $4$ or more because of Rademacher prior, then the number of such bipartite $1$-multicycles with the length $n^\gamma$ is $O(N^{n^\gamma-1})$.
	Since we only consider graphs with finite lenght, we conclude that \eqref{eq:1-cycle-small-finite} holds. 

	To estimate the second term in the right side of \eqref{eq:1-cycle-decompose}, we apply the following lemma:
	\begin{lem}\label{lem:1-cycle-large}
		If $\tau_1<\snr_c\sqrt{\ratio}$,
		\begin{equation}\label{eq:1-cycle-large}
			\sum_{\gamma\in \mathcal{B}^1, V(\gamma)>1}\left( \frac{\tau'_1}{M} \right)^{n^\gamma}(V(\gamma)^2-1)=o(1).
		\end{equation}
	\end{lem}

	We prove Lemma \ref{lem:1-cycle-large} in Section \ref{sec:technical}.
	Combining \eqref{eq:1-cycle-small-cutoff}, \eqref{eq:1-cycle-small-finite} and \eqref{eq:1-cycle-large}, we conclude that Lemma \ref{lem:1-cycle} holds.

\end{proof}

\subsection{Bipartite 2-multicycle universality}

This step adapts \cite[Lemma 5.7]{Choo-Han-Lee2026}.
We give the double-edge counting and the prior-factor estimates explicitly, since they now involve $MN$ possible edges and the moments of both priors.
As the second step, we prove the following estimate on bipartite $2$-multicycles, which extend Lemma \ref{lem:1-cycle}.
\begin{lem}\label{lem:2-cycle}
	Recall the definition of $S$ and $T$ in Definition \ref{defn:prior}.
	Define
	\begin{equation}\label{eq:2-cycle}
		A_2 := \sum_{\gamma\in \mathcal{B}^2}\frac{1}{M^{m^\gamma / 2}}V(\gamma)W(\gamma)
	\end{equation}
	and
	\begin{equation}\label{eq:2-cycle-rademacher}
		R_2 := \sum_{\gamma\in \mathcal{B}^2}\frac{1}{M^{m^\gamma / 2}}W(\gamma)
	\end{equation}
	Then, $\mathbb{E}[(A_2-R_2)^2]=o(1)$.
\end{lem}
\begin{proof}
	Recall that $n_1^\gamma$ is the number of the single edges in $\gamma$ and $n_2^\gamma$ the number of the double edges, and in particular, $\mathbb{E}[W(\gamma)^2]=(\tau'_1)^{n_1^\gamma}(\tau'_2)^{n_2^\gamma}.$
	If $\gamma$ and $\gamma'$ in $\mathcal{B}^2$ are distinct, then they are orthogonal since we are assuming $\mathbb{E}[P_{\hidx\widx}^{(1)}P_{\hidx\widx}^{(2)}]=0$.
	Thus, from $m^\gamma = n_1^\gamma +2n_2^\gamma$,
	\begin{equation}\label{eq:2-cycle-decompose}
		\mathbb{E}[(A_2-R_2)^2]=\sum_{\gamma\in \mathcal{B}^2}\left( \frac{1}{M} \right)^{m^\gamma}(\tau'_1)^{n_1^\gamma}(\tau'_2)^{n_2^\gamma}(V(\gamma)-1)^2
		=\sum_{\gamma\in \mathcal{B}^2}\left(\frac{\tau'_1}{M} \right)^{n^\gamma_1}\left(\frac{\tau'_2}{M^2}\right)^{n_2^\gamma}(V(\gamma)-1)^2. 
	\end{equation}
	We first prove the cutoff property for the right side of \eqref{eq:2-cycle-decompose}, which asserts that in this sum the contribution from large graphs is negligible.
	To estimate the contribution from large graphs, we will use the fact that $(V(\gamma)-1)^2 \le V(\gamma)^2+1$.

	For $\gamma\in \mathcal{B}^2$, let $\gamma^+$ be the set of bipartite $2$-multicycles that can be obtained by adding a double edge to $\gamma$.
	Then, since connecting vertices $\hidx$ and $\widx$ by a double edge increases the degrees $d_\hidx^\gamma$ and $d_\widx^\gamma$ by $2$, respectively, and remain other degrees unchanged,
	\begin{align*}
		\sum_{\gamma'\in \gamma^+}\left( \frac{\tau'_1}{M} \right)^{n_1^{\gamma'}}\left( \frac{\tau'_2}{M^2}\right)^{n_2^{\gamma'}} 
		&V(\gamma')^2 \\
		\le \left( \frac{\tau'_1}{M} \right)^{n_1^{\gamma}}\left( \frac{\tau'_2}{M^2}\right)^{n_2^{\gamma}+1} 
		&\sum_{\hidx,\widx}
		\left( \mathbb{E}_{\mathcal{S}\times \mathcal{T}}\left[\mathds{1}(\Omega)s_\hidx^2t_\widx^2\left(\prod_{k=1}^{M}s_k^{d_k^\gamma}\right)\left(\prod_{\kappa=1}^{N}t_\kappa^{d_\kappa^\gamma}\right)\right] \right)^2 \\
		\le \left( \frac{\tau'_1}{M} \right)^{n_1^{\gamma}}\left( \frac{\tau'_2}{M^2}\right)^{n_2^{\gamma}+1} 
		&\sum_{\hidx,\widx}
		\left( \mathbb{E}_{\mathcal{S}\times \mathcal{T}}\left[\mathds{1}(\Omega)s_\hidx^4t_\widx^4\left(\prod_{k=1}^{M}s_k^{d_k^\gamma}\right)\left(\prod_{\kappa=1}^{N}t_\kappa^{d_\kappa^\gamma}\right) \right] \right)\\
		&\qquad\left( \mathbb{E}_{\mathcal{S}\times \mathcal{T}}\left[\mathds{1}(\Omega)\left(\prod_{k=1}^Ms_k^{d_k^\gamma}\right)\left(\prod_{\kappa=1}^{N}t_\kappa^{d_\kappa^\gamma}\right)\right] \right)\\
		= \left( \frac{\tau'_1}{M} \right)^{n_1^{\gamma}}\left( \frac{\tau'_2}{M^2}\right)^{n_2^{\gamma}+1} 
		&\left( \mathbb{E}_{\mathcal{S}\times \mathcal{T}}\left[\mathds{1}(\Omega)\left(\sum_\hidx s_\hidx^4\right)\left(\sum_\widx t_\widx^4\right)\left(\prod_{k=1}^{M}s_k^{d_k^\gamma}\right)\left(\prod_{\kappa=1}^{N}t_\kappa^{d_\kappa^\gamma}\right) \right] \right)\\
		&\left( \mathbb{E}_{\mathcal{S}\times \mathcal{T}}\left[\mathds{1}(\Omega)\left(\prod_{k=1}^Ms_k^{d_k^\gamma}\right)\left(\prod_{\kappa=1}^{N}t_\kappa^{d_\kappa^\gamma}\right)\right] \right)
	\end{align*}

	From the definition of $\Omega$, we have
	\begin{align*}
		&\mathbb{E}_{\mathcal{S}\times \mathcal{T}}\left[\mathds{1}(\Omega)\left( \sum_\hidx s_\hidx^4 \right)\left( \sum_\widx t_\widx^4 \right)\left(\prod_{k=1}^M s_k^{d_k^\gamma}\right)\left(\prod_{\kappa=1}^N t_\kappa^{d_\kappa^\gamma}\right)\right] \\
		&\le MNm_4n_4(1+M^{-1 / 5})(1+N^{-1 /5})\;\mathbb{E}_{\mathcal{S}\times \mathcal{T}}\left[\mathds{1}(\Omega)\left(\prod_{k=1}^M s_k^{d_k^\gamma}\right)\left(\prod_{\kappa=1}^N t_\kappa^{d_\kappa^\gamma}\right)\right].
	\end{align*}
	We now have
	\begin{equation}\label{eq:2-cycle-add-edge}
		\sum_{\gamma'\in \gamma^+}\left( \frac{\tau'_1}{M} \right)^{n_1^{\gamma'}}\left( \frac{\tau'_2}{M^2}\right)^{n_2^{\gamma'}} V(\gamma')^2 \le  \left( 2m_4n_4 \tau'_2 \frac{N}{M}\right) \left( \frac{\tau'_1}{M} \right)^{n_1^\gamma}\left( \frac{\tau'_2}{M^2} \right)^{n_2^\gamma} V(\gamma)^2
	\end{equation}

	Applying \eqref{eq:2-cycle-add-edge} inductively, if we let $\gamma^{+r}$ be the set of $2$-multicycles that can be obtained by adding $r$ double edges to $\gamma$, then
	\[ \sum_{\gamma'\in \gamma^{+r}}\left( \frac{\tau'_1}{M}\right)^{n_1^{\gamma'}}\left( \frac{\tau'_2}{M^2} \right)^{n_2^{\gamma'}}V(\gamma')^2 
	\le \frac{(2m_4n_4\tau'_2 N/ M)^r}{r!}\left(\frac{\tau'_1}{M} \right)^{n_1^\gamma} \left(\frac{\tau'_2}{M^2} \right)^{n^\gamma_2}V(\gamma)^2.\]
	where the factor $1 / (r!)$ is due to the fact that $r$ double edges added in the inductive step were counted $(r!)$-times.
	Now we take sufficiently large $M,N$ such that $\tau_2'N/ M < 2\tau_2/\ratio$, since $\tau_2' = \tau_2+o(1)$.
	Summing it over $r$, we get
	\begin{equation}\label{eq:2-cycle-1-edge-cutoff}
		\begin{split}
			\sum_{\gamma\in \mathcal{B}^2, n_1^\gamma\ge s}\left( \frac{\tau'_1}{M} \right)^{n_1^\gamma}\left( \frac{\tau'_2}{M^2} \right)^{n_2^\gamma}V(\gamma)^2
			&\le \sum_{\gamma'\in \mathcal{B}^1,n^{\gamma'}\ge s}\left( \frac{\tau'_1}{M} \right)^{n^{\gamma'}}V(\gamma')^2\sum_{r=0}^{\infty}\frac{(4m_4n_4\tau_2 / \ratio)^r}{r!} \\
			&= e^{4m_4n_4\tau_2 / \ratio}\sum_{\gamma'\in \mathcal{B}^1, n^{\gamma'}\ge s}\left( \frac{\tau'_1}{M} \right)^{n^{\gamma'}}V(\gamma')^2.
		\end{split}
	\end{equation}
	To estimate the right side of \eqref{eq:2-cycle-1-edge-cutoff}, we need a cutoff property that generalizes Lemma \ref{lem:1-cycle-cutoff-rademacher}.

	\begin{lem}\label{lem:1-cycle-cutoff}
		If $\tau'_1 < \snr_c \sqrt{\ratio}$, then for any fixed $\epsilon > 0$, there exists ($M,N$-independent) constant $K'_\epsilon \equiv K'_\epsilon(\tau'_1)$ such that for any $s\ge K'_\epsilon$ and for any sufficiently large $M,N$,
		\[ \sum_{\gamma\in \mathcal{B}^1, n^\gamma \ge s}\left( \frac{\tau'_1}{M} \right)^{n^\gamma}V(\gamma)^2 <\epsilon.\]
	\end{lem}

	We prove Lemma \ref{lem:1-cycle-cutoff} in Section \ref{sec:technical}.
	By Lemma \ref{lem:1-cycle-cutoff}, we conclude from \eqref{eq:2-cycle-1-edge-cutoff} that
	\[ \sum_{\gamma\in\mathcal{B}^2, n_1^\gamma\ge s}\left( \frac{\tau'_1}{M} \right)^{n_1^\gamma}\left( \frac{\tau'_2}{M^2}\right)^{n_2^\gamma} V(\gamma)^2< e^{4m_4n_4\tau_2 /\ratio}\epsilon,\]
	and by considering the Rademacher prior,
	\[ \sum_{\gamma\in\mathcal{B}^2, n_1^\gamma\ge s}\left( \frac{\tau'_1}{M} \right)^{n_1^\gamma}\left( \frac{\tau'_2}{M^2}\right)^{n_2^\gamma}< e^{4m_4n_4\tau_2 /\ratio}\epsilon,\]
	This shows that when we prove that the right side of \eqref{eq:2-cycle-decompose} is $o(1)$, we only need to consider bipartite $2$-multicycles $\gamma$ such that $n_1^\gamma <K'_\epsilon$.

	Similarly, we can find a constant $\tilde K_\epsilon$ such that for any $r_0 \ge \tilde K_\epsilon$,
	\begin{equation}\label{eq:2-cycle-2-edge-cutoff}
		\begin{split}
			&\sum_{\gamma\in \mathcal{B}^2, n_1^\gamma<K'_\epsilon, n_2^\gamma\ge r_0}\left( \frac{\tau'_1}{M} \right)^{n_1^\gamma}\left( \frac{\tau'_2}{M^2} \right)^{n_2^\gamma}V(\gamma)^2 \\
			&\le \left( \sum_{\gamma'\in \mathcal{B}^1, n^{\gamma'}<K'_\epsilon}\left( \frac{\tau'_1}{M} \right)^{n^{\gamma'}}V(\gamma')^2 \right)\left( \sum_{r=r_0}^{\infty}\frac{(4m_4n_4\tau_2 / \ratio)^r}{r!} \right) <\epsilon,
		\end{split}
	\end{equation}
	since in the right side of \eqref{eq:2-cycle-2-edge-cutoff}, the first sum is $O(1)$ as $K'_\epsilon$ is a constant, and the second sum is a tail part of the Taylor series of the exponential function.
	Also, by considering the Rademacher prior,
	\[ \sum_{\gamma\in \mathcal{B}^2, n_1^\gamma<K'_\epsilon, n_2^\gamma\ge r_0}\left( \frac{\tau'_1}{M} \right)^{n_1^\gamma}\left( \frac{\tau'_2}{M^2} \right)^{n_2^\gamma} < \epsilon.\]
	This shows that when we prove that the right side of \eqref{eq:2-cycle-decompose} is $o(1)$, we only need to consider bipartite $2$-multicycles $\gamma$ such that $n^\gamma_1 <K'_\epsilon$ and $n_2^\gamma<\tilde K_\epsilon$. 

	It remains to estimate
	\begin{equation}\label{eq:2-cycle-finite}
		\sum_{\gamma\in \mathcal{B}^2, n_1^\gamma<K'_\epsilon, n_2^\gamma<\tilde K_\epsilon} \left( \frac{\tau'_1}{M} \right)^{n_1^\gamma}\left( \frac{\tau'_2}{M^2} \right)^{n_2^\gamma}(V(\gamma)^2-1).
	\end{equation}
	Following the proof of Lemma \ref{lem:1-cycle}, for $\gamma\in \mathcal{B}^2$ with $n_1^\gamma<K'_\epsilon$ and $n_2^\gamma<\tilde K_\epsilon$, we separate the case $V(\gamma)\le 1$ and $V(\gamma)>1$ to find that the sum in \eqref{eq:2-cycle-finite} is $o(1)$.

	If both prior is not a Rademacher prior, we apply the first part of Lemma \ref{lem:prior-dichotomy} for the former.
	For the latter, we apply the second part of Lemma \ref{lem:prior-dichotomy} to find that $\gamma$ has a vertex with degree $4$ or higher.
	Since a bipartite $1$-multicycle with $s_1$ edges has at most $s_1$ vertices with degree $2$ or higher, each double edge connects $2$ vertices, and at least one vertex is with degree $4$ or higher, there are at most $N^{s_1+2s_2-1}$ different ways of choosing the vertices of $\gamma$.
	Since the number of ways of connecting the chosen vertices accordingly is bounded by a constant (depending only on $s_1$ and $s_2$), we find that
	\begin{equation}\label{eq:2-cycle-overlap}
		\sum_{\gamma\in \mathcal{B}^2, n_1^\gamma=s_1<K'_\epsilon, n_2^\gamma=s_2<\tilde K_\epsilon, V(\gamma)>1} \left(\frac{\tau'_1}{M} \right)^{n_1^\gamma}\left( \frac{\tau'_2}{M^2} \right)^{n_2^\gamma}(V(\gamma)^2-1) = O(N^{-1}) .
	\end{equation}
	Summing it over $s_1$ and $s_2$, we find that \eqref{eq:2-cycle-finite} is $o(1)$.

	If both prior is Rademacher prior, then clearly \eqref{eq:2-cycle-finite} is $0$.
	Lastly, if one only prior is Rademacher prior, we separate the case $V(\gamma)\le 1$ and $V(\gamma)>1$ and apply Lemma \ref{lem:prior-dichotomy} similarly.
	For the former, the number of graphs with $s_1$ single edges and $s_2$ double edges is $O(N^{s_1+2s_2})$.
	But since $1-V(\gamma)=o(1)$, $V(\gamma)\le 1$ part of \eqref{eq:2-cycle-overlap} is $o(1)$.
	For the latter, by the second part of Lemma \ref{lem:prior-dichotomy}, there exists at least one node with degree $4$ or higher.
	Following \eqref{eq:2-cycle-overlap}, we see that $V(\gamma)>1$ part of \eqref{eq:2-cycle-overlap} is also $o(1)$ and we finish the proof of the desired lemma.
\end{proof}

\subsection{Bipartite 3-multicycle estimate}
The next estimate adapts the similarity-class argument of Lemma 5.9 in \cite{Choo-Han-Lee2026}. 
\begin{lem}\label{lem:3-cycle}
	Recall the definition of $R_2$ in \eqref{eq:2-cycle-rademacher}.
	Define
	\begin{equation}\label{eq:3-cycle}
		A_3 := \sum_{\gamma\in \mathcal{B}^3} \frac{1}{M^{m^\gamma /2}}V(\gamma)W(\gamma).
	\end{equation}
	Then, $\mathbb{E}[(A_3-R_2)^2]=o(1)$.
\end{lem}
\begin{proof}[Proof of Lemma \ref{lem:3-cycle}]
	By Lemma \ref{lem:2-cycle} and
	\[ \mathbb{E}[(A_3-R_2)^2]\le 2\mathbb{E}[(A_3-A_2)^2]+2\mathbb{E}[(A_2-R_2)^2],\]
	it suffices to bound the second moment of
	\begin{equation}
		A_3-A_2=\sum_{\gamma\in \mathcal{B}^3\setminus\mathcal{B}^2}\frac{1}{M^{m^\gamma /2}}V(\gamma)W(\gamma).
	\end{equation}
	We use the equivalence relation of Definition 5.10 in \cite{Choo-Han-Lee2026}, restated in the bipartite notation.
	\begin{defn}[Similarity]\label{defn:similarity}
		We say that two bipartite multicycles $\gamma$ and $\gamma'$ are similar and write $\gamma\sim\gamma'$ if the following holds: for any nodes $\hidx$ and $\widx$, $m_{\hidx\widx}^\gamma>0$ iff $m_{\hidx\widx}^{\gamma'}>0$, and the difference of the multiplicities $|m_{\hidx\widx}^\gamma-m_{\hidx\widx}^{\gamma'}|$ is even.
	\end{defn}

	We notice that different similarity classes are orthogonal since an edge belonging to only one support contributes a mean $0$ factor, while an edge with different multiplicity parities contributes one of the vanishing mixed moments in Proposition \ref{prop:general}.
	Each class has a unique representative in $\mathcal{B}^2$, obtained by replacing every triple edge by a single edge.
	Hence
	\begin{equation}\label{eq:3-cycle-expansion}
		\mathbb{E}[(A_3-A_2)^2]=\sum_{\gamma\in \mathcal{B}^2}\mathbb{E}\left[\left( \sum_{\gamma'\in \mathcal{B}^3\setminus\mathcal{B}^2,\,\gamma'\sim\gamma}\frac{1}{M^{m^{\gamma'} /2}}V(\gamma')W(\gamma') \right)^2 \right].
	\end{equation}
	For $\gamma_1,\gamma_2\sim\gamma$, independence across edges gives
	\begin{equation}\label{eq:3-cycle-similar-expansion}
		\begin{split}
			&\mathbb{E}\left[\left(\frac{1}{M^{m^{\gamma_1}/2}}V(\gamma_1)W(\gamma_1)\right)\left(\frac{1}{M^{m^{\gamma_2}/2}}V(\gamma_2)W(\gamma_2)\right)\right] \\
			&=\frac{1}{M^{(m^{\gamma_1}+m^{\gamma_2})/2}}V(\gamma_1)V(\gamma_2)\prod_{\hidx,\widx}\mathbb{E}\left[P_{\hidx\widx}^{(m_{\hidx\widx}^{\gamma_1})}P_{\hidx\widx}^{(m_{\hidx\widx}^{\gamma_2})}\right].
		\end{split}
	\end{equation}
	Now applying the Schwarz inequality to the last factor of \eqref{eq:3-cycle-similar-expansion}, we get
	\begin{equation}\label{eq:3-cycle-similar-expansion-full}
		\begin{split}
			&\left|\mathbb{E}\left[\left(\frac{1}{M^{m^{\gamma_1}/2}}V(\gamma_1)W(\gamma_1)\right)\left(\frac{1}{M^{m^{\gamma_2}/2}}V(\gamma_2)W(\gamma_2)\right)\right]\right| \\
			&\le \frac{1}{M^{(m^{\gamma_1}+m^{\gamma_2})/2}}V(\gamma_1)V(\gamma_2)\prod_{\hidx,\widx}\left(\mathbb{E}\left[\left(P_{\hidx\widx}^{(m_{\hidx\widx}^{\gamma_1})}\right)^2\right]^{1/2}\mathbb{E}\left[\left(P_{\hidx\widx}^{(m_{\hidx\widx}^{\gamma_2})}\right)^2\right]^{1 /2}\right).
		\end{split}
	\end{equation}
	and \eqref{eq:3-cycle-expansion} yields
	\begin{equation}\label{eq:3-cycle-expansion-bound}
		\mathbb{E}[(A_3-A_2)^2]\le\sum_{\gamma\in\mathcal{B}^2}\left(\sum_{\gamma'\in\mathcal{B}^3\setminus\mathcal{B}^2,\,\gamma'\sim\gamma}\frac{1}{M^{m^{\gamma'}/2}}V(\gamma')\prod_{\hidx,\widx}\mathbb{E}\left[\left(P_{\hidx\widx}^{(m_{\hidx\widx}^{\gamma'})}\right)^2\right]^{1/2} \right)^2.
	\end{equation}

	The bipartite prior weight is controlled on $\Omega$ by $s_\hidx,t_\widx\in[0,K]$, where $K=2N^{1/60}$.
	Replacing $s$ single edges by triple edges raises the total degree by $2s$ in each vertex class.
	Consequently, for $\gamma'\sim\gamma$,
	\[ V(\gamma')\le K^{4n_3^{\gamma'}}V(\gamma),\qquad n_1^{\gamma'}=n_1^\gamma-n_3^{\gamma'},\qquad n_2^{\gamma'}=n_2^\gamma. \]
	Since the edge weight product $\mathbb{E}[(P_{\hidx\widx}^{(m_{\hidx\widx}^{\gamma'})})^2]^{1/2}$ equals $(\tau_1')^{n_1^{\gamma'}/2}(\tau_2')^{n_2^{\gamma'}/2}(\tau_3')^{n_3^{\gamma'}/2}$, we obtain
	\[ \frac{1}{M^{m^{\gamma'}/2}}V(\gamma')\prod_{\hidx,\widx}\mathbb{E}\left[\left(P_{\hidx\widx}^{(m_{\hidx\widx}^{\gamma'})}\right)^2\right]^{1/2} \le\left(\frac{\tau_1'}{M}\right)^{n_1^\gamma/2}\left(\frac{\tau_2'}{M^2}\right)^{n_2^\gamma/2}V(\gamma)\left(\frac{K^4\sqrt{\tau_3'}}{M\sqrt{\tau_1'}}\right)^{n_3^{\gamma'}}. \]
	There are $\binom{n_1^\gamma}{s}$ choices of the $s\ge1$ edges to replace. 
	Summing over these choices by the binomial theorem gives
	\begin{equation}
		\begin{split}
			&\left(\sum_{\gamma'\in\mathcal{B}^3\setminus\mathcal{B}^2,\,\gamma'\sim\gamma}\frac{1}{M^{m^{\gamma'}/2}}V(\gamma')\prod_{\hidx,\widx}\mathbb{E}\left[\left(P_{\hidx\widx}^{(m_{\hidx\widx}^{\gamma'})}\right)^2\right]^{1/2} \right)^2 \\
			&\le\left(\frac{\tau_1'}{M}\right)^{n_1^\gamma}\left(\frac{\tau_2'}{M^2}\right)^{n_2^\gamma}V(\gamma)^2\left[\left(1+\frac{K^4\sqrt{\tau_3'}}{M\sqrt{\tau_1'}}\right)^{n_1^\gamma}-1\right]^2.
		\end{split}
	\end{equation}
	Thus, the estimate reduces to a weighted sum over $2$-multicycles:
	\begin{equation}\label{eq:3-cycle-2-cycle-bound}
		\begin{split}
			\mathbb{E}[(A_3-A_2)^2]
			&\le\sum_{\gamma\in \mathcal{B}^2}\left(\frac{\tau_1'}{M} \right)^{n_1^\gamma}\left( \frac{\tau_2'}{M^2} \right)^{n_2^\gamma}V(\gamma)^2\left[\left( 1+\frac{K^4\sqrt{\tau_3'}}{M\sqrt{\tau_1'}} \right)^{n_1^\gamma}-1 \right]^2 \\
			&\le\sum_{\gamma\in \mathcal{B}^2}\left(\frac{\tau_1'}{M} \right)^{n_1^\gamma}\left( \frac{\tau_2'}{M^2} \right)^{n_2^\gamma}V(\gamma)^2\left[\left( 1+\frac{K^4\sqrt{\tau_3'}}{M\sqrt{\tau_1'}} \right)^{2n_1^\gamma}-1 \right].
		\end{split}
	\end{equation}
	To prove that the right side of the equation above is $o(1)$, we use the cutoff properties, Lemma \ref{lem:1-cycle-cutoff} and Equation \eqref{eq:2-cycle-2-edge-cutoff}.
	Recall that we defined $K'_\epsilon\equiv K'_\epsilon(\tau_1')$ in the proof of Lemma \ref{lem:1-cycle-cutoff}.
	Letting
	\[ \tau_1'' :=\tau_1'\left( 1+\frac{K^4\sqrt{\tau_3'}}{M\sqrt{\tau_1'}} \right)^2, \]
	we find
	\begin{align*}
		&\sum_{\gamma\in \mathcal{B}^2, n_1^\gamma\ge K'_\epsilon(\tau_1'')}\left(\frac{\tau_1'}{M} \right)^{n_1^\gamma}\left( \frac{\tau_2'}{M^2} \right)^{n_2^\gamma}V(\gamma)^2\left( 1+\frac{K^4\sqrt{\tau_3'}}{M\sqrt{\tau_1'}} \right)^{2n_1^\gamma} \\
		&<\sum_{\gamma\in \mathcal{B}^2, n_1^\gamma\ge K'_\epsilon(\tau_1'')}\left(\frac{\tau_1''}{M} \right)^{n_1^\gamma}\left(\frac{\tau_2'}{M^2} \right)^{n_2^\gamma}V(\gamma)^2 <\epsilon,
	\end{align*}
	since $\tau_1'<\tau_1''<\snr_c\sqrt{\ratio}$ for any sufficiently large $M, N$.
	Set $K''_\epsilon = K'_\epsilon(\tau_1'')$.
	Then,
	\begin{align*}
		&\sum_{\gamma\in \mathcal{B}^2}\left(\frac{\tau_1'}{M} \right)^{n_1^\gamma}\left( \frac{\tau_2'}{M^2} \right)^{n_2^\gamma}V(\gamma)^2\left[\left( 1+\frac{K^4\sqrt{\tau_3'}}{M\sqrt{\tau_1'}} \right)^{2n_1^\gamma}-1 \right] \\
		&\le \sum_{\gamma\in \mathcal{B}^2,n_1^\gamma< K''_\epsilon}\left(\frac{\tau_1'}{M} \right)^{n_1^\gamma}\left( \frac{\tau_2'}{M^2} \right)^{n_2^\gamma}V(\gamma)^2\left[\left( 1+\frac{K^4\sqrt{\tau_3'}}{M\sqrt{\tau_1'}} \right)^{2n_1^\gamma}-1 \right]+\epsilon \\
		&\le\frac{CK^4}{M} \sum_{\gamma\in \mathcal{B}^2,n_1^\gamma< K''_\epsilon}\left(\frac{\tau_1'}{M} \right)^{n_1^\gamma}\left( \frac{\tau_2'}{M^2} \right)^{n_2^\gamma}V(\gamma)^2+\epsilon \\
		&\le\frac{CK^4}{M} \sum_{\gamma\in \mathcal{B}^2,n_1^\gamma< K''_\epsilon,n_2^\gamma<\tilde{K}_\epsilon}\left(\frac{\tau_1'}{M} \right)^{n_1^\gamma}\left( \frac{\tau_2'}{M^2} \right)^{n_2^\gamma}V(\gamma)^2+2\epsilon 
	\end{align*}
	for some constant $C$.
	In the proof of Lemma \ref{lem:2-cycle}, we find that the right side in the equation above can be estimated as
	\[ \sum_{\gamma\in \mathcal{B}^2,n_1^\gamma< K''_\epsilon,n_2^\gamma<\tilde{K}_\epsilon}\left(\frac{\tau_1'}{M} \right)^{n_1^\gamma}\left( \frac{\tau_2'}{M^2} \right)^{n_2^\gamma}V(\gamma)^2 =\sum_{\gamma\in \mathcal{B}^2,n_1^\gamma< K''_\epsilon,n_2^\gamma<\tilde{K}_\epsilon}\left(\frac{\tau_1'}{M} \right)^{n_1^\gamma}\left( \frac{\tau_2'}{M^2} \right)^{n_2^\gamma}+o(1),\]
	and we find by naive power counting that the right side is $O(1)$.
	Thus, we find that the right side of \eqref{eq:3-cycle-2-cycle-bound} is bounded by $2\epsilon+O(N^{- 14 /15})$, and this completes the proof of the desired lemma.
\end{proof}

\subsection{Bipartite 4-multicycle estimate}

Recall that our main object $A$ is defined as the sum over $\gamma\in \mathcal{B}^4$, the set of all bipartite $4$-multicycles.
We will prove the following result, which adapts Lemma 5.11 in \cite{Choo-Han-Lee2026}.
\begin{lem}\label{lem:4-cycle}
	Recall the definition of $R_2$ in \eqref{eq:2-cycle-rademacher}.
	Then, $\mathbb{E}[(A-R_2)^2]=o(1)$.
\end{lem}
\begin{proof}[Proof of Lemma \ref{lem:4-cycle}]
	By Lemma \ref{lem:2-cycle} and
	\[ \mathbb{E}[(A-R_2)^2]\le 2\mathbb{E}[(A-A_2)^2]+2\mathbb{E}[(A_2-R_2)^2],\]
	it remains to estimate
	\begin{equation}\label{eq:4-cycle-2-cycle-diff}
		A-A_2=\sum_{\gamma\in \mathcal{B}^4\setminus\mathcal{B}^2} \frac{1}{M^{m^\gamma /2}}V(\gamma)W(\gamma).
	\end{equation}
	Using the similarity classes of Definition \ref{defn:similarity}, we now replace triple and quadruple edges by single and double edges, respectively.
	Applying the orthogonality argument and Schwarz inequality in \eqref{eq:3-cycle-similar-expansion-full}, we want to bound
	\[ \frac{1}{M^{m^{\gamma'}/2}}V(\gamma')\prod_{\hidx,\widx}\mathbb{E}\left[\left(P_{\hidx\widx}^{(m_{\hidx\widx}^{\gamma'})}\right)^2\right]^{1/2} \le  \left(\frac{\tau_1'}{M}\right)^{n_1^\gamma/2}\left(\frac{\tau_2'}{M^2}\right)^{n_2^\gamma/2}V(\gamma)\left(\frac{K^4\sqrt{\tau_3'}}{M\sqrt{\tau_1'}}\right)^{n_3^{\gamma'}}\left(\frac{K^4\sqrt{\tau_4'}}{M\sqrt{\tau_2'}}\right)^{n_4^{\gamma'}},\]
	for $\gamma'\in \mathcal{B}^4\setminus\mathcal{B}^2, \gamma'\sim\gamma$.
	Two applications of the binomial theorem therefore yield
	\begin{equation}
		\begin{split}
			&\mathbb{E}\left[\left(\sum_{\gamma'\in\mathcal{B}^4\setminus\mathcal{B}^2,\,\gamma'\sim\gamma}\frac{1}{M^{m^{\gamma'}/2}}V(\gamma')W(\gamma')\right)^2\right] \\
			&\le\left(\frac{\tau_1'}{M}\right)^{n_1^\gamma}\left(\frac{\tau_2'}{M^2}\right)^{n_2^\gamma}V(\gamma)^2\left[\left(1+\frac{K^4\sqrt{\tau_3'}}{M\sqrt{\tau_1'}}\right)^{n_1^\gamma}\left(1+\frac{K^4\sqrt{\tau_4'}}{M\sqrt{\tau_2'}}\right)^{n_2^\gamma}-1\right]^2.
		\end{split}
	\end{equation}
	Summing over $\gamma\in\mathcal{B}^2$ and applying the cutoff properties in Lemma \ref{lem:1-cycle-cutoff} and \eqref{eq:2-cycle-2-edge-cutoff}, as in the preceding proof, completes the argument.
\end{proof}

\subsection{Decomposition for the Rademacher case}

From Lemma \ref{lem:4-cycle}, $\mathbb{E}[(A-R_2)^2]=o(1)$, and thus we find that the limiting distribution of $A$ does not depend on the prior.
More precisely, if we let
\begin{equation}\label{eq:4-cycle-rademacher}
	R=\sum_{\gamma\in \mathcal{B}^4}\frac{1}{M^{m^\gamma /2}}W(\gamma),
\end{equation}
then
\[ \mathbb{E}[(A-R)^2]\le 2\mathbb{E}[(A-R_2)^2]+2\mathbb{E}[(R-R_2)^2]=o(1) .\]

We now use the simple-cycle product strategy of \cite{Aizenman-Lebowitz-Ruelle1987}, following the reduction in Section 5.6 of \cite{Choo-Han-Lee2026}.
Here, $A$ can be approximated by a product that contains only even simple cycles and double edges.
Our definition of the simple cycles is as follows.
\begin{defn}\label{defn:simple-cycle}
	A bipartite $1$-multicycle is a simple cycle if the degree of any vertex is at most $2$, and its nonisolated vertices form a connected graph.
	The set of all non-empty bipartite simple cycles is denoted by $\mathcal{B}_S$.
\end{defn}
Recall from \eqref{eq:1-cycle-rademacher}, 
\[ R_l := \sum_{\gamma\in \mathcal{B}^l}\frac{1}{M^{m^\gamma / 2}}W(\gamma) \]
for $l = 1,2$.
We have the following claim:
\begin{equation}\label{eq:rademacher-decomp}
	R_2 - \prod_{\gamma\in \mathcal{B}_S} \left( 1+\frac{W(\gamma)}{M^{m^\gamma / 2}} \right)\prod_{\hidx,\widx}\left( 1+\frac{P_{\hidx\widx}^{(2)}}{M} \right)= o_{\mathcal{\Noise}}(1).
\end{equation}
Note that for $\gamma\in \mathcal{B}_S, W(\gamma) = \prod_{\hidx, \widx: m_{\hidx\widx}^\gamma = 1}P_{\hidx\widx}^{(1)}$.
First, we show that
\begin{equation}\label{eq:rademacher-decomp-r1}
	R_2 = \left(\sum_{\gamma\in \mathcal{B}^1}\frac{1}{M^{m^\gamma / 2}}W(\gamma) \right)\prod_{\hidx,\widx}\left( 1+\frac{P_{\hidx\widx}^{(2)}}{M} \right)+o_\mathcal{\Noise}(1) 
	=R_1 \prod_{\hidx,\widx}\left( 1+\frac{P_{\hidx\widx}^{(2)}}{M} \right)+o_\mathcal{\Noise}(1).
\end{equation}
We can check that 
\[ \prod_{\hidx,\widx}\left( 1+\frac{P_{\hidx\widx}^{(2)}}{M} \right) =\sum_{\gamma'\in \mathcal{B}_D} \frac{1}{M^{m^\gamma / 2}} W(\gamma'),\]
where $\mathcal{B}_D$ denotes the set of all bipartite multicycles consisting of double edges only.
For $\gamma\in \mathcal{B}^1$ and $\gamma'\in\mathcal{B}_D$, we define the sum $\gamma+\gamma'$ to be the bipartite $3$-multicycle $\tilde{\gamma} \equiv \tilde{\gamma}(\gamma,\gamma')$ such that $m_{\hidx\widx}^{\tilde{\gamma}} = m_{\hidx\widx}^\gamma+m_{\hidx\widx}^{\gamma'}$.
Note that any bipartite $3$-multicycle can be uniquely decomposed into the sum of a bipartite $1$-multicycle and a bipartite multicycle in $\mathcal{B}_D$.
We also define
\[ \tilde{W}(\tilde{\gamma}(\gamma, \gamma')):= \prod_{\hidx,\widx}P_{\hidx\widx}^{(m_{\hidx\widx}^\gamma)}P_{\hidx\widx}^{(m_{\hidx\widx}^{\gamma'})} .\]
Then,
\[ \left(\sum_{\gamma\in \mathcal{B}^1}\frac{1}{M^{m^\gamma / 2}}W(\gamma) \right)\prod_{\hidx,\widx}\left( 1+\frac{P_{\hidx\widx}^{(2)}}{M} \right) = \sum_{\tilde{\gamma}\in \mathcal{B}^3}\frac{1}{M^{m^\gamma / 2}}\tilde{W}(\tilde{\gamma})=:R_3 .\]
Recall the definition of $A_3$ in \eqref{eq:3-cycle}.
If we consider the case where the prior is Rademacher and $P_{\hidx\widx}^{(3)}=P_{\hidx\widx}^{(1)}P_{\hidx\widx}^{(2)}$, then $A_3$ is equal to $R_3$.
(Note that we can check that the assumptions on $P_{\hidx\widx}^{(3)}$ holds in this case since $\mathbb{E}[P_{\hidx\widx}^{(3)}]=\mathbb{E}[P_{\hidx\widx}^{(1)}P_{\hidx\widx}^{(2)}]=0$ and $\mathbb{E}[P_{\hidx\widx}^{(2)}P_{\hidx\widx}^{(3)}]=\mathbb{E}[P_{\hidx\widx}^{(1)}( P_{\hidx\widx}^{(2)} )^2]=0$.)
Then, from the bipartite $3$-multicycle estimate, Lemma \ref{lem:3-cycle}, we find that \eqref{eq:rademacher-decomp-r1} holds.

It remains to prove the following simple cycle estimate:
\begin{lem}\label{lem:simple-cycle}
	\begin{equation}
		R_1 = \prod_{\gamma\in \mathcal{B}_S} \left( 1+\frac{W(\gamma)}{M^{m^\gamma / 2}} \right)+o_\mathcal{\Noise}(1) .
	\end{equation}
\end{lem}
We prove Lemma \ref{lem:simple-cycle} in Section \ref{sec:technical}.
The proof is based on the first part of the proof of Proposition 2.2 in \cite{Aizenman-Lebowitz-Ruelle1987}, but in a bipartite version.

With the claim \eqref{eq:rademacher-decomp}, we are now ready to finish the proof of Proposition \ref{prop:general}.
\begin{proof}[Proof of Proposition \ref{prop:general}]
	Taking the logarithm of the product in the left side of \eqref{eq:rademacher-decomp},
	\begin{align*}
		\log \prod_{\gamma\in\mathcal{B}_S}\left( 1+\frac{W(\gamma)}{M^{m^\gamma / 2}} \right)\prod_{\hidx,\widx}\left( 1+\frac{P_{\hidx\widx}^{(2)}}{M} \right) 
		=\sum_{\gamma\in \mathcal{B}_S}\log\left( 1+\frac{W(\gamma)}{M^{m^\gamma / 2}} \right) +\sum_{\hidx,\widx}\log\left( 1+\frac{P_{\hidx\widx}^{(2)}}{M} \right)=:L_1 + L_2
	\end{align*}
	To prove that $L_1$ converges to a Gaussian, we follow the analysis in the cluster expansion.
	Consider
	\[ L_1 = \sum_{\gamma\in \mathcal{B}_S}\log\left( 1+\frac{W(\gamma)}{M^{m^\gamma / 2}} \right)=\sum_{\gamma\in \mathcal{B}_S}\left( \frac{W(\gamma)}{M^{m^\gamma / 2}} - \frac{1}{2}\left( \frac{W(\gamma)}{M^{m^\gamma / 2}} \right)^2 \right)+o_\mathcal{\Noise}(1). \]
	Now $\mathbb{E}[W(\gamma)] = 0$ and distinct cycles are orthogonal, we get
	\begin{equation}\label{eq:rademacher-variance-eq} 
		\mathbb{E}\left[\left(\sum_{\gamma\in \mathcal{B}_S} \frac{W(\gamma)}{M^{m^\gamma / 2}}\right)^2\right]=\sum_{\gamma\in \mathcal{B}_S} \mathbb{E}\left[\frac{W(\gamma)^2}{M^{m^\gamma}}\right] =\sum_{s=2}^{\min\{M,N\}}\sum_{\gamma\in\mathcal{B}_S:|\gamma|=2s}\frac{{\tau'_1}^{2s}}{M^{2s}}.
	\end{equation}
	Since the number of distinct bipartite simple cycles with the length $2s$ is $M(M-1)\cdots(M-s+1)N(N-1)\cdots(N-s+1) / (2s)$, \eqref{eq:rademacher-variance-eq} leads to
	\[ \sum_{s=2}^{\min\{M,N\}}\left(\frac{N}{M}\right)^s\frac{{\tau'_1}^{2s}}{2s}+\left( \sum_{s=2}^{\min\{M,N\}}\left(\frac{N}{M}\right)^s\frac{{\tau'_1}^{2s}}{2s} \right)\times O(N^{-1}) .\]
	From $\tau_1 < \snr_c \sqrt{\ratio}$, we can take $\epsilon >0$ such that $(1 / \ratio+\epsilon)\tau_1^2 < \snr_c^2$.
	Since $\tau'_1 = \tau_1+o(1)$, for sufficiently large $M, N$, we have
	\[ \left( \frac{N}{M} \right){\tau'_1}^{2}<\left( \frac{1}{\ratio}+\epsilon \right)\tau_1^{2} < \snr_c^2 \le 1,\]
	since $M / N \to \ratio$ as $M,N\to \infty$.
	Therefore, the sum converges to
	\[ \sum_{s=2}^{\infty}\frac{1}{2s}\left( \frac{\tau_1^2}{\ratio} \right)^s = -\frac{1}{2}\left( \log\left( 1-\frac{\tau_1^2}{\ratio} \right)+\frac{\tau_1^2}{\ratio}\right) .\]
	From similar calculations, by Wick's identity for the moments of the Gaussian distribution, we find
	\[ L_1 \implies \mathcal{N}(-\rho_1, 2\rho_1), \]
	where
	\[ \rho_1 = \frac{1}{2}\sum_{s=2}^{\infty}\frac{1}{2s}\left(\frac{\tau_1^2}{d_0} \right)^s = -\frac{1}{4}\left( \log\left(1-\frac{\tau_1^2}{d_0}\right)+\frac{\tau_1^2}{d_0} \right)  .\]

	For $L_2$, since $P_{\hidx\widx}^{(2)}$'s are i.i.d., we can simply apply the central limit theorem and the law of large numbers to find that
	\[ L_2 = \sum_{\hidx,\widx}\log\left( 1+\frac{P_{\hidx\widx}^{(2)}}{M} \right) =\sum_{\hidx,\widx}\left(\frac{P_{\hidx\widx}^{(2)}}{M}-\frac{1}{2}\left(\frac{P_{\hidx\widx}^{(2)}}{M} \right)^2 \right)+o_\mathcal{\Noise}(1) \implies \mathcal{N}(-\rho_2, 2\rho_2), \]
	where $\rho_2 = \tau_2 / (2\ratio)$.

	Since $P_{\hidx\widx}^{(1)}$ and $P_{\hidx\widx}^{(2)}$ are orthogonal to each other, from \eqref{eq:rademacher-decomp}, we find that $\log A$ converges in distribution to Gaussian with the mean $(-\rho_1-\rho_2)$ and the variance $2(\rho_1+\rho_2)$. This completes the proof of Proposition \ref{prop:general}.

\end{proof}

\section{Technical details}\label{sec:technical}

This section provides the details used in Sections \ref{sec:multigraph} and \ref{sec:prop-proof}.
The proofs below distinguish the arguments inherited from \cite{Choo-Han-Lee2026} from the calculations needed for two overlaps, bipartite cycle counts, and the varying aspect ratio.

\begin{proof}[Proof of Lemma \ref{lem:prior-strong}]
	We adapt the idea of the proof of Lemma 4.3 in \cite{Choo-Han-Lee2026}.
	First, we consider the case when $\epsilon_u=\epsilon_v=0$.
	We notice that $\sqrt{M}\langle\U,\U'\rangle$ and $\sqrt{N}\langle\V,\V'\rangle$ are independent and converges in distribution to $\mathcal{N}(0, 1)$ by the central limit theorem.
	Thus, for the event $\Omega_{M,N}$ in Condition \ref{cond:prior}, by choosing $\Omega'_{M,N}=\Omega_{M,N}$,
	\[\mathds{1}(\Omega'_{M,N}(\U,\V))\mathds{1}(\Omega'_{M,N}(\U',\V'))
	\exp\left(N\snr \langle\U,\U'\rangle \langle\V,\V'\rangle \right)\implies\exp\left( \frac{\snr}{\sqrt{\ratio}}Z_1Z_2 \right)  .\]
	where $Z_1, Z_2$ are the standard Gaussian distribution.
	Further, if we let $\delta >0$ be a small constant satisfying $\snr(1+\delta)<\snr_c\sqrt{\ratio}$, then there exists a small constant $\epsilon> 0$ such that  $\snr(1+\delta)< (1-\epsilon)\snr_c \sqrt{\ratio}$.	
	Now if we take $M,N$ sufficiently large such that $(1-\epsilon)\sqrt{\ratio} < \sqrt{M / N}$, then
	\begin{align*}
		&\mathbb{E}_{\substack{\U,\U'\sim\mathcal{U},\\\V,\V'\sim\mathcal{V}}}
		\left[\left(\mathds{1}(\Omega'_{M,N}(\U,\V))\mathds{1}(\Omega'_{M,N}(\U',\V'))
		\exp\left(N\snr \langle\U,\U'\rangle \langle\V,\V'\rangle \right)\right)^{1+\delta}\right] \\
		&\le \mathbb{E}_{\substack{\U,\U'\sim\mathcal{U},\\\V,\V'\sim\mathcal{V}}}
		\left[\mathds{1}(\Omega'_{M,N}(\U,\V))\mathds{1}(\Omega'_{M,N}(\U',\V'))
		\exp\left( \sqrt{MN}\snr_c \langle\U,\U'\rangle \langle\V,\V'\rangle \right)\right]=O(1),
	\end{align*}
	by Condition \ref{cond:prior}.
	Thus, we have the convergence in $L^{1+\delta}$, which implies the uniform integrability and also, together with the convergence in distribution, implies the convergence in expectation.
	Therefore, by the moment generating function of the standard Gaussian distribution and the chi-squared distribution,
	\begin{align*}
		&\mathbb{E}_{\substack{\U,\U'\sim\mathcal{U},\\\V,\V'\sim\mathcal{V}}}
		\left[\mathds{1}(\Omega'_{M,N}(\U,\V))\mathds{1}(\Omega'_{M,N}(\U',\V'))
		\exp\left(N\snr \langle\U,\U'\rangle \langle\V,\V'\rangle \right)\right] \\
		& \to \mathbb{E}\left[\exp\left( \frac{\snr}{\sqrt{\ratio}}Z_1 Z_2 \right) \right]=\left.\mathbb{E}\left[\mathbb{E}\left[\exp\left(\frac{\snr}{\sqrt{\ratio}}Z_1 Z_2 \right)\right| Z_1\right]\right] \\
		& \quad = \mathbb{E}\left[\exp\left(\frac{\snr^2}{2\ratio}Z_2^2 \right) \right] = \left( 1-\frac{\snr^2}{\ratio} \right)^{-\frac{1}{2}}.
	\end{align*}

	To prove the case when $\epsilon_v=1$, we define $\Omega'_{M,N}(\U,\V)=\Omega_{M,N}(\U,\V)\cap\{|\sum_\widx v_\widx^2 - 1| \le N^{-1 / 3}\}$.
	Then, $\mathbb{P}_{\mathcal{U}\times \mathcal{V}}((\Omega'_{M,N})^c) = o(1)$.
	Since $\|\V\|$ converges to $1$ in probability, we again find that
	\[\mathds{1}(\Omega'_{M,N}(\U,\V))\mathds{1}(\Omega'_{M,N}(\U',\V'))
	\exp\left(\frac{N\snr \langle\U,\U'\rangle \langle\V,\V'\rangle}{\|\V\|\|\V'\|} \right)\implies\exp\left( \frac{\snr}{\sqrt{\ratio}}Z_1Z_2 \right)  .\]
	Let $\delta'>0$ be a small constant such that $\snr(1+\delta') / (\|\V\|\|\V'\|) < \snr_c \sqrt{\ratio}$ on the event $\{|\sum_\widx v_\widx^2 -1| \le N^{-1 /3}\}\cap\{|\sum_\widx (v'_\widx)^2 -1| \le N^{-1 /3}\}$, for any sufficiently large $M, N$.
	Then we can also take a small constant $\epsilon'>0$ such that $\snr(1+\delta')/ (\|\V\|\|\V'\|) < (1-\epsilon')\snr_c \sqrt{\ratio}$ and $(1-\epsilon')\sqrt{\ratio} < \sqrt{M / N}$ for sufficiently large $M,N$. 
	We then have the convergence in $L^{1+\delta'}$ as in the first case of the lemma.
	We now proceed as in the first part and concluded that the second case holds.
	Adding conditions on $\U$, we can prove the remaining part of the lemma similarly.
	We omit the details.
\end{proof}

\begin{proof}[Proof of Lemma \ref{lem:1-cycle-cutoff-rademacher}]
	We closely follow the proof of Lemma 5.4 in \cite{Choo-Han-Lee2026}.
	Suppose that $\gamma \in \mathcal{B}^1$.
	We can decompose it into the union $\gamma_1\cup\gamma_2\cup\cdots\cup\gamma_n$ where each $\gamma_i(i=1,2,\ldots,n)$ is a simple cycle defined in Definition \ref{defn:simple-cycle}.

	Since $n^\gamma=\sum_{i=1}^{n}n^{\gamma_i}$ with the decomposition, we have the inequality
	\begin{equation}\label{eq:simple-cycle-decomp}
		\sum_{\gamma\in \mathcal{B}^1}\left( \frac{\alpha}{M} \right)^{n^\gamma} \le \prod_{\gamma'\in \mathcal{B}_S}\left( 1+\left( \frac{\alpha}{M} \right)^{n^{\gamma'}} \right).
	\end{equation}
	(The inequality \eqref{eq:simple-cycle-decomp} is strict, since the right side contains terms of a form
	\[ \left( \frac{\alpha}{M} \right)^{n^{\gamma'_1}} \left( \frac{\alpha}{M} \right)^{n^{\gamma'_2}}\cdots \left( \frac{\alpha}{M} \right)^{n^{\gamma'_k}}\]
	where some of the $\gamma'_i$'s share a common edge.)

	For given $\eta >0$, since $\alpha^2 < \ratio$, we take $\epsilon > 0$ such that $(1 /\ratio +\epsilon)\alpha^2 < \min\{1,(\alpha /\sqrt{\ratio} + \eta)^2\}$.
	Then, we take $M, N$ sufficiently large so that $N /M < 1 / \ratio + \epsilon$.
	Now, consider a function
	\[ \Phi(k) = \begin{cases}
		e^{\delta s} & \text{if } k\ge s \\
		0 & \text{if } k<s
	\end{cases}, \]
	with
	\[ \delta = \frac{1}{2} \log\left[ \left(\frac{1}{\ratio}+\epsilon\right)^{-1} \alpha^{-2} \left( 1-\frac{\log s}{2s} \right)  \right], \]
	or equivalently $(1 /\ratio +\epsilon) e^{2\delta}\alpha^2 = 1- (\log s)/ (2s).$ 
	Then, since $\Phi(k) \le  e^{\delta k}$,
	\[ e^{\delta s} \sum_{\gamma\in \mathcal{B}^1, n^\gamma \ge s} \left( \frac{\alpha}{M} \right) ^{n^\gamma} = \sum_{\gamma\in \mathcal{B}^1} \Phi(n^\gamma)\left( \frac{\alpha}{M} \right)^{n^\gamma} \le \sum_{\gamma\in \mathcal{B}^1}\left( \frac{e^\delta \alpha}{M} \right)^{n^\gamma} .\]
	Thus, from \eqref{eq:simple-cycle-decomp},
	\[ e^{\delta s}\sum_{\gamma\in \mathcal{B}^1, n^\gamma \ge s}\left( \frac{\alpha}{M} \right)^{n^\gamma}\le \prod_{\gamma'\in \mathcal{B}_S}\left( 1+\left(\frac{e^\delta \alpha}{M} \right)^{n^{\gamma'}} \right)\le \exp\left( \sum_{\gamma'\in \mathcal{B}_S}\left( \frac{e^\delta \alpha}{M} \right)^{n^{\gamma'}} \right) .\]
	Since the number of distinct simple cycles with length $2k$ is $M(M-1)\ldots(M-k+1)N(N-1)\ldots(N-k+1) / (2k)$, for $k\ge 2$, we obtain
	\begin{align*}
		&\sum_{\gamma\in \mathcal{B}^1, n^\gamma \ge s}\left( \frac{\alpha}{M} \right)^{n^\gamma}
		\le e^{-\delta s}\exp\left( \sum_{k=2}^{\infty}\left( \frac{N}{M} \right)^k \frac{(e^\delta \alpha)^{2k}}{2k} \right) <e^{-\delta s} \exp\left( \sum_{k=2}^{\infty}\left( \frac{1}{\ratio}+\epsilon \right)^k\frac{(e^{\delta}\alpha)^{2k}}{2k} \right) \\
		&= \exp\left( -\delta s - \frac{1}{2}\log\left(1-\left(\frac{1}{\ratio}+\epsilon\right)e^{2\delta}\alpha^2 \right)- \frac{1}{2} \left( \frac{1}{\ratio}+\epsilon \right)e^{2\delta}\alpha^2  \right) \\
		&= \exp\left( s(\log\alpha)+\frac{s}{2}\log\left( \frac{1}{\ratio}+\epsilon \right) -s\log\left( 1-\frac{\log s}{2s} \right)-\frac{1}{2}\log\left( \frac{\log s}{2s} \right)-\frac{1}{2}\left( 1-\frac{\log s}{2s} \right) \right)  \\
		&= s\alpha^s\left(\frac{1}{\ratio}+\epsilon \right)^{s /2} \exp\left( -\frac{1}{2}\log\log s+O(1) \right).
	\end{align*}
	This proves the desired lemma.
\end{proof}

\begin{proof}[Proof of Lemma \ref{lem:1-cycle-large}]
	We follow the second-moment comparison of Lemma 5.6 in \cite{Choo-Han-Lee2026}, using the bipartite overlap estimate in Lemma \ref{lem:prior-strong}.
	Assume the GOE disorder for the log likelihood ratio in Theorem \ref{thm:lr-additive}, i.e.,
	\begin{equation}\label{eq:p-defn-for-goe-llr}
		p_{\hidx\widx}^{(n)}=\frac{(-1)^n \snr^{n / 2}p^{(n)}(\sqrt{N}\noise)}{n! p(\sqrt{N}\noise)}\;(n=1,2,3,4).
	\end{equation}
	where $\sqrt{N}\noise$'s are independent standard Gaussian random variables.
	Then, we get $\tau'_1=\tau_1 = \snr<\snr_c\sqrt{\ratio}$.

	Recall that the left side of \eqref{eq:prior-strong-iid} is the second moment of the likelihood ratio with GOE disorder, conditioned on $\Omega'_{M,N}$.
	Further, from the proof of Lemma \ref{lem:prior-strong}, it is not hard to see that the conclusion of Lemma \ref{lem:prior-strong} does not change even if $\Omega'_{M,N}$ is replaced by $\Omega$.
	Thus, from Lemma \ref{lem:prior-strong} and the proof of Theorem \ref{thm:lr-additive},
	\begin{equation}\label{eq:second-moment-graph-form}
		\mathbb{E}\left[\left( \sum_{\gamma\in \mathcal{B}^4}\frac{1}{M^{m^\gamma / 2}}V(\gamma)W(\gamma) \right)^2\right]=\left(1-\frac{\snr^2}{\ratio}\right)^{-1 / 2}+o(1).
	\end{equation}
	From the GOE assumption, $p_{\hidx\widx}^{(n)}$ is a constant multiple of the $n$-th Hermite polynomial with respect to the Gaussian of the form $e^{-x^2 / 2}$.
	Due to the orthogonality of the Hermite polynomials, we can easily see that if $\gamma$ and $\gamma'$ are distinct multicylces, then they are orthogonal.
	Thus, we find from \eqref{eq:second-moment-graph-form} that
	\[ \sum_{\gamma\in \mathcal{B}^4} \frac{1}{M^{m^\gamma}}V(\gamma)^2\mathbb{E}[W(\gamma)^2]=\left(1-\frac{\snr^2}{\ratio}\right)^{-1 / 2}+o(1).\]
	In particular, considering the Rademacher prior,
	\begin{equation}\label{eq:second-moment-graph-form-rademacher} 
		\sum_{\gamma\in \mathcal{B}^4}\frac{1}{M^{m^\gamma}}\mathbb{E}[W(\gamma)^2]=\left(1-\frac{\snr^2}{\ratio}\right)^{-1 / 2}+o(1).
	\end{equation}
	Comparing the equations above, we obtain
	\begin{equation}\label{eq:4-cycle-rademacher-difference}
		\sum_{\gamma\in \mathcal{B}^4}\frac{1}{M^{m^\gamma}}(V(\gamma)^2-1)\mathbb{E}[W(\gamma)^2]=o(1) .
	\end{equation}

	To proceed further, we introduce the bipartite $4$-multicycle cutoff property for Rademacher prior, which is the following lemma.
	\begin{lem}[$4$-multicycle cutoff property for Rademacher prior with GOE disorder]\label{lem:4-cycle-cutoff-rademacher}
		Assume the GOE disorder for the log likelihood ratio in Theorem \ref{thm:lr-additive}.
		If $\snr < \snr_c\sqrt{\ratio}$, then for any fixed $\epsilon>0$, there exists ($N$-independent) constant $K'''_\epsilon=K'''_\epsilon(\snr)$ such that for any $s >K'''_\epsilon$, and for any sufficiently large $N$,
		\[ \sum_{\gamma\in \mathcal{B}^4, n^\gamma\ge s}\frac{1}{M^{m^\gamma}}\mathbb{E}[W(\gamma)^2]<\epsilon .\]
	\end{lem}

	We will prove Lemma \ref{lem:4-cycle-cutoff-rademacher} later in this section.
	With Lemma \ref{lem:4-cycle-cutoff-rademacher}, we can estimate the sum in \eqref{eq:4-cycle-rademacher-difference} for the case $V(\gamma)\le 1$ and $n^\gamma\ge K'''_\epsilon$ by
	\begin{align*}
			&\sum_{\gamma\in \mathcal{B}^4, n^\gamma\ge K'''_\epsilon, V(\gamma)\le 1}\frac{1}{M^{m^\gamma}}(1-V(\gamma)^2)\mathbb{E}[W(\gamma)^2] \\
			&\le \sum_{\gamma\in \mathcal{B}^4, n^\gamma\ge K'''_\epsilon, V(\gamma)\le 1}\frac{1}{M^{m^\gamma}}\mathbb{E}[W(\gamma)^2] 
			\le \sum_{\gamma\in \mathcal{B}^4, n^\gamma \ge K'''_\epsilon} \frac{1}{M^{m^\gamma}}\mathbb{E}[W(\gamma)^2]<\epsilon.
	\end{align*}
	If $V(\gamma)\le 1$ and $n^\gamma<K'''_\epsilon$, by following the argument we used to derive \eqref{eq:1-cycle-small-finite}, we find that
	\[ \sum_{\gamma\in \mathcal{B}^4, n^\gamma\le K'''_\epsilon, V(\gamma)\le 1} \frac{1}{M^{m^\gamma}}(1-V(\gamma)^2)\mathbb{E}[W(\gamma)^2]<\epsilon.\]
	Thus, we find that
	\[ \sum_{\gamma\in \mathcal{B}^4, V(\gamma)\le 1}\frac{1}{M^{m^\gamma}}(1-V(\gamma)^2)\mathbb{E}[W(\gamma)^2]=o(1) .\]
	Plugging it into \eqref{eq:4-cycle-rademacher-difference}, we also find that
	\[ \sum_{\gamma\in \mathcal{B}^4, V(\gamma)>1}\frac{1}{M^{m^\gamma}}(V(\gamma)^2-1)\mathbb{E}[W(\gamma)^2]=o(1) .\]
	Since the summand is positive, by restricting $\gamma$ to be in $\mathcal{B}^1$, we find that the desired lemma holds.

\end{proof}

\begin{proof}[Proof of Lemma \ref{lem:1-cycle-cutoff}]
	This is the cutoff argument of Lemma 5.8 in \cite{Choo-Han-Lee2026}, applied to the bipartite estimates in Lemmas \ref{lem:1-cycle-cutoff-rademacher} and \ref{lem:1-cycle-large}.
	We again separate the sum into the case $V(\gamma)\le 1$ and the case $V(\gamma)>1$, which leads us to
	\begin{align*} 
		&\sum_{\gamma\in \mathcal{B}^1, n^\gamma\ge s}\left(\frac{\tau'_1}{M} \right)^{n^\gamma}V(\gamma)^2 \\
		&= \sum_{\gamma\in \mathcal{B}^1, n^\gamma\ge s, V(\gamma)\le 1}\left(\frac{\tau'_1}{M} \right)^{n^\gamma}V(\gamma)^2
		+ \sum_{\gamma\in \mathcal{B}^1, n^\gamma\ge s, V(\gamma)>1}\left(\frac{\tau'_1}{M} \right)^{n^\gamma}V(\gamma)^2.
	\end{align*}
	For the first term in the right side, from Lemma \ref{lem:1-cycle-cutoff-rademacher}, for given $\epsilon>0$, there exists $K_\epsilon$ such that
	\[ \sum_{\gamma\in \mathcal{B}^1, n^\gamma \ge K_\epsilon}\left( \frac{\tau'_1}{M} \right)^{n^\gamma}<\epsilon .\]
	Then if $s\ge K_\epsilon$,
	\begin{equation}\label{eq:2-cycle-small-cutoff}
		\sum_{\gamma\in\mathcal{B}^1,n^\gamma\ge s, V(\gamma)\le 1}\left(\frac{\tau'_1}{M} \right)^{n^\gamma}V(\gamma)^2
		\le \sum_{\gamma\in \mathcal{B}^1, n^\gamma\ge s, V(\gamma)\le 1}\left(\frac{\tau'_1}{M} \right)^{n^\gamma}<\epsilon.
	\end{equation}

	For the case $V(\gamma)>1$, from Lemma \ref{lem:1-cycle-large}, for given $\epsilon>0$, there exists $K'_\epsilon$ such that
	\[ \sum_{\gamma\in \mathcal{B}^1,n^\gamma\ge K'_\epsilon,V(\gamma)>1}\left(\frac{\tau'_1}{M} \right)^{n^\gamma}(V(\gamma)^2-1)<\epsilon .\]
	Now, choosing $s>\max\{K_\epsilon, K'_\epsilon\}$, we find that
	\begin{align*}
		&\sum_{\gamma\in \mathcal{B}^1, n^\gamma \ge s, V(\gamma)>1}\left(\frac{\tau'_1}{M} \right)^{n^\gamma}V(\gamma)^2 \\
		&\le \sum_{\gamma\in \mathcal{B}^1, n^\gamma \ge s, V(\gamma)>1}\left(\frac{\tau'_1}{M} \right)^{n^\gamma}(V(\gamma)^2-1)+\sum_{\gamma\in \mathcal{B}^1, n^\gamma \ge s, V(\gamma)>1}\left(\frac{\tau'_1}{M} \right)^{n^\gamma} < 2\epsilon.
	\end{align*}
	Combining with \eqref{eq:2-cycle-small-cutoff}, we prove the desired lemma.

\end{proof}

\begin{proof}[Proof of Lemma \ref{lem:simple-cycle}]
	Recall the definition of $R_1$,
	\[ R_1 = \sum_{\gamma\in \mathcal{B}^1}\frac{W(\gamma)}{M^{m^{\gamma}/2}}=\sum_{\gamma\in\mathcal{B}^1}e(\gamma) , \]
	where we let $e(\gamma)$ as a normalized edge weight of $\gamma$.
	Denote
	\[ S = \prod_{\gamma\in \mathcal{B}_S}\left( 1+\frac{W(\gamma)}{M^{m^\gamma /2}} \right)=\prod_{\gamma\in \mathcal{B}_S}(1+e(\gamma)) .\]
	We will approximate both $R_1$ and $S$ with graphs of finite length, uniformly in $M,N$, and prove that the approximated terms are close enough.
	Define the truncated $R_1$ as
	\[ R_1^K :=\sum_{\gamma\in \mathcal{B}^1,\, n^\gamma< K}e(\gamma).\]
	For a subset of simple cycles $\mathcal{F}\subset \mathcal{B}_S$, let
	\[ n^{\mathcal{F}} :=\sum_{\gamma\in \mathcal{F}}n^{\gamma},\quad e(\mathcal{F}):=\prod_{\gamma\in \mathcal{F}}e(\gamma) .\]
	Then, truncated $S$ can be defined as
	\[ S^K = \sum_{\mathcal{F}\subset\mathcal{B}_S,\, n^\mathcal{F}< K}e(\mathcal{F}).\]
	First, by the orthogonality of two distinct cycles, we have
	\[ \mathbb{E}[(R_1-R_1^K)^2]=\sum_{\gamma\in \mathcal{B}^1,\, n^\gamma \ge K}\mathbb{E}[e(\gamma)^2]=\sum_{\gamma\in \mathcal{B}^1,\,n^\gamma\ge K}\left(\frac{\tau_1'}{M}\right)^{n_1^\gamma }.\]
	By Lemma \ref{lem:1-cycle-cutoff-rademacher}, for any fixed $\epsilon$, there exists ($M,N$-independent) constant $K_\epsilon$ such that for any $K\ge K_\epsilon$, and for any sufficiently large $M$, $\mathbb{E}[(R_1-R_1^K)^2]<\epsilon$.

	Next, we claim that as $K\to \infty$, $S - S^K$ converges to $0$ in probability, uniformly in $M,N$.
	For any $z \in \mathbb{C}$, define
	\[ \varphi(z) = \prod_{\gamma\in \mathcal{B}_S}(1+z^{n^\gamma}e(\gamma)) .\]
	Then, $S = \varphi(1)$ and
	\[ S^K = \sum_{n=0}^{K-1} \frac{\varphi^{(n)}(0)}{n!}.\]
	By the remainder formula, we have
	\[ |S-S^K| =\left|\varphi(1)-\sum_{n=0}^{K-1}\frac{\varphi^{(n)}(0)}{n!}\right|\le \frac{1}{R^{K-1}(R-1)}\sup_{z\in \mathbb{C},|z|=R}|\varphi(z)| \]
	for any $R>1$. 
	Using $|1+\zeta|\le \exp(\operatorname{Re} \zeta+\frac{1}{2}|\zeta|^2)$, we get
	\begin{equation}\label{eq:remainder-bound}
		\begin{split}
			\sup_{z\in \mathbb{C},|z|=R}|\varphi(z)|
			&\le \exp\left(\sum_{s=2}^{\min\{M,N\}}\operatorname{Re}\left(z^{2s}\sum_{\gamma\in \mathcal{B}^s, |\gamma| = 2s}e(\gamma)\right)+\frac{1}{2}\sum_{\gamma\in \mathcal{B}_S}R^{2n^\gamma}e(\gamma)^2 \right) \\
			&\le \exp\left(\sum_{s=2}^{\min\{M,N\}}R^{2s}\left|\sum_{\gamma\in \mathcal{B}^s, |\gamma| = 2s}e(\gamma)\right|+\frac{1}{2}\sum_{\gamma\in \mathcal{B}_S}R^{2n^\gamma}e(\gamma)^2 \right) \\
		\end{split}
	\end{equation}
	To bound the first term, we see that
	\[ \mathbb{E}\left[\sum_{s=2}^{\min\{M,N\}} R^{2s}\left|\sum_{\gamma\in \mathcal{B}^s, |\gamma|=2s}e(\gamma)\right|\right]\le \sum_{s=2}^\infty R^{2s}\mathbb{E}\left[\left(\sum_{\gamma\in \mathcal{B}^s,|\gamma|=2s}e(\gamma)\right)^2\right]^{1 /2}  .\]
	From the orthogonality, if we take $R>1$ such that $R^2\tau_1<\snr_c\sqrt{\ratio}$, then as in \eqref{eq:rademacher-variance-eq}, there exists a ($M,N$-independent) constant $C'$ such that
	\begin{align*} 
		\sum_{s=2}^\infty R^{2s}\mathbb{E}\left[\left(\sum_{\gamma\in \mathcal{B}^s,|\gamma|=2s}e(\gamma)\right)^2\right]^{1 /2}
		&=\sum_{s=2}^{\min\{M,N\}} R^{2s}\left(\sum_{\gamma\in \mathcal{B}^s,|\gamma|=2s}\mathbb{E}\left[e(\gamma)^2\right]\right)^{1 /2} \\
		&\to \sum_{s=2}^\infty \frac{1}{\sqrt{2s}}\left( \frac{R^2 \tau_1}{\sqrt{\ratio}} \right)^s < C'.
	\end{align*}
	Similarly, there exists a ($M,N$-independent) constant $C''$ such that 
	\[ \mathbb{E}\left[\sum_{\gamma\in \mathcal{B}_S}R^{2n^\gamma}e(\gamma)^2\right]= \sum_{\gamma\in \mathcal{B}^S}\left( \frac{R^2\tau_1'}{M} \right)^{m^\gamma} \to \sum_{s=2}^\infty \frac{1}{2s}\left(\frac{R^4\tau_1^2}{\ratio} \right)^s < C'' .\]
	By Markov inequality, there exists a ($M,N$-independent) constant $C$ such that for any $\eta>1$,
	\[ \mathbb{P}\left(\sup_{z\in \mathbb{C}, |z|=R}|\varphi(z)|>\eta\right)\le \frac{C}{\log\eta}.	\]
	Therefore, for any given $\epsilon>0$, we have
	\[ \lim_{K\to \infty}\limsup_{M,N\to \infty}\mathbb{P}(|S-S^K|>\epsilon)=0 .\]

	Finally, let $\mathcal{B}^1_S\subset \mathcal{B}^1$ be the set of all bipartite $1$-multicycles where the degree of any vertex is $2$, i.e., union of vertex-disjoint bipartite $1$-multicycles.
	Then, define
	\[ G^K := \sum_{\gamma\in \mathcal{B}^1_S,\,n^\gamma <K }e(\gamma)  .\]
	For $\gamma \in \mathcal{B}^1 \setminus\mathcal{B}^1_S$, at least one vertex has degree of at least $4$, so $n^\gamma$, the number of edges of $\gamma$ is bigger than the number of vertices of $\gamma$.
	Since $n^\gamma <K$, we have finite number of graph isomorphisms, which leads to
	\[ \mathbb{E}[(R^K_1 -G^K)^2] =O(N^{-1}).\]

	Next, denote $\gamma^\mathcal{F}$ as a union multigraph of $\mathcal{F}\subset \mathcal{B}_S$.
	If every graphs in $\mathcal{F}$ are vertex disjoint, then the union graph $\gamma^\mathcal{F}$ is a bipartite $1$-multicycle with every vertex of degree $2$, i.e., $\gamma^{\mathcal{F}}\in \mathcal{B}^1_S$.
	Further, there is a unique correspondence between the subset $\mathcal{F}\subset \mathcal{B}_S$ and a union graph $\gamma^\mathcal{F}$ if $\gamma^\mathcal{F}\in \mathcal{B}^1_S$.
	We now have
	\[ S^K-G^K = \sum_{\substack{\mathcal{F}\subset\mathcal{B}_S, n^{\mathcal{F}}<K \\ \gamma^{\mathcal{F}}\not\in\mathcal{B}^1_S}}e(\mathcal{F}) .\]
	Since $\gamma^{\mathcal{F}}\not\in \mathcal{B}^1_S$, there exists at least one vertex with degree at least $4$, and $n^\mathcal{F}<K$ gives finite number of decomposing $\gamma^\mathcal{F}$ into $\mathcal{F}$, we conclude that
	\[ \mathbb{E}[(S^K-G^K)^2]=O(N^{-1}) .\]
\end{proof}

\begin{proof}[Proof of Lemma \ref{lem:4-cycle-cutoff-rademacher}]
	This Lemma is the bipartite counterpart of Lemma 6.1 in \cite{Choo-Han-Lee2026}.
	We first follow the proof of Lemma \ref{lem:4-cycle}.
	Note that in Lemma \ref{lem:4-cycle-cutoff-rademacher}, we are assuming the GOE disorder for log likelihood ratio
	From \eqref{eq:p-defn-for-goe-llr}, we have
	\[ \tau_n'=\frac{\snr^n}{(n!)^2}\mathbb{E}\left[\left( \frac{p^{(n)}(\sqrt{N}\noise)}{p(\sqrt{N}\noise)} \right)^2 \right]= \frac{\snr^n}{(n!)^2}\times n! = \frac{\snr^n}{n!}.\]
	Since
	\begin{equation}\label{eq:4-cycle-rademacher-decomp}
		\sum_{\gamma\in \mathcal{B}^4, n^\gamma \ge s}\frac{1}{M^{m^\gamma}}\mathbb{E}[W(\gamma)^2]\le \sum_{\gamma\in \mathcal{B}^2,n^\gamma \ge s} \frac{1}{M^{m^\gamma}}\mathbb{E}[W(\gamma)^2]+\sum_{\gamma\in \mathcal{B}^4\setminus\mathcal{B}^2} \frac{1}{M^{m^\gamma}}\mathbb{E}[W(\gamma)^2].
	\end{equation}
	For the second term, we group the bipartite multicycles in $\mathcal{B}^4\setminus\mathcal{B}^2$ according to their equivalence class defined through the similarity in Definition \ref{defn:similarity}.
	Then we get
	\begin{align*}
		&\sum_{\gamma\in \mathcal{B}^4\setminus\mathcal{B}^2} \frac{1}{M^{m^\gamma}}\mathbb{E}[W(\gamma)^2]
	=\sum_{\gamma\in \mathcal{B}^2}\mathbb{E}\left[\left( \sum_{\gamma'\in \mathcal{B}^4\setminus\mathcal{B}^2, \gamma'\sim\gamma} \frac{1}{M^{m^{\gamma'} /2}}W(\gamma') \right)^2 \right] \\
		&\le \sum_{\gamma\in \mathcal{B}^2}\left( \frac{\tau_1'}{M} \right)^{n_1^\gamma}\left( \frac{\tau_2'}{M^2} \right)^{n_2^\gamma}\left[\left( 1+\frac{K^4\sqrt{\tau_3'}}{M\sqrt{\tau_1'}} \right)^{n_1^\gamma}\left( 1+\frac{K^4\sqrt{\tau_4'}}{M\sqrt{\tau_2'}} \right)^{n_2^\gamma}-1\right]^2\\
		&\le \frac{CK^4}{M}\sum_{\gamma\in \mathcal{B}^2}\left( \frac{\tau_1'}{M} \right)^{n_1^\gamma}\left( \frac{\tau_2'}{M^2} \right)^{n_2^\gamma}=\frac{CK^4}{M}\sum_{\gamma\in \mathcal{B}^2}\left( \frac{\snr}{M} \right)^{n_1^\gamma}\left( \frac{\snr^2}{2M^2} \right)^{n_2^\gamma},
	\end{align*}
	for some constant $C$.
	Since the sum is bounded by $(1-\snr^2 /\ratio)^{-1 /2}+o(1)$ from \eqref{eq:second-moment-graph-form-rademacher}, we find that the second term of \eqref{eq:4-cycle-rademacher-decomp} is $o(1)$.

	Now for the first term, since $n^\gamma = n_1^\gamma +n_2^\gamma$ for $\gamma\in \mathcal{B}^2$, we have
	\begin{equation}\label{eq:2-cycle-rademacher-decomp}
		\sum_{\gamma\in \mathcal{B}^2, n^\gamma>2s}\frac{1}{M^{m^\gamma}}\mathbb{E}[W(\gamma)^2]\le \sum_{\gamma\in \mathcal{B}^2, n_1^\gamma>s} \frac{1}{M^{m^\gamma}}\mathbb{E}[W(\gamma)^2]+\sum_{\gamma\in \mathcal{B}^2, n_2^\gamma>s}\frac{1}{M^{m^\gamma}}\mathbb{E}[W(\gamma)^2].	
	\end{equation}
	Our goal is to show that the two sums in the right side of \eqref{eq:2-cycle-rademacher-decomp} are small.
	We have
	\[ \frac{1}{M^{m^\gamma}}\mathbb{E}[W(\gamma)^2]=\left(\frac{\snr}{M}\right)^{n_1^\gamma}\left(\frac{\snr^2}{2M^2} \right)^{n_2^\gamma} .\]
	To estimate the second term of \eqref{eq:2-cycle-rademacher-decomp}, we group the graphs in $\mathcal{B}^2$ according to their bipartite 1-submulticycles, i.e., bipartite $2$-multicycle but double edges removed.
	Then,
	\[ \sum_{\gamma\in \mathcal{B}^2, n_2^\gamma>s} \frac{1}{M^{m^\gamma}}\mathbb{E}[W(\gamma)^2]\le \left( \sum_{\gamma'\in \mathcal{B}^1}\frac{1}{M^{m^{\gamma'}}}\mathbb{E}[W(\gamma')^2] \right)\left( \sum_{r=s}^\infty \frac{M^r N^r}{r!}\left( \frac{\snr^2}{2M^2} \right)^r \right)  .\]
	The first term in the right side is still bounded by $(1-\snr^2 / \ratio)^{-1 /2}+o(1)$.
	Since the second part is a tail part of the Taylor series of the $e^{N\snr^2 / (2M)}$, which is smaller than $e^{\snr^2 /\ratio}$ for sufficiently large $M,N$, we can find an $M,N$-independent constant $s_1$ such that $s>s_1$, the left side is smaller than $\epsilon$. 

	Similarly, for the first term of \eqref{eq:2-cycle-rademacher-decomp}, since we group the graphs in $\mathcal{B}^2$ according to their double edges.
	Since we still have bipartite $1$-multicycle after removing all the double edges, we get
	\[ \sum_{\gamma\in \mathcal{B}^2, n_1^\gamma >s} \frac{1}{M^{m^\gamma}}\mathbb{E}[W(\gamma)^2]\le \left( \sum_{\gamma'\in \mathcal{B}^2,n_1^{\gamma'}=0}\frac{1}{M^{m^{\gamma'}}}\mathbb{E}[W(\gamma')^2] \right)\left( \sum_{\gamma''\in \mathcal{B}^1,n^{\gamma''}>s}\left( \frac{\snr}{M} \right)^{m^{\gamma''}} \right)   .\]
	By Lemma \ref{lem:1-cycle-cutoff-rademacher}, we can find an $M,N$-independent constant $s_2$ such that if $s>s_2$, the left side is smaller than $\epsilon$.
	Now we can finish the proof of the desired lemma by letting $K_\epsilon''' = 2\max\{s_1, s_2\}$.
\end{proof}

\subsection*{Acknowledgement}
Ji Oon Lee is grateful to the Department of Mathematics, Harvard University, for their kind hospitality during a visit in the acadmeic year 2026-2027, where part of this work was completed. This work was partially supported by National Research Foundation of Korea under grant number RS-2023-NR076695.

\bibliographystyle{abbrv}
\bibliography{spiked_rectangular}

\end{document}